\documentclass[11pt]{amsart}

\usepackage[english]{babel}
\usepackage[utf8]{inputenc}
\usepackage[T1]{fontenc}
\usepackage{float}
\usepackage{url}
\usepackage{mathrsfs}
\usepackage[toc,page]{appendix} 
\usepackage{latexsym}
\usepackage{array}
\usepackage{psfrag}
\usepackage{epsfig}
\usepackage{dsfont}
\usepackage{color}
\usepackage{soul}  
\usepackage{listings}
\usepackage{caption}
\usepackage{subcaption}
\usepackage{tikz}
\usetikzlibrary{spy}

\usepackage{fullpage}

\usepackage[backref=page]{hyperref}
\hypersetup{
	plainpages=false,
	unicode=false,          
	pdftoolbar=true,        
	pdfmenubar=true,        
	pdffitwindow=true,      
	pdftitle={My title},    
	pdfauthor={Author},     
	pdfsubject={Subject},   
	pdfcreator={Creator},   
	pdfproducer={Producer}, 
	pdfkeywords={keywords}, 
	pdfnewwindow=true,      
	colorlinks=true,       
	linkcolor=blue,          
	citecolor=blue,        
	filecolor=blue,      
	urlcolor=magenta,          
	urlbordercolor=0 1 1
}
\usepackage{amsmath,amsfonts,amssymb,amsthm}

\graphicspath{{./},{./fig/},{./expe_num/}}

\usepackage{palatino}

\DeclareFontFamily{OT1}{pzc}{}
\DeclareFontShape{OT1}{pzc}{m}{it}{<-> s * [1.100] pzcmi7t}{}
\DeclareMathAlphabet{\mathpzc}{OT1}{pzc}{m}{it}

\newtheorem{theorem}{Theorem}
\newtheorem{lemma}[theorem]{Lemma}
\newtheorem{corollary}[theorem]{Corollary}
\newtheorem{remark}[theorem]{Remark}
\newtheorem{comment}[theorem]{Comment}
\newtheorem{definition}[theorem]{Definition}
\newtheorem{proposition}[theorem]{Proposition}
\newtheorem{example}[theorem]{Example}

\newcommand{\be}{\begin{equation}}
\newcommand{\ee}{\end{equation}}
\newcommand{\beq}{\begin{eqnarray}}
\newcommand{\eeq}{\end{eqnarray}}
\newcommand{\beqs}{\begin{eqnarray*}}
\newcommand{\eeqs}{\end{eqnarray*}}
\newcommand{\bt}{\begin{theorem}}
\newcommand{\et}{\end{theorem}}
\newcommand{\bex}{\begin{example}}
\newcommand{\eex}{\end{example}}
\newcommand{\br}{\begin{remark}}
\newcommand{\er}{\end{remark}}
\newcommand{\bc}{\begin{corollary}}
\newcommand{\ec}{\end{corollary}}
\newcommand{\bl}{\begin{lemma}}
\newcommand{\el}{\end{lemma}}
\newcommand{\bp}{\begin{proposition}}
\newcommand{\ep}{\end{proposition}}
\newcommand{\bd}{\begin{definition}}
\newcommand{\ed}{\end{definition}}
\newcommand{\bas}{\begin{assumption}}
\newcommand{\eas}{\end{assumption}}

\newcommand{\R}{\mathbb{R}}
\newcommand{\N}{\mathbb{N}}

\def\0{{\bf 0}}

\def\dist{\operatorname{dist}}

\renewcommand{\div}{\operatorname{div}}

\newcommand{\MCH}{{\bf M-CH}}
\newcommand{\NMNCH}{{\bf NMN-CH}}
\newcommand{\MbSAV}{{\bf Mb-SAV}}

\title{
A mobility-SAV approach for a Cahn-Hilliard equation\\ with degenerate mobilities
}
\author{Elie Bretin}
\address{Univ Lyon, INSA de Lyon, CNRS UMR 5208, Institut Camille Jordan\\ 20 avenue Albert Einstein, F-69621 Villeurbanne, France\\ elie.bretin@insa-lyon.fr}
\author{Luca Calatroni}
\address{CNRS, Universit\'e C\^ote d'Azur, Inria, UMR 7271, Laboratoire I3S,\\ 2000 Route des Lucioles, 06903, Sophia-Antipolis, France\\ calatroni@i3s.unice.fr}
\author{Simon Masnou}
\address{Univ Lyon, Universit\'e Claude Bernard Lyon 1, CNRS UMR 5208, Institut Camille Jordan \\43 boulevard du 11 novembre
	1918, F-69622 Villeurbanne, France\\masnou@math.univ-lyon1.fr}

\makeatletter
\@namedef{subjclassname@2020}{\textup{2020} Mathematics Subject Classification}
\makeatother
\subjclass{74N20, 35A35, 53E10, 53E40, 65M32, 35A15}
\keywords{Phase field approximation, scalar auxiliary variable (SAV), Cahn--Hilliard equation, surface diffusion, degenerate mobilities, numerical approximation}

\date{\today}

\begin{document} 
\begin{abstract}
A novel numerical strategy is introduced for computing approximations of solutions to a Cahn-Hilliard model with degenerate mobilities. This model  has recently been introduced as a second-order phase-field approximation for surface diffusion flows. Its  numerical discretization  is challenging due to the degeneracy of the mobilities, which generally requires an implicit treatment to avoid stability issues at the price of increased complexity costs. To mitigate this drawback, we consider new first- and second-order Scalar Auxiliary Variable (SAV) schemes that, differently from existing approaches, focus on the  relaxation of the mobility, rather than the Cahn-Hilliard energy. These schemes are introduced and analysed theoretically in the general context of gradient flows and then specialised for the Cahn-Hilliard equation with mobilities. Various numerical experiments are conducted to highlight the advantages of these new schemes  in terms of accuracy, effectiveness and computational cost.
\end{abstract}
\maketitle

\section{Introduction}

The classical mathematical model for the surface diffusion flow of the boundary $\partial\Omega(t)$ of a domain $\Omega(t)\subset\R^N$ stipulates that the evolution of $\partial\Omega(t)$ is governed by the law
\[
 V(t) =  \Delta_{\partial \Omega(t)} H_{\partial \Omega(t)},\qquad \Omega(0)=\Omega_0,
 \]
where $V(t)$ is the normal velocity, $H_{\partial \Omega(t)}$ is the (scalar) mean curvature and $\Delta_{\partial \Omega(t)}$ the Laplace-Beltrami operator on $\partial \Omega(t)$. We use the orientation convention that the scalar mean curvature along the boundary of a convex domain is non negative and that $V(t)$ is positive if $\Omega(t)$ grows.

Surface diffusion flow can be interpreted as the $\dot{H}^{-1}$-gradient flow of the perimeter energy
\begin{equation}  \label{eq:perimeter}
P(\Omega(t)) = \int_{\partial \Omega(t)} d{\mathcal H}^{N-1},
\end{equation}
with ${\mathcal H}^{N-1}$ the $(N-1)$-dimensional Hausdorff measure.

Classical phase field models to approximate surface diffusion flow are based on the Cahn-Hilliard energy 
\begin{equation} \label{eq:CahnHi}
 P_{\varepsilon}(u) =   \int_{\R^N}\frac{\varepsilon}{2} | \nabla u |^2 + \frac{1}{\varepsilon} W(u)\  dx
 \end{equation}
that approximates the perimeter in the sense of $\Gamma$-convergence \cite{Modica1977,Modica1977a}. Here, $\varepsilon>0$ is a small parameter that quantifies the sharpness of the approximation and $W$ is a reaction potential, typically a smooth double-well function such as
\begin{equation} \label{eq:doubleW}
	W(s) = \frac{1}{2}s^2 (1-s)^2.
	\end{equation}
The Cahn-Hilliard functional \eqref{eq:CahnHi} involves two competing terms: as $\varepsilon$ goes to $0$, the reaction term $\frac{1}{\varepsilon} W(u)$ forces $u$ to take values in $\{0,1\}$ (if $W$ is chosen as in \eqref{eq:doubleW}) while the gradient term $\frac{\varepsilon}{2}| \nabla u |^2 $ promotes smoothness in the transition zone between $\{u=0\}$ and $\{u=1\}$. Therefore, any minimizer $u_\varepsilon$ of $P_{\varepsilon}$ is a smooth approximation of the characteristic function of a set $\Omega$ and $\{u =\frac 1 2\}$ approximates $\partial\Omega$.

Sharp surface diffusion flow is the $\dot{H}^{-1}$-gradient flow of the perimeter energy and  the Cahn-Hilliard energy is a good approximation of the perimeter in the sense of $\Gamma$-convergence, which has consequences for the convergence of global minimizers. It is therefore rather legitimate to take as a phase field model for surface diffusion flow the $\dot{H}^{-1}$-gradient flow of the Cahn-Hilliard energy, that is the Cahn-Hilliard equation.
\begin{equation}\label{eq:CH_eq}
\begin{cases}
    \partial_t u  &= \Delta \mu \\
    \mu &=  - \Delta u + \frac{1}{\varepsilon^2} W'(u)
   \end{cases}
\tag{\textbf{CH}}
\end{equation}
This equation has been widely used, since the seminal work of Cahn and Hilliard~\cite{cahnspin,CAHN:1958}, to model the dynamics of phase separation in many different contexts, see the examples and the references listed in the survey~\cite{novick2008cahn}, see also the recent book~\cite{bookMiranville} where state-of-the-art results and numerous applications are presented. Among them, we briefly mention recent studies on the use of such equation fo model tumour growth \cite{Colli2015} and transmission problems \cite{Colli2020}. 

%

However, although used for this purpose, it has been proved by different authors \cite{pego1989front,alikakos1994convergence} that the solutions to the Cahn-Hilliard equation do not converge, as $\varepsilon$ goes to $0$, to surface diffusion flows but rather to solutions of a non-local Hele-Shaw model. Fortunately, as we will see now, it is possible to modify the equation in order to recover in the limit the correct flow.

 \subsection{A Cahn-Hilliard model with a single degenerate mobility}
In \cite{cahn1996cahn}, Cahn et al extended \eqref{eq:CH_eq} to accommodate a concentration-dependent mobility that cancels out in pure state regions, i.e. prevents any motion there. Because it has this cancellation property, this mobility is often referred to in literature as {\it degenerate}. Cahn et al's model, that we call the \MCH{} model, can be written as follows:
\begin{equation}   \label{eq:MCH}
\left\lbrace
\begin{aligned}
& \partial_t u = \div\left(M(u)\nabla\mu\right),\\
& \mu = - \Delta u  + \frac{W'(u)}{\varepsilon^2}.
\end{aligned}
\right.  \tag{\textbf{M-CH}}
\end{equation} 
With suitable assumptions on $M$ and $W$, a formal convergence proof to the correct motion, i.e. the surface diffusion flow, is shown in  \cite{cahn1996cahn}. However, the particular model studied there involves a logarithmic potential $W$, which raises several issues from a numerical viewpoint. The double well potential  \eqref{eq:doubleW} is more convenient from a numerical perspective, we will consider it in this work.

The choice of a suitable mobility $M$ has been discussed theoretically in several works, see, e.g.,  \cite{gugenberger2008comparison,lee2015degenerate,MR3466205}. By formal asymptotic arguments it can be showed that the choice $M(s)=s(1-s)$ does not lead to the correct velocity as an additional bulk diffusion term appears in the dynamics. These conclusions are  corroborated numerically in \cite{dai2012motion,dai2014coarsening,MR3457961} where undesired coarsening effects are observed. Actually a quartic mobility such as $M(s)= 36s^2(1-s)^2$
is necessary to recover the correct velocity, see also the extension to the anisotropic case in~\cite{Dziwnik2019}. From now on, we will use $M(s)= 36s^2(1-s)^2$.

Although the \eqref{eq:MCH} model captures successfully the sharp interface limits for suitable choices of $M$ and $W$, and gives satisfactory numerical results, it does suffer from a well-known drawback. In the asymptotic regime, the leading error term of the model is of order $1$ with the solution $u_{\varepsilon}$ being of the form
$$u_\varepsilon(t) = q\left( \frac{\dist(x,\Omega_\varepsilon(t))}{\varepsilon} \right) + \mathcal{O}(\varepsilon), \quad\text{where } \Omega_\varepsilon(t) = \{ x,\, u_{\varepsilon}(x,t) \geq 1/2 \} $$ 
Here, $q$ is the optimal profile associated with the parameter free 1d Cahn-Hilliard energy and $\dist(\cdot,\Omega_\varepsilon(t))$ is the signed distance function to $\Omega_\varepsilon(t)$ (negative inside, positive outside). When $W$ is chosen as in \eqref{eq:doubleW}, $q$ satisfies the explicit formula
$$q(s) = \frac{1-\tanh(\frac{s}{2})}{2}.$$

The approximation by \eqref{eq:MCH}, being of first order only, presents challenges when approaching the pure states $0$ or $1$ due to the emergence of undesired oscillations and imprecise solution profiles. Furthermore, as demonstrated in \cite{bretin2020approximation,refId0}, this approximation leads to numerical volume losses, contradicting the inherent volume-preserving nature of the Cahn-Hilliard equation.

The authors of \cite{ratz2006surface} managed to improve the numerical accuracy by introducing  another degeneracy in the model. It has been successfully adapted in various applications,  see for example \cite{albani2016dynamics,naffouti2017complex,salvalaglio2017morphological,salvalaglio2015faceting}.  However, the proposed model does not derive from an energy, it is thus more difficult to prove rigorously theoretical properties and to extend to complex multiphase applications. A variational adaptation has been proposed in \cite{salvalaglio2019doubly} where the second degeneracy is injected in the energy. But because it relies on modifying the energy, the approach is hard to extend to complex multiphase or anisotropic  applications.

\subsection{A second-order Cahn-Hilliard model with two degenerate mobilities} \label{sec:doubly_mob}

A different approach is proposed in  \cite{bretin2020approximation,refId0} where another mobility $N$, in addition to $M$,
is incorporated in the metric used to define the gradient flow, thus possibly changing the overall geometry of the evolution problem.
The so-called \NMNCH{} model proposed in \cite{bretin2020approximation,refId0} reads
\begin{equation}  \label{eq:NMN_CH}
\begin{cases}
   \partial_t u &= N(u) \div\left(M(u) \nabla (N(u)\mu) \right), \\
 \mu &= - \Delta u + \frac{W'(u)}{\varepsilon^2}.\\
  \end{cases} 
\tag{\textbf{NMN-CH}}
\end{equation}
This equation, considered on a bounded open domain $Q\subset\R^N$ with smooth boundary, derives from the weighted $\dot{H}^{-1}$ gradient flow of the Cahn-Hilliard energy associated with the weighted $\dot{H}^1$ scalar product 
$$\langle f,g\rangle_{\dot{H}^1}=\int_QM\nabla(Nf)\cdot\nabla(Ng)\ dx$$
complemented with the introduction of a dependence on $u$ of $M$ and $N$.
Using formal asymptotic expansion, it is shown in~\cite{bretin2020approximation,refId0} that with
$$
M(s) =  s^2 (1-s)^2 + \varepsilon^2 \quad\text{ and }\quad  N(s) = \frac{1}{\sqrt{M(s)}},
$$
the error term of order $1$ in the solution cancels out, and the \eqref{eq:NMN_CH} model is of order~$2$.
The profile obtained for the solution $u_\varepsilon$ to \eqref{eq:NMN_CH} is therefore more accurate, it satisfies 
$$u_\varepsilon(t) = q\left( \frac{\dist(x,\Omega_\varepsilon(t))}{\varepsilon} \right) + \mathcal{O}(\varepsilon^2).$$ 
Volume conservation is thus ensured up to an error of order $2$, in comparison with the order $1$ observed for \eqref{eq:MCH}.

However, the numerical approximation of \eqref{eq:NMN_CH} raises a number of difficulties, particularly with regard to $N$ and $M$ mobilities, which can take on very different values: in regions corresponding to pure phases ($u=1$ or $u=0$), $M(u)$ is very close to $0$, while $N(u)\sim 1/\epsilon$. The convex-concave splitting scheme used in~\cite{bretin2020approximation} after \cite{MR1676409} overcomes these difficulties and gives good numerical results. But beyond the good empirical performance of the scheme, a rigorous proof of convergence and stability is missing.\\

The objective of this paper is to design a novel numerical strategy based on the 
{\it Scalar Auxiliary Variable} (SAV) method \cite{HOU2019307,SHEN2018407,doi:10.113717M1150153,doi:10.113719M1264412,wang2022application,huang2023structurepreserving}  for the effective treatment of the mobility terms. This approach offers the advantage of generating provable unconditionally stable schemes (i.e., stable for all time steps $\delta t$) with order one or two, see Section \ref{sec:SAV_review} for more details.

\paragraph{Structure of the paper.}
In Section \ref{sec:numerical_strat}, we review some numerical schemes classically considered to approximate numerically the solutions of the Cahn-Hilliard model.  We recall in particular both the popular convex-concave and the SAV schemes in the homogeneous case (that is, without mobility) but explain also the difficulties encountered when applying this strategy to solve models \eqref{eq:MCH} and \eqref{eq:NMN_CH}. 
In Section \ref{sec:grad_flow} we introduce and analyse two SAV schemes incorporating mobilities in the general context of gradient flows of a convex function with respect to a suitable energy . The numerical schemes obtained are simple and endowed with unconditional stability and accuracy properties, unlike standard convex-concave schemes used, e.g., in \cite{bretin2020approximation,refId0}, to solve models  \eqref{eq:MCH} and \eqref{eq:NMN_CH}. 
The final Section \ref{sec:applic_CH} focuses on the application of these schemes to Cahn-Hilliard models endowed in the inhomogeneous case. The effectiveness of the schemes is verified through various numerical experiments.

\section{Numerical strategies for Cahn-Hilliard flows}  \label{sec:numerical_strat}

We recall some classical numerical methods to discretize the Cahn-Hilliard model in the case of an homogeneous mobility ($M \equiv 1$) in \eqref{eq:CH_eq}.
First, we define:
\begin{equation}  \label{eq:P_epsilon}
 P_{\varepsilon}(u) =  \int_Q \frac{\varepsilon}{2} | \nabla u |^2 + \frac{1}{\varepsilon} W(u) dx.
\end{equation}
and assume that all equations are solved on a square-box $Q = [0,\ell_1]\times \cdots \times [0,\ell_d]$ with periodic boundary conditions. Thanks to the homogeneity of the operators involved fast semi-implicit Fourier spectral method in the spirit  of \cite{Chen_fourier,Bretin_brassel,bretin_droplet,bretin_largephases,BRETIN2018324} can be used, see also  \cite{DU2020425} for a recent review of numerical methods for the phase field approximation of various geometric flows. \\
We recall that the Fourier $\boldsymbol K$-approximation of a function $u$ defined in 
$Q $ is given by
$$u^{\boldsymbol K}(x) = \sum_{{\boldsymbol k}\in K_N  } c_{\boldsymbol k} e^{2i\pi{\boldsymbol \xi}_k\cdot x},$$
where  $K_N =  [ -\frac{N_1}{2},\frac{N_1}{2}-1 ]\times [ -\frac{N_2}{2},\frac{N_2}{2}-1] \cdots \times   [ -\frac{N_d}{2},\frac{N_d}{2}-1] $,   ${\boldsymbol k} = (k_1,\dots,k_d)$ and ${\boldsymbol \xi_k} = (k_1/\ell_1,\dots,k_d/\ell_d)$. In this formula, the coefficients $c_{\boldsymbol k}$'s denote the first discrete Fourier coefficients of $u$. 
The use of the inverse discrete Fourier transform entails that  for the values  at the points 
$x_{\boldsymbol k} = (k_1 h_1, \cdots, k_d h_d)$, $h_{\alpha} = L_{\alpha}/N_{\alpha}$ for $\alpha\in\{1,\cdots,d\}$, there holds $u^{\boldsymbol K}_{\boldsymbol k} =   \textrm{IFFT}[c_{\boldsymbol k}]$. Conversely,
$c_{\boldsymbol k}$ can be computed as the discrete Fourier transform of $u^K_{\boldsymbol k},$ by $c_{\boldsymbol k} =
\textrm{FFT}[u^K_{\boldsymbol k}].$

 Given a positive time discretization parameter $\delta_t$, we now revise some classical numerical schemes used for constructing sequences $(u^n)_{n \geq 0}$ approximating the exact solution $u$ at times $n\delta_t$ on $Q$.
 \subsection{Convex-concave splitting}
A popular approach proposed by Eyre in  \cite{MR1676409} uses a convex-concave splitting of the Cahn-Hilliard energy \eqref{eq:P_epsilon}. This approach provides a simple, efficient, and stable  scheme to approximate various evolution problems with a gradient flow structure, see for instance \cite{MR2418360,MR2519603,MR2799512,MR3100769,MR3564350,MR3682074,Schonlieb2011} and the recent second-order extensions proposed in \cite{MR3874087,salvalaglio2019doubly,Doubly_anisotropic}.

In Eyre's approach, $P_{\varepsilon}$ is decomposed as the sum of a convex energy $P_{\varepsilon,c}$ and a concave energy $P_{\varepsilon,e}$ (a classical explicit decomposition will be given below):
$$P_{\varepsilon}(u) = P_{\varepsilon,c}(u) + P_{\varepsilon,e}(u).$$
An effective numerical scheme can be designed by integrating implicitly the convex part, and explicitly the concave one, that is:
\begin{equation}   \label{eq:eyre}
\begin{cases}
 \frac{u^{n+1} - u^{n}}{\delta_t} &=  \Delta \mu^{n+1} \\
 \varepsilon \mu^{n+1} &=  \nabla_u  P_{\varepsilon,c}(u^{n+1}) + \nabla_u  P_{\varepsilon,e}(u^{n}).
\end{cases}
\end{equation}
Stability can be easily proved by interpreting this scheme as one step, starting from $u^n$, of the implicit discretization of the semi-linearized  PDE 
$$
\begin{cases}
 \partial_t u &=  \Delta \mu \\
 \varepsilon \mu &= \nabla_u \overline{P}_{\varepsilon,u^{n}}(u) 
\end{cases},
$$
 where $\overline{P}_{\varepsilon,u^{n}}$ is  defined by
 $$ 
\overline{P}_{\varepsilon,u^{n}}(u)  =P_{\varepsilon,c}(u) + P_{\varepsilon,e}(u^n) +
\langle \nabla_u  P_{\varepsilon,e}(u^n)  , (u - u^n) \rangle.
$$ 
The convexity of $\overline{P}_{\varepsilon,u^{n}}$ implies that
$$ \overline{P}_{\varepsilon,u^{n}}(u^{n+1}) +  \langle \nabla \overline{P}_{\varepsilon,u^{n}}(u^{n+1}), u^{n} - u^{n+1} \rangle  \leq   \overline{P}_{\varepsilon,u^{n}}(u^{n}).$$
Moreover, since
$$ \langle \nabla \overline{P}_{\varepsilon,u^{n}}(u^{n+1}), u^{n} - u^{n+1} \rangle = -  \delta_t \langle  \varepsilon \mu^{n+1}, \Delta \mu^{n+1} \rangle =  \delta_t \varepsilon \| \nabla \mu^{n+1} \|^2,$$
we deduce that
$$ \overline{P}_{\varepsilon,u^{n}}(u^{n+1}) \leq  \overline{P}_{\varepsilon,u^{n}}(u^n)   = P_{\varepsilon}(u^{n}).$$
The concavity of $P_{\varepsilon,e}$ implies that  $ P_{\varepsilon}(u) \leq \overline{P}_{\varepsilon,u^{n}}(u)$, and we conclude that 
$$ P_{\varepsilon}(u^{n+1}) \leq    P_{\varepsilon}(u^{n}),$$
therefore the above scheme \eqref{eq:eyre} is unconditionally stable, i.e., it decreases the energy without any particular assumption on the time step $\delta_t$.

When $W$ is a smooth double well potential 
as in \eqref{eq:doubleW}, a standard splitting choice is:
$$ P_{\varepsilon,c}(u) =  \frac{1}{2}\int_Q  \left( \varepsilon |\nabla u|^2  + \frac{\alpha}{\varepsilon} u^2\right)  dx,\quad P_{\varepsilon,e}(u) =  \int_{Q} \frac{1}{\varepsilon} \left( W(u) - \alpha \frac{u^2}{2} \right)dx.$$
with $\alpha \geq \max_{s \in [0,1]} \left| W''(s) \right|$ to ensure the concavity of $P_{\varepsilon,e}$. With this choice, \eqref{eq:eyre} specifies as: 
$$
\begin{cases}
 (u^{n+1} - u^{n})/\delta_t &=  \Delta \mu^{n+1} \\
 \mu^{n+1} &=  \left( - \Delta u^{n+1} + \frac{\alpha}{\epsilon^2} u^{n+1} \right) + \left( \frac{1}{\epsilon^2} (W'(u^{n}) - \alpha u^{n}) \right),   
\end{cases}
$$
which can be written in matrix form as:
$$ \begin{pmatrix}
    I_d &  - \delta_t \Delta \\
    \Delta - \alpha/\epsilon^2 & I_d
   \end{pmatrix}  
   \begin{pmatrix}
    u^{n+1} \\
    \mu^{n+1}
   \end{pmatrix}
   =
    \begin{pmatrix}
    u^{n} \\
     \frac{1}{\epsilon^2} (W'(u^{n}) - \alpha u^{n}).
   \end{pmatrix}
  $$
The pair $(u^{n+1},\mu^{n+1})$  can be thus expressed as
  $$u^{n+1} = L \left[ u^{n} + \frac{\delta_t}{\epsilon^2} \Delta \left(  W'(u^{n}) - \alpha u^{n} \right) \right],\quad \mu^{n+1}= L\left[ \frac{1}{\epsilon^2} (W'(u^{n} )  -  \Delta u^{n}) \right].$$
 with $L = \left( I_d +  \delta_t \Delta ( \Delta - \alpha/\epsilon^2 I_d ) \right)^{-1}$ having the Fourier symbol
 $$\hat{L}(\xi) = 1/(1 + \delta_t 4 \pi^2 |\xi|^2 ( 4 \pi^2 |\xi|^2 + \alpha/\epsilon^2)),$$
which makes it particularly well suited for computation in Fourier space.
  
An extension of this numerical scheme to approximate solutions to the Cahn-Hilliard model \eqref{eq:MCH} with degenerate mobility has been proposed in various contributions, see, e.g., \cite{guillengonzalez2023energystable}. However, the extension requires an implicit treatment of the mobility term, which prevents any explicit computation in Fourier space, which is why alternative methods are needed.

 \subsubsection{A convex-concave splitting for the mobility}   \label{eq:cvx-conc_split}
In \cite{bretin2020approximation,refId0}, a convex-concave splitting strategy was considered for the numerical treatment of \eqref{eq:MCH} and \eqref{eq:NMN_CH} based on the gradient structure of a functional $J$ defined as
$$J_{M,N,u}(\mu) :=  \int_Q \frac{M(u)}{2} |\nabla (N(u) \mu)| ^2 dx$$
with $N\equiv 1$ in the case of \eqref{eq:MCH}. The motivation for such a functional is the following: considering that, in a Cahn-Hilliard type flow, surface tensions are energetic terms whereas mobilities can be associated with the metric governing the flow, $J$ corresponds to a modification of the classical $\dot{H}^1$ semi-norm in order to incorporate the mobilities.

Using $J_{M,N,u}$, both \eqref{eq:MCH} and \eqref{eq:NMN_CH} can be rewritten as
$$
\begin{cases}
	u_t &=  - \nabla J_{M,N,u}(\mu) \\
	\mu &= \nabla P_{\varepsilon}(u)/\varepsilon 
\end{cases}
$$
Following Eyre's approach,  the mobility-weighted functional $J_{M,N,u}$ is decomposed in~\cite{bretin2020approximation,refId0} as the sum of a convex and a concave function:  
$$J_{M,N,u}(\mu) = J_{M,N,u,1}(\mu) - J_{M,N,u,2}(\mu),$$
and the following numerical scheme is considered:
$$
\begin{cases}
	(u^{n+1} - u^{n})/\delta_t& =  - \nabla J_{M,N,u^n,1}(\mu^{n+1})  + \nabla J_{M,N,u^n,2}(\mu^{n}), \\
	\mu^{n+1} &= \nabla \overline{P}_{\varepsilon,u^{n}}(u^{n+1})/\varepsilon.
\end{cases}
$$
As before, $\overline{P}_{\varepsilon,u^{n}}$ is defined by
$$ \overline{P}_{\varepsilon,u^{n}}(u)  :=P_{\varepsilon,c}(u) + P_{\varepsilon,e}(u^n) +
\langle \nabla_u  P_{\varepsilon,e}(u^n)  , (u - u^n) \rangle,$$
with
$$ P_{\varepsilon,c}(u) =  \frac{1}{2}\int_Q \varepsilon |\nabla u|^2  + \frac{\alpha}{\varepsilon^2} u^2  dx, \quad  P_{\varepsilon,e}(u) =  \int_{Q} \frac{1}{\varepsilon} \left( W(u) - \alpha \frac{u^2}{2}\right)dx.$$
The following settings are used in  \cite{bretin2020approximation,refId0}:
\begin{itemize}
	\item For the \eqref{eq:MCH} model ($N\equiv 1$):
	\begin{equation}   \label{eq:def_m_MCH}
J_{M,u,1}(\mu) =  \int_Q \frac{m}{2} |\nabla (\mu)| ^2dx, \quad \text{ with }\quad m: = \sup_{s \in [0,1]} M(s).
\end{equation}
	It is easily seen that
	\begin{equation}  \label{eq:def_e2_MCH}
	J_{M,u,2}(\mu) = \int_Q \frac{m-M(u)}{2} |\nabla \mu| ^2 dx,
	\end{equation}
	is  convex and quadratic. 
	\item For the \eqref{eq:NMN_CH} model (with $N(u)=1/\sqrt{M(u)}$):
	\begin{equation} \label{eq:def_e1_NMNCH}
	J_{M,N,u,1}(\mu) =  \int_Q \frac{m}{2} |\nabla (\mu)| ^2 + \frac{\beta}{2} |\mu|^2dx.
	\end{equation}
	Observing that
	\begin{align*}
		J_{M,N,u}(\mu) = \frac{1}{2}\int_Q M(u) | \nabla (N(u)\mu)|^2 dx =  \frac{1}{2} \int_Q |\nabla \mu|^2 dx + 
		\int_Q G(u) \cdot \nabla \mu \mu dx + \frac{1}{2} \int_Q |G(u)|^2 \mu^2 dx,
	\end{align*}
with $G(u) = - \frac{1}{2} \nabla (\log(M(u)))$, it follows that
	\begin{equation} \label{eq:def_e2_NMNCH}
	J_{M,N,u,2}(\mu)=  \frac{1}{2} \int_Q (m-1) |\nabla \mu|^2 dx - 
	\int_Q G(u) \cdot \nabla \mu \mu dx + \frac{1}{2} \int_Q  (\beta - |G(u)|^2) \mu^2 dx,
\end{equation}
	is positive, convex, and quadratic for $m, \beta$ sufficiently large. 
	
\end{itemize}
Despite the efficiency of the numerical schemes associated with these settings, and their empirical numerical stability observed in  \cite{bretin2020approximation,refId0}, a rigorous proof of stability (i.e., the decay of the underlying Cahn-Hilliard energy) could not be proved.
In the following section we consider a numerical variant of these schemes which maintains the same level of efficiency but for which stability properties can be proved. 

%
 \subsection{Scalar auxiliary variable schemes}   \label{sec:SAV_review}

The {\it Scalar Auxiliary Variable} (SAV) approach has been considered in a variety of papers \cite{HOU2019307,SHEN2018407,doi:10.113717M1150153,doi:10.113719M1264412,Huang2020} 
to tackle a large class of gradient flow systems, in particular the Cahn-Hilliard equation \cite{wang2022application,huang2023structurepreserving,yang}, dissipative systems \cite{bouchriti2020remarks,zhang2022generalized} and variants \cite{ESAV,ExpSAV}. 

In the particular case of the Cahn-Hilliard energy $P_\varepsilon$ \eqref{eq:P_epsilon}, the standard SAV approach  consists in introducing an auxiliary variable $r\,:Q\times [0,T]\to \R^+$  to stabilize the numerical discretization. This variable is usually associated with the reaction term in $P_{\varepsilon}$ and defined as
 \begin{equation}  \label{eq:def_rW}
 r := \sqrt{ \int_Q W(u) dx}.
\end{equation}
 Assuming that everything is smooth and that $\int_Q W(u) dx$ does not vanish, the derivation of this equality gives
 $$ r_t =\displaystyle \frac{1}{2 \sqrt{ \int_Q W(u) dx}  } \int_{Q} W'(u) u_t dx$$
 This leads to the following standard SAV relaxation of the Cahn-Hilliard system \eqref{eq:CH_eq}:
 $$
\left\{ \begin{array}{lll}
   \partial_t u &=&  \Delta \mu, \\
  \mu &=&\displaystyle -\Delta u + \frac{r}{ \sqrt{ \int_Q W(u) dx} } \frac{ W'(u)}{\varepsilon^2}, \\
  r_t &=&\displaystyle \frac{1}{2 \sqrt{ \int_Q W(u) dx}  } \int_{Q} W'(u) u_t dx,
 \end{array}\right.
 $$
 A first-order natural discretization scheme for this SAV flow is:
 $$
 \left\{ \begin{array}{lll}
 (u^{n+1} - u^{n})/\delta_t &=& \Delta \mu^{n+1}  \\
 \mu^{n+1} &=&\displaystyle  - \Delta u^{n+1} + \frac{r^{n+1}}{ \sqrt{\int_Q W(u^{n})dx}} \frac{W'(u^{n})}{\varepsilon^2} \\
 r^{n+1} - r^{n} &=& \displaystyle  \frac{1}{2 \sqrt{\int_Q W(u^n) dx } } \int_Q W'(u^n) (u^{n+1} - u^n) dx,
 \end{array}\right.
 $$
Updating $r$ in the third equation requires efficient integration methods, it is typically done using standard integration rules combined with a suitable discretization.

The interest of the above SAV numerical scheme comes by observing that:
\begin{align}
 - \delta_t \int_Q | \nabla \mu^{n+1}|^2 dx &=  \langle \mu^{n+1},  u^{n+1} - u^{n} \rangle = \int_Q \nabla u^{n+1} \cdot \nabla (u^{n+1} - u^n) \ dx + 2 (r^{n+1} -  r^{n}) r^{n+1}  \notag \\
 &= \tilde{P}_{\varepsilon}(u^{n+1},r^{n+1}) - \tilde{P}_{\varepsilon}(u^{n},r^n) +  \tilde{P}_{\varepsilon}(u^{n+1} - u^{n},r^{n+1}-r^n),   \label{eq:SAV_general}
\end{align}  
where  $\tilde{P}_{\varepsilon}$ is a relaxation of $P_{\varepsilon}$ defined by 
$$ \tilde{P}_{\varepsilon}(u,r): = \frac{1}{2}\int_Q |\nabla u|^2 dx + r^2.$$
It follows from \eqref{eq:SAV_general} and the non-negativity of $ \tilde{P}_{\varepsilon}$ that, for all $n$ and all $\delta_t>0$,
$$  \tilde{P}_{\varepsilon}(u^{n+1},r^{n+1}) \leq  \tilde{P}_{\varepsilon}(u^{n},r^{n}),$$
i.e., the scheme is unconditionally stable with respect to  $\tilde{P}_{\varepsilon}$.

The matrix form of \eqref{eq:SAV_general} is
$$ \begin{pmatrix}
    I_d &  - \delta_t \Delta \\
    \Delta  & I_d
   \end{pmatrix}  
   \begin{pmatrix}
    u^{n+1} \\
    \mu^{n+1}
   \end{pmatrix}
   =
    \begin{pmatrix}
    u^{n} \\
     \frac{r^{n+1}}{ \sqrt{\int_Q W(u^n) dx }  } \frac{1}{\epsilon^2} W'(u^{n})
   \end{pmatrix}.
  $$

Observe that
\begin{equation}
u^{n+1} = u^{n+1}_1 + r^{n+1} u^{n+1}_2,\quad \mu^{n+1} = \mu^{n+1}_1 + r^{n+1} \mu^{n+1}_2,\label{un+1,mun+1}
\end{equation}
where $( u_1^{n+1},\mu_1^{n+1})$ and $( u_2^{n+1},\mu_2^{n+1})$ are solutions to the following systems that can be solved easily in Fourier domain:
$$
 \begin{pmatrix}
    I_d &  - \delta_t \Delta \\
    \Delta  & I_d
   \end{pmatrix}  
   \begin{pmatrix}
    u_1^{n+1} \\
    \mu_1^{n+1}
   \end{pmatrix}
   = \begin{pmatrix}
    u^{n} \\
     0
   \end{pmatrix}, \qquad
    \begin{pmatrix}
    I_d &  - \delta_t \Delta \\
    \Delta  & I_d
   \end{pmatrix}  
   \begin{pmatrix}
    u_2^{n+1} \\
    \mu_2^{n+1}
   \end{pmatrix}
   = \begin{pmatrix}
    0 \\
      \frac{1}{ \sqrt{\int_Q W(u^n) dx }  } \frac{1}{\epsilon^2} W'(u^{n})
   \end{pmatrix}
 $$
Observe now that $r^{n+1}$ can be calculated with the last equation of \eqref{eq:SAV_general}. Indeed,
 $$  r^{n+1}  - r^{n} = h_n(u_1^{n+1}) + r^{n+1} h_n(u_2^{n+1}) - h_n(u^{n})$$
 where
 $$h_n(u) := \frac{1}{2 \sqrt{\int_Q W(u^n) dx } } \int_Q W'(u^n) u\  dx,
 $$
thus
$$ r^{n+1} = \frac{r^{n} -  h_n(u^{n}) + h_n(u_1^{n+1}) }{1 -  h_n(u_2^{n+1}) }.$$
Given  $( u_1^{n+1},\mu_1^{n+1})$, $( u_2^{n+1},\mu_2^{n+1})$, and $r^{n+1}$, the values of $u^{n+1},\mu^{n+1}$ follow from \eqref{un+1,mun+1}.

\medskip
 SAV approaches yield relatively efficient numerical schemes, with algorithmic costs that closely resemble those of convex-concave approaches, but endowed with stability guarantees. Furthermore, since higher-order time discretization can be used, more accurate orders  can be obtained while maintaining good stability properties, like for example in  \cite{yang} where a second-order scheme 
  is considered. As shown in \cite{JIANG2022110954}, SAV schemes perform well as long as $r^n$ remains close to the quantity $\sqrt{\int_Q W(u^n) dx}$ for every $n$, which, unfortunately, is not always the case. To enhance the accuracy of these schemes, the authors therein add an inertial term to enforce a better approximation, while preserving the decreasing property of the relaxation of $P_{\varepsilon}$.

The SAV approach can be straightforwardly applied to models \eqref{eq:MCH} and \eqref{eq:NMN_CH},  leading in the latter case to the following system   
$$
\left\{\begin{array}{lll}
 u_t &=& N(u)\div( M(u) \nabla ( N(u)\mu)) , \\
 \mu &=&\displaystyle -\varepsilon\Delta u + \frac{r(t)}{\varepsilon\sqrt{\int_Q W'(u) dx}}~W'(u),\\
 r_t & =&\displaystyle \frac{1}{2\varepsilon\sqrt{\int_Q W'(u) dx}}\int_Q W'(u)u_t~dx
\end{array}\right.
$$
which is naturally associated to the following numerical scheme:
\begin{equation}   \label{eq:SAV_naive}
\left\{\begin{array}{lll}
(u^{n+1} - u^{n})/\delta_t &=&  N(u^n)\div( M(u^n) \nabla ( N(u^n)\mu^{n+1}))  \\
\mu^{n+1} &= &\displaystyle- \Delta u^{n+1} + \frac{r^{n+1}}{ \sqrt{\int_Q W(u^{n})dx}} \frac{W'(u^{n})}{\varepsilon^2} \\
r^{n+1} - r^{n} &=&\displaystyle  \frac{1}{2 \sqrt{\int_Q W(u^n) dx } } \int_Q W'(u^n) (u^{n+1} - u^n) dx
\end{array}\right.
\end{equation}
The decay
$$  \tilde{P}_{\varepsilon}(u^{n+1},r^{n+1}) \leq  \tilde{P}_{\varepsilon}(u^{n},r^{n}).$$ 
can be easily obtained by simply adapting the proof of \cite[Theorem 2.1]{yang}, observing that:
\begin{align*}
 - \delta_t \int_Q M(u^n)| \nabla (N(u^n) \mu^{n+1})|^2 dx &= \langle \mu^{n+1},  u^{n+1} - u^{n} \rangle \\
 &= \int_Q \nabla u^{n+1} \cdot \nabla (u^{n+1} - u^n) dx + 2 (r^{n+1} -  r^{n}) r^{n+1} \\
 &= \tilde{P}_{\varepsilon}(u^{n+1},r^{n+1}) - \tilde{P}_{\varepsilon}(u^{n},r^n) +  \tilde{P}_{\varepsilon}(u^{n+1} - u^{n},r^{n+1}-r^n).
\end{align*}
However, compared to the non-degenerate case, \eqref{eq:SAV_naive} depends on a linear system that is no longer explicitly computable in Fourier space and is highly susceptible to numerical instabilities due to the choice of the mobilities $M$ and $N$, possibly leading to poor conditioning.  Note that to avoid possible vanishing of the denominators, in the definition  \eqref{eq:def_rW} a positive constant $C_0$ can be added.

With the perspective of defining SAV numerical schemes able to deal effectively with degenerate models like \eqref{eq:NMN_CH}, we show in the following section how  first and second-order in time SAV discretizations can be defined for the numerical resolution of rather general gradient flows. To illustrate the numerical accuracy and the stability property, we first focus on an exemplar linear problem. The SAV schemes introduced are then applied to the \eqref{eq:NMN_CH} model in Section  \ref{sec:applic_CH}.

\section{SAV approaches for general $J^{-1}$-gradient flows}   \label{sec:grad_flow}
We consider the following general gradient flow structure:
\begin{equation}  \label{eq:grad-flow}
\begin{cases}
 u_t &= - \nabla J(\mu) \\
 \mu &= \nabla E(u),
\end{cases}
\end{equation}
which is obviously related to flows previously discussed. Indeed, when $E$ is the Cahn-Hilliard energy,
\begin{enumerate}
\item The $L^2$-gradient flow of $E$, i.e. the Allen-Cahn equation, corresponds to the choice $J(\mu)=\frac 1 2\|\mu\|_{L^2}^2$,
\item The $\dot{H}^{-1}$ flow of $E$, i.e. the Cahn-Hilliard equation,  is obtained with $J(\mu)=\frac 1 2\|\nabla \mu\|_{L^2}^2$,
\end{enumerate}
The \eqref{eq:MCH} and \eqref{eq:NMN_CH} models do not strictly enter this general structure, for they both involve mobilities that depend on $u$, hence the metric structure changes along the flow. However, what really motivates our study is the derivation of numerical schemes associated with \eqref{eq:grad-flow}, and in a time-discrete setting a reformulation is possible that associates directly the \eqref{eq:MCH} and \eqref{eq:NMN_CH} models with numerical discretizations of \eqref{eq:grad-flow}, as we shall see with details in the next sections.

%
%

For simplicity, we leave the characterization and analysis of  \eqref{eq:grad-flow} in general function spaces for future work, and we examine the problem for $u, \mu \in {\mathcal H}$, where ${\mathcal H}$ is a Hilbert space. We focus on functionals $J,\, E\,: {\mathcal H}\to\R$ of the form
\begin{equation}
J(\mu) = \frac{1}{2} \| L \mu \|^2\quad  \text{ and } \quad  E(u) = \frac{1}{2} \| A u - b \|^2\label{def:JE}
\end{equation}
with $L,\,A\in{\mathcal L}(\mathcal H)$ continuous and $b\in{\mathcal H}$. With this choice, system \eqref{eq:grad-flow} specifies as
\begin{equation}   \label{eq:CH_gen2}
\left\{\begin{array}{lll}
 u_t &=& - L^*L \mu, \\
 \mu &=& A^* (Au - b).
\end{array}\right.
\end{equation}
Both $E$ and $J$ decrease along this flow since
\[\
 \frac{d}{dt} E(u) = \langle \nabla E(u) , u_t \rangle =  - \langle \mu , L^* L \mu \rangle = - \| L \mu\|^2 \leq 0,
\]
and
\[
  \frac{d}{dt} J(\mu) = \langle \nabla J(\mu) , \mu_t \rangle = - \langle u_t , A^* A u_t \rangle = - \| A u_t \|^2 \leq 0.
  \]

We mentioned above that the Allen-Cahn and the Cahn-Hilliard equations, as well as the \eqref{eq:MCH} and \eqref{eq:NMN_CH} models in a time discrete setting, are directly related to \eqref{sec:grad_flow}. But these four models are based on $P_{\varepsilon}$ (see \eqref{eq:P_epsilon}) which is not quadratic in contrast with the choice $E(u) = \frac{1}{2} \| A u - b \|^2$ in \eqref{def:JE}. This is again a matter of time discretization: for these four models, we will identify $E$ (up to an additive constant) with:
 $$    \frac 1\varepsilon\overline{P}_{\varepsilon,u^{n}}(u)  =\frac 1\varepsilon P_{\varepsilon,c}(u) + \frac 1\varepsilon P_{\varepsilon,e}(u^n) +
\frac 1\varepsilon\langle \nabla_u  P_{\varepsilon,e}(u^n)  , u - u^n \rangle.$$
As for $J$, let us examine for the sake of illustration the case of the \eqref{eq:MCH} model (see the next sections for the other models). We define for \eqref{eq:MCH} (still in a time discrete setting):
$$J(\mu)=J_{M,u^n}(\mu) =  \frac{1}{2} \int_Q M(u^n) |\nabla  \mu| ^2 dx$$
With these definitions, $J_{M,u^n}$ and $\overline{P}_{\varepsilon,u^{n}}$  are convex, quadratic and of the form \eqref{def:JE} with
$$L \mu =  \sqrt{M(u^n)} \nabla  \mu,$$
and 
$$A^*A u = - \Delta u + \frac{\alpha}{\varepsilon^2} u ,\quad A^*b =  \frac{1}{\varepsilon^2} (W'(u^n) - \alpha u^n).$$
Obviously, in the continuous setting such operators $L$ and $A$ are not continuous on $L^2$ or $\dot{H}^1$, but it will be the case in the discrete setting.


\medskip
We now discuss the numerical discretization of \eqref{def:JE} and \eqref{eq:CH_gen2}. Following the convex-concave splitting used in section \ref{eq:cvx-conc_split}, we decompose $J$ to obtain a scheme with a convex part treated implicitly and a concave part treated explicitly. For this, we introduce two linear operators $L_1$, $L_2$ such that
\begin{equation}   \label{eq:L1L2}
J(\mu) = J_1(\mu) - J_2(\mu)\qquad \text{ with }\quad  J_i(\mu) = \frac{1}{2} \| L_i  \mu \|^2,\;\;i\in\{1,2\}
\end{equation}
so that
$$L^*  L=L_1^*  L_1-L_2^*  L_2.$$

For the numerical discretization of \eqref{eq:grad-flow} with $J$ defined as in \eqref{eq:L1L2}, let us first recall the approach introduced in \cite{bretin2020approximation,refId0} which treats implicitly the mobility term $L_1$ and explicitly the mobility term $L_2$. Given a time step $\delta t>0$, this approach gives the following first-order discretization scheme:
\begin{equation} \label{eq:Conv_concave1st}
\begin{cases}  
  \frac{u^{n+1} - u^{n}}{\delta t} &=  -{L_1}^*  L_1 \mu^{n+1}  + {L_2}^* L_2 \mu^{n}   , \\
  \mu^{n+1} &= \nabla E(u^{n+1}),
\end{cases}
\tag{\textbf{Cvx split}}
\end{equation}
Although this numerical scheme is very effective in practice, no energy stability result of the form $E(u^{n+1}) \leq E(u^n)$ has been proved. We actually believe there is no such stability, we will show later a numerical counter-example where the energy rises over a few iterations before decreasing again.   

\begin{remark}[Choosing $L_1$ and $L_2$]
Splitting \eqref{eq:L1L2} is computationally interesting if $L_1$ is ``nice enough" to be treated implicitly.  The explicit treatment of $L_2$ reduces the algorithmic cost without affecting the stability. A typical situation where such decomposition makes sense is when $L_1$ is a spatially homogeneous operator that can be dealt with efficiently using Fourier techniques, whereas $L_2$ is spatially inhomogeneous.  For instance, in the particular case of \eqref{eq:MCH},  $J_1$ and $J_2$ can be chosen as
 $$ J_1(\mu) = \frac{1}{2}\int_Q m | \nabla \mu |^2 dx, \qquad\text{ and }\quad J_2(\mu) =  \frac{1}{2}\int_Q (m - M(u^n)) | \nabla \mu |^2 dx,$$
 where $M$ and $m$ are defined in \eqref{eq:def_m_MCH}.
\end{remark}

Let us now introduce an auxiliary scalar variable $r=\sqrt{J_2(\mu)}$ to relax the mobility $J_2$. After derivation, the following new gradient flow can be defined in the spirit of SAV approaches, we shall denote it as \MbSAV{}:
\begin{equation}
\begin{cases}  \label{eq:CH_gen_SAV}
	  u_t &=  -L_1^*  L_1 \mu  + \frac{r}{\sqrt{J_2(\mu)}} {L_2}^* L_2 \mu  , \\
  r_t &=  \frac{1}{2\sqrt{J_{2}(\mu)}} \langle  L_2  \mu,   L_2  \mu_t \rangle\\
 \mu &= \nabla E(u),
\end{cases}
\tag{\textbf{Mb-SAV}}
\end{equation}

\begin{remark} Note that when $L_2 \mu = 0$, the previous system reduces to 
$u_t =  -L_1^*  L_1 \nabla E(u)$, which can be discretized using standard implicit techniques.
\end{remark}

\begin{remark}
	We stress that, in contrast with existing SAV schemes, the relaxation in \eqref{eq:CH_gen_SAV} is not based on the splitting of the energy $E$, but rather of the mobility term $J$.
\end{remark}

In the next paragraphs, we examine first and second-order discretization schemes for this new gradient flow.

\subsection{A first-order SAV scheme}  \label{sec:SAV}

 We consider the following first-order discretization scheme:
\begin{equation} \label{eq:SAV_1st}
\begin{cases}  
  \frac{u^{n+1} - u^{n}}{\delta t} &=  -{L_1}^*  L_1 \mu^{n+1}  + \frac{r^{n+1}}{\sqrt{J_{2}(\mu^n)}} {L_2}^* L_2 \mu^{n}   , \\
  r^{n+1} - r^{n} &=  \frac{1}{2\sqrt{J_{2}(\mu^n)}} \langle  L_2  \mu^{n},   L_2 (\mu^{n+1} - \mu^n) \rangle .\\
 \mu^{n+1} &= \nabla E(u^{n+1}),
\end{cases}
\tag{\textbf{Mb-SAV 1}}
\end{equation}
which can be seen as a semi-implicit scheme that treats $-L_1^*L_1\mu$ and $r$ implicitly, while it treats explicitly $J_2(\mu)$ and its gradient  $L_2^*L_2\mu$.

\begin{proposition}  \label{th:decay_1stSAV}
Let $(u^n,\mu^n,r^n)_{n\geq 0}$ be the sequence of solutions to \eqref{eq:SAV_1st} starting from $(u^0,\mu^0,r^0)$ with $\mu^0=\nabla E(u^0)$ and $r^0= \sqrt{J_2(\mu^{0})}$. 
$$\forall n\in\N, \quad r^{n+1}\geq 0\Longrightarrow E(u^{n+1}) \leq E(u^{n}).$$
Furthermore, the functional $\tilde J$ defined by
\begin{equation}   \label{eq:def_Jtilde}
	 \tilde{J}(\mu,r) := \frac{1}{2} \| L_1 \mu\|^2 - r^2.
	\end{equation}
satisfies the following decay property:
$$\forall n\in\N,\quad \tilde{J}(u^{n+1},r^{n+1}) \leq \tilde{J}(u^{n},r^{n})$$ 
\end{proposition}
\begin{proof}~\\
\noindent {\it Decay of $E$}: from the definition of $r^n$ we get that
$$
    r^{n+1}-r^n = \frac{1}{2\sqrt{J_{2}(\mu^n)}}  \langle  L_2  \mu^{n},   L_2 \mu^{n+1} \rangle  - \sqrt{J_{2}(\mu^n)}.\notag
$$
Then, by Cauchy-Schwarz inequality:
$$  \frac{1}{2} \langle  L_2 \mu^{n+1} ,  L_2 \mu^{n} \rangle \leq  \frac{1}{2} \| L_2 \mu^n \| \| L_2 \mu^{n+1}  \| = \sqrt{J_2(\mu^n)} \sqrt{J_2(\mu^{n+1})},$$
thus
$$ r^{n+1} - \sqrt{J_2(\mu^{n+1})} \leq  r^n - \sqrt{J_2(\mu^{n})}.$$
By induction, we deduce from the assumption $r^0 = \sqrt{J_2(\mu^{0})}$ that
$$r^{n+1} \leq \sqrt{J_{2}(\mu^{n+1})}.$$
\noindent $E$ being convex, we have that
$$ E(u^{n+1}) + \langle \nabla E(u^{n+1}) , u^{n} - u^{n+1} \rangle \leq E(u^{n}),$$
so $E$ is decreasing if $\langle \nabla E(u^{n+1}) , u^{n+1} - u^{n} \rangle \leq 0$.
To prove it, we multiply the first equation of \eqref{eq:SAV_1st} by  $ \mu^{n+1}$ and we obtain, using  $r^{n+1}\leq \sqrt{J_{2}(\mu^{n+1})}$ and the assumption  $r^{n+1}\geq 0$, that
\begin{align*}
\frac{1}{\delta_t} \langle \nabla E(u^{n+1}) , u^{n+1} - u^{n} \rangle  &=~  - \left( \| L_1 \mu^{n+1} \|^2  -  \frac{r^{n+1}}{\sqrt{J_{2}(\mu^n)}}  \langle   L_2 \mu^{n+1} ,  L_2 \mu^{n} \rangle \right),  \\
&\leq ~- \left( \| L_1 \mu^{n+1} \|^2   - 2 r^{n+1}\sqrt{J_{2}(\mu^{n+1})} \right),\\
&\leq ~- \left( \| L_1 \mu^{n+1} \|^2   - 2J_{2}(\mu^{n+1}) \right)\\
&= ~- 2  J(\mu^{n+1})  \leq 0,
\end{align*}
which implies that 
$$ E(u^{n+1})  \leq  E(u^{n}).$$
\noindent {\it Decay of $\tilde J$}: As $E$ is convex, we have for all $n\in\N$
\begin{align*}
0 &\leq~\frac 1{\delta t}\langle \nabla E(u^{n+1}) - \nabla E(u^{n}), u^{n+1} - u^{n} \rangle \\ 
  &=~  \langle  \mu^{n+1} - \mu^{n},   -L_1^*  L_1 \mu^{n+1}  + \frac{r^{n+1}}{\sqrt{J_{2}(\mu^n)}} L_2^* L_2 \mu^{n}   \rangle \\
  &=~  \frac{1}{2} \left( - \| L_1 \mu^{n+1} \|^2  - \| L_1 (\mu^{n+1} - \mu^n) \|^2 +  \| L_1 \mu^{n} \|^2 \right) + \left( 2 r^{n+1} (r^{n+1} - r^{n}) \right) \\
  &=~  \tilde{J}(\mu^n,r^n) - \tilde{J}(\mu^{n+1},r^{n+1}) -  \tilde{J}(\mu^{n+1}- \mu^{n},r^{n+1} - r^{n}), 
\end{align*}
where $\tilde{J}$ is defined as in \eqref{eq:def_Jtilde}. Observing that, by Cauchy-Schwarz inequality,
\begin{align*}
    |r^{n+1}-r^n| =  \frac{1}{2\sqrt{J_{2}(\mu^n)}} |\langle  L_2  \mu^{n},   L_2 (\mu^{n+1} - \mu^n)| \rangle \leq  \sqrt{J_{2}(\mu^{n+1} - \mu^n)},
\end{align*}
we deduce that
$$  \tilde{J}(\mu^{n+1}- \mu^{n},r^{n+1} - r^{n}) \geq  \frac{1}{2} \| L_1 (\mu^{n+1} - \mu^{n})\|^2 - J_{2}(\mu^{n+1} - \mu^n)  = J(\mu^{n+1} - \mu^n)  \geq 0,$$
from which follows the decay property of $\tilde J$:
$$  \tilde{J}(\mu^{n+1},r^{n+1})  \leq  \tilde{J}(\mu^n,r^n).$$
\end{proof}

\begin{remark}
The result above relies on the assumptions $r^0= \sqrt{J_2(\mu^{0})}$ and $r^{n+1}\geq 0$. The first one is very natural for the SAV approach. The second one may not be true along the iterations even if $r^{n+1}$ is a good approximation of the positive quantity $\sqrt{J_{2}(\mu^{n+1})}$. In practice, we observe that for most numerical simulations $r^{n+1}$ becomes negative after sufficiently many iterations, and from there the scheme becomes unstable, i.e. the energy may increase. To solve this issue, the simplest solution is to add to the numerical scheme the constraint that $r^{n+1}\geq 0$, we shall discuss it later.
\end{remark}

We now explain how to compute the solution of the scheme \eqref{eq:SAV_1st} easily. 
Recall that whenever $E(u) = \frac{1}{2} \| Au - b \|^2$ the first and third equations of the scheme are equivalent to the system:
$$
\begin{pmatrix} I_d  & \delta_t {L_1}^*  L_1 \\ - A^* A  & I_d  \end{pmatrix} 
\begin{pmatrix}  u^{n+1} \\  \mu^{n+1}   \end{pmatrix} =  \begin{pmatrix}  u^{n} \\  -A^* b   \end{pmatrix} + r^{n+1}  \begin{pmatrix} \delta_t {L_2}^* L_2 \mu^{n}/\sqrt{J_{2}(\mu^n)} \\  0   \end{pmatrix}.
$$
It is easily seen that $ u^{n+1}=u_1^{n+1} + r^{n+1} u_2^{n+1}$ and $\mu^{n+1} =   \mu_1^{n+1} + r^{n+1} \mu_2^{n+1}$ with $(u^{n+1}_1, \mu^{n+1}_1)$, $(u^{n+1}_2, \mu^{n+1}_2)$ the solutions to the systems

\begin{equation}
\label{eq:u_1-mu_1}
\begin{pmatrix}  u^{n+1}_1 \\  \mu^{n+1}_1   \end{pmatrix} =  
\begin{pmatrix} I_d  & \delta_t {L_1}^*  L_1 \\ - A^* A  & I_d  \end{pmatrix}^{-1} 
\begin{pmatrix}  u^{n} \\ - A^*b   \end{pmatrix}
\end{equation}
and
\begin{equation}
\label{eq:u_2-mu_2}
 \begin{pmatrix}  u^{n+1}_2 \\  \mu^{n+1}_2   \end{pmatrix} =  
\begin{pmatrix} I_d  & \delta_t {L_1}^*  L_1 \\ - A^* A  & I_d  \end{pmatrix}^{-1} 
\begin{pmatrix}  \delta_t {L_2}^* L_2 \mu^{n}/\sqrt{J_{2}(\mu^n)} \\  0   \end{pmatrix}.
\end{equation}

To calculate $r^{n+1}$, observe that
$$   r^{n+1} - r^{n} =   h_n(\mu^{n+1}) - h_n(\mu^{n}) \text{ where } h_n(\mu) := \frac{1}{2\sqrt{J_{2}(\mu^{n})}}  \langle  L_2 \mu^{n} ,   L_2 \mu  \rangle.$$
Using $\mu^{n+1} =   \mu_1^{n+1} + r^{n+1} \mu_2^{n+1}$, we  deduce that 
\begin{equation}
\label{eq:r-nplus1}
r^{n+1} = \frac{ r^{n} -h_n(\mu^{n}) +  h_n(\mu_1^{n+1}) }{1 - h_n(\mu_2^{n+1})}
\end{equation}

In practice, the numerical scheme decomposes into three steps:
\begin{description}
\item[Step \#1:] Compute $(u^{n+1}_1, \mu^{n+1}_1)$, $(u^{n+1}_2, \mu^{n+1}_2)$ using \eqref{eq:u_1-mu_1} and \eqref{eq:u_2-mu_2}
\item[Step \#2:] Compute $r^{n+1}$ using \eqref{eq:r-nplus1}
\item[Step \#3:]  Compute $u^{n+1}=u_1^{n+1} + r^{n+1} u_2^{n+1}$ and $\mu^{n+1} =   \mu_1^{n+1} + r^{n+1} \mu_2^{n+1}$
\end{description}

\begin{remark} 
In view of Proposition~\ref{th:decay_1stSAV}, the stability of the scheme is ensured if $r^{n+1}\geq 0$. This condition can be easily plugged into the scheme by simply replacing Step $2$ with 
 $$ r^{n+1} = max\left\{\frac{ r^{n} -h_n(\mu^{n}) +  h_n(\mu_1^{n+1}) }{1 - h_n(\mu_2^{n+1})},0 \right\}.$$
This replacement does not modify the stability proof as in the case of $r^{n+1}=0$,  $(u^{n+1},\mu^{n+1})$ is just the solution of an implicit discretization scheme for the $L_1$ gradient flow of $E$:
 $$
 \begin{cases}
 u_t &= - \nabla J_1(\mu) = -L_1^{*} L_1 \mu , \\
 \mu &= \nabla E(u),
\end{cases}
 $$
 
\end{remark}

\subsection{A first-order improved SAV scheme}  \label{sec:SAV_improved}

The accuracy of the SAV scheme presented in the previous section depends on
the quality of the approximation $r^{n} \simeq \sqrt{J_2(\mu^n)}$. Unfortunately, this is a recurring problem with SAV approaches, and it appears that over fairly long time scales, this approximation can become quite false, and even $r^{n}$ can cancel out.  
However, unlike the classical SAV approach, it is not difficult to force this equality without disturbing the energy stability result, since $r$ only appears in the flow with respect to $J$ and not in the flow with respect to $E$. This brings us to consider the following first order improved scheme:

\begin{equation} \label{eq:SAV_1st_improved}
\begin{cases} 
  \tilde{r}^{n} &=  \sqrt{J_{2}(\mu^n)} \\
  \frac{u^{n+1} - u^{n}}{\delta t} &=  -{L_1}^*  L_1 \mu^{n+1}  + \frac{r^{n+1}}{\sqrt{J_{2}(\mu^n)}} {L_2}^* L_2 \mu^{n}   , \\
  r^{n+1} - \tilde{r}^{n} &=  \frac{1}{2\sqrt{J_{2}(\mu^n)}} \langle  L_2  \mu^{n},   L_2 (\mu^{n+1} - \mu^n) \rangle .\\
 \mu^{n+1} &= \nabla E(u^{n+1}),
\end{cases}
\tag{\textbf{Mb-SAV 1}+}
\end{equation}
which can be computed in three steps: 
\begin{description}
\item[Step \#1:] Compute $(u^{n+1}_1, \mu^{n+1}_1)$, $(u^{n+1}_2, \mu^{n+1}_2)$ using \eqref{eq:u_1-mu_1} and \eqref{eq:u_2-mu_2}
\item[Step \#2:] Since $\tilde{r}^n=h_n(\mu^n)$, compute $r^{n+1}$ using 
$$r^{n+1} = \frac{  h_n(\mu_1^{n+1}) }{1 - h_n(\mu_2^{n+1})}$$

\item[Step \#3:]  Compute $u^{n+1}=u_1^{n+1} + r^{n+1} u_2^{n+1}$ and $\mu^{n+1} =   \mu_1^{n+1} + r^{n+1} \mu_2^{n+1}$
\end{description}

%
%
%
In the limit $\delta_t\to+ \infty$, this scheme ensures the property that $r^{n+1} > 0$. Indeed,
$$
 \begin{pmatrix}  u^{n+1}_1 \\  \mu^{n+1}_1   \end{pmatrix} = \begin{pmatrix} (A^* A)^{-1} A^* b \\ 0 \end{pmatrix} \quad\text{ and }\quad 
  \begin{pmatrix}  u^{n+1}_2 \\  \mu^{n+1}_2   \end{pmatrix} = \begin{pmatrix} (A^* A)^{-1} L_2^{*} L_2 \mu^n / \sqrt{J_{2}(\mu^n)}  \\   (L_1^* L_1)^{-1} L_2^{*} L_2 \mu^n / \sqrt{J_{2}(\mu^n)}  \end{pmatrix},
$$
therefore 
$$ h_n(\mu_2^{n+1}) =  \frac{\| L_2 L_1^{-1} L_2 \mu^n  \|^2}{\|  L_2 \mu^n  \|^2}.$$
Observe that, being $L$ invertible,
$$\forall x\not=0,\qquad \|Lx\|^2= \|L_1x\|^2- \|L_2x\|^2>0$$
therefore, being $L_1$ invertible,
$$\forall x\not=0,\quad \|L_2(L_1)^{-1}x\|<\|x\|.$$
It follows that $ h_n(\mu_2^{n+1}) <1$ and then $r^{n+1} > 0$ (in the limit $\delta_t\to+ \infty$).

In all numerical simulations, we observe that the non-negativity condition $r^{n+1} \geq 0$ is always satisfied but we have not been able so far to prove it rigorously. We emphasize that even if the property were not true in general, we could force it by simply replacing Step \#2 in the scheme with
$$r^{n+1} = \max\left(\frac{  h_n(\mu_1^{n+1}) }{1 - h_n(\mu_2^{n+1})},0\right)$$

%
%

\subsection{A second-order SAV scheme}   \label{sec:SAV2}  

We now consider the second-order discretization scheme:
\begin{equation}   \label{eq:SAV_2}
\begin{cases}
  \frac{u^{n+1} - u^{n}}{\delta t} &=  -{L_1}^*  L_1 \overline{\mu}^{n+1/2}  + \frac{ \overline{r}^{n+1/2}}{\sqrt{J_{2}(\mu^{n+1/2})}} {L_2}^* L_2  \mu^{n+1/2}    , \\
  r^{n+1} - r^{n} &=  \frac{1}{2\sqrt{J_{2}(\mu^{n+1/2})}} \langle  L_2 \mu^{n+1/2},   L_2 (\mu^{n+1} - \mu^n) \rangle .\\
 \mu^{n+1} &= \nabla E(u^{n+1}),
\end{cases}
\tag{\textbf{Mb-SAV 2}}
\end{equation}
where 
$$\mu^{n+1/2} := \frac{\mu^{n} + \tilde{\mu}^{n+1}}{2}, \quad  \overline{\mu}^{n+1/2} := \frac{\mu^{n+1} + \mu^{n}}{2}, \quad \overline{r}^{n+1/2}: =  \frac{r^{n+1} + r^{n}}{2}$$
and  $\tilde{\mu}^{n+1}$ is an approximation of order $1$ of $\mu$ at  time $t_{n+1}$
which can be obtained, for instance, using the previous SAV scheme \eqref{eq:SAV_1st} as a preliminary step.

\begin{proposition}   \label{th:decay_2ndSAV}
Let $(u^n,\mu^n,r^n)_{n\geq 0}$ be the sequence of solutions to \eqref{eq:SAV_2} starting from $(u^0,\mu^0,r^0)$ with $\mu^0=\nabla E(u^0)$ and $r^0= \sqrt{J_2(\mu^{0})}$. 
$$\forall n\in\N, \quad \bar r^{n+1/2}\geq 0~\Longrightarrow ~E(u^{n+1}) \leq E(u^{n}).$$

Furthermore, the functional $\tilde J$ as in \eqref{eq:def_Jtilde} by
$$ 
	 \tilde{J}(\mu,r) := \frac{1}{2} \| L_1 \mu\|^2 - r^2.
$$
satisfies
$$  \tilde{J}(u^{n+1},r^{n+1}) \leq \tilde{J}(u^{n},r^{n}).$$
\end{proposition}

\begin{proof}
\noindent {\it Decay of $E$}: we first observe that
$$r^{n+1}-r^n = h_n(\mu^{n+1}) - h_n(\mu^{n})$$
where
$$h_n(\mu) = \frac{1}{2\sqrt{J_{2}(\mu^{n+1/2})}}  \langle  L_2\mu^{n+1/2},   L_2\mu \rangle.$$
Moreover, we remark that 
$$ h_n(\mu)  \leq   \frac{1}{2\sqrt{J_{2}(\mu^{n+1/2})}}  \| L_2 \mu^{n+1/2}] \|  ~ \| L_2 \mu  \| =  \sqrt{J_{2}(\mu)}.$$

In contrast with the first-order discretization scheme, and even when $r=0$, showing the decay of $E$ under a general convexity assumption for $E$ is not straightforward. To simplify the proof of stability, we assume in this case that $E$ is both convex and quadratic, which implies that it satisfies the following equality:
$${ E(u^{n+1}) + \langle (\nabla E(u^{n+1}) + \nabla E(u^{n}))/2 , u^{n} - u^{n+1} \rangle = E(u^{n})}.$$
In particular, $E$ is decreasing as soon as $ \langle (\nabla E(u^{n+1}) +\nabla E(u^{n}))/2  , u^{n+1} - u^{n} \rangle \leq 0$.

Multiplying the first equation by  $ \overline{\mu}^{n+1/2}$ and assuming that $\overline{r}^{n+1/2} \geq 0$, we obtain that

\begin{eqnarray*}
  \frac{1}{\delta_t} \langle (\nabla E(u^{n+1})  + \nabla E(u^{n+1}))/2 , u^{n+1} - u^{n} \rangle &=&  - \left(  \|  L_1[ \overline{\mu}^{n+1/2}] \|^2  - 2 \overline{r}^{n+1/2}  h_n( \overline{\mu}^{n+1/2})   \right),    \\
   &\leq&   - \left(  \|  L_1[ \overline{\mu}^{n+1/2}] \|^2  -  \left( J_2(\overline{\mu}^{n+1/2}) +  (\overline{r}^{n+1/2})^2 \right)     \right), \\
     &\leq&   - \left(   J\left(  \overline{\mu}^{n+1/2} \right) + \tilde{J}\left(  \overline{\mu}^{n+1/2},\overline{r}^{n+1/2}\right)    \right),\\
\end{eqnarray*}
where $\tilde{J}$ is defined as in \eqref{eq:def_Jtilde}.

\medskip
\noindent {\it Decay of $\tilde{J}$}: As $E$ is convex, we have that 
\begin{eqnarray*}
0 &\leq& \langle \nabla E(u^{n+1}) - \nabla E(u^{n}), u^{n+1} - u^{n} \rangle, \\ 
  &\leq&  \langle  \mu^{n+1} - \mu^{n},   -{L_1}^*  L_1 (\mu^{n+1} + \mu^n)/2  + \frac{(r^{n+1} + r^n)/2}{\sqrt{J_{2}(\mu^{n+1)}}} {L_2}^* L_2 \mu^{n+1/2}   \rangle, \\
  &\leq&  \frac{1}{2} \left( - \| L_1 \mu^{n+1} \|^2  +  \| L_1 \mu^{n} \|^2 \right) +   (r^{n+1})^2  - (r^{n})^2,  \\
  &\leq&  \tilde{J}(\mu^n,r^n) - \tilde{J}(\mu^{n+1},r^{n+1}),
\end{eqnarray*}
which implies the decay of $\tilde{J}$:
$$  \tilde{J}(\mu^{n+1},r^{n+1}) \leq  \tilde{J}(\mu^n,r^n).$$
 \end{proof}

Let us see now how the solution to scheme \eqref{eq:SAV_2} can be explicitly computed. We deduce from the first and third equations that 

$$
\begin{pmatrix} I_d  &  \frac{1}{2}\delta_t {L_1}^*  L_1 \\ - A^* A  & I_d  \end{pmatrix} 
\begin{pmatrix}  u^{n+1} \\  \mu^{n+1}   \end{pmatrix} =  \begin{pmatrix}  u^{n} - \frac{1}{2}\delta_t {L_1}^*  L_1 \mu^n \\  -A^* b   \end{pmatrix} + \overline{r}^{n+1/2}  \begin{pmatrix} \delta_t {L_2}^* L_2 \mu^{n+1/2}/\sqrt{J_{2}(\mu^{n+1/2})} \\  0   \end{pmatrix}.
$$
 The first step of the scheme computes
 $$\begin{pmatrix}  u^{n+1}_1 \\  \mu^{n+1}_1   \end{pmatrix} =  
\begin{pmatrix} I_d  & \frac{1}{2}\delta_t {L_1}^*  L_1 \\ - A^* A  & I_d  \end{pmatrix}^{-1} 
\begin{pmatrix}  u^{n} - \frac{1}{2}\delta_t {L_1}^*  L_1 \mu^n  \\ - A^*b   \end{pmatrix}$$
and
$$\begin{pmatrix}  u^{n+1}_2 \\  \mu^{n+1}_2   \end{pmatrix} =  
\begin{pmatrix} I_d  & \frac{1}{2}\delta_t {L_1}^*  L_1 \\ - A^* A  & I_d  \end{pmatrix}^{-1} 
\begin{pmatrix}  \delta_t {L_2}^* L_2 \mu^{n+1/2}/\sqrt{J_{2}(\mu^{n+1/2})} \\  0   \end{pmatrix},$$
so that $u^{n+1} =   u_1^{n+1} +\overline{r}^{n+1/2} u_2^{n+1}$ and $\mu^{n+1} =   \mu_1^{n+1} +\overline{r}^{n+1/2}\mu_2^{n+1}$.

The second step of the scheme estimates the value of $r^{n+1}$ using
$$   r^{n+1} - r^{n} =   h_n(\mu^{n+1}) - h_n(\mu^{n}) \text{ where } h_n(\mu) = \frac{1}{2\sqrt{J_{2}(\mu^{n+1/2})}}  \langle  L_2[\mu^{n+1/2}],   L_2[\mu] \rangle.$$
In particular, with $\mu^{n+1} =   \mu_1^{n+1} + \frac{r^{n+1}+r^{n}}{2} \mu_2^{n+1}$, we  deduce that 
$$
r^{n+1} = \frac{ r^{n} -h_n(\mu^{n}) +  h_n(\mu_1^{n+1}) + \frac{1}{2}r^n  h_n(\mu_2^{n+1}) }{1 - \frac{1}{2}h_n(\mu_2^{n+1})}
$$

As for the first order scheme, the positivity of $r^{n+1/2}$ is not guaranteed but can easily be imposed by setting
$$
r^{n+1} = \max \left( \frac{ r^{n} -h_n(\mu^{n}) +  h_n(\mu_1^{n+1}) + \frac{1}{2}r^n  h_n(\mu_2^{n+1}) }{1 - \frac{1}{2}h_n(\mu_2^{n+1})},0 \right)
$$

\subsection{A second-order improved SAV scheme}  \label{sec:SAV2impr}  
Just as we did for the first-order scheme, it is possible to improve the accuracy of the second-order SAV scheme by considering the following: 
\begin{equation}  \label{eq:SAV_2impr}
\begin{cases}
  \tilde{r}^{n} &= \sqrt{J_{2}(\mu^{n})} \\
  \frac{u^{n+1} - u^{n}}{\delta t} &=  -{L_1}^*  L_1 \overline{\mu}^{n+1/2}  + \frac{ \overline{r}^{n+1/2}}{\sqrt{J_{2}(\mu^{n+1/2})}} {L_2}^* L_2  \mu^{n+1/2}    , \\
  r^{n+1} - \tilde{r}^{n} &=  \frac{1}{2\sqrt{J_{2}(\mu^{n+1/2})}} \langle  L_2 \mu^{n+1/2},   L_2 (\mu^{n+1} - \mu^n) \rangle .\\
 \mu^{n+1} &= \nabla E(u^{n+1}),
\end{cases}
\tag{\textbf{Mb-SAV 2}+}
\end{equation}

With respect to first order, the stability is not perturbed, especially if we also impose the positivity of $r^{n+1}$ with 

$$
r^{n+1} = \max \left\{ \frac{h_n(\mu_1^{n+1}) + \frac{1}{2}r^n  h_n(\mu_2^{n+1}) }{1 - \frac{1}{2}h_n(\mu_2^{n+1})},0 \right\}. 
$$
This new improved scheme decreases $E$ like the other schemes.

\subsection{Numerical experiments}

We test the consistency of the numerical schemes described so far by considering  the following optimization problem defined for $u\in\R^2$:
\[
\min_u~ E(u) := \frac{1}{2}\| A u - b\|^2
\] 
for $A\in\R^{2\times 2}$ and $b\in\R^2$ given by:
\[
A = \begin{pmatrix} 0.25 & 0\\ 0 & 2 \end{pmatrix},\quad b = \begin{pmatrix} 0 \\ 0 \end{pmatrix}.
\]
We address the problem using the gradient flow structure \eqref{eq:CH_gen2} and the splitting \eqref{eq:L1L2} with
\[
L_1 = \begin{pmatrix} 1& 0\\ 0 & 1\end{pmatrix} \quad \text{ and }\quad L_2 = \begin{pmatrix} 0.5& -0.4\\ -0.4 & 0.5\end{pmatrix}.
\]
All schemes are initialised with $u^0=(0.1, 2)^t$.

We compare the performances of the following numerical schemes:
\begin{enumerate}
	\item The convex-concave splitting scheme \eqref{eq:Conv_concave1st} described at the beginning of Section \ref{sec:grad_flow};
	\item The 1st order SAV scheme \eqref{eq:SAV_1st} and its improved variant \eqref{eq:SAV_1st_improved} described in Sections \ref{sec:SAV} and \ref{sec:SAV_improved}, respectively;
	\item The 2nd order SAV scheme \eqref{eq:SAV_2}  and its improved variant \eqref{eq:SAV_2impr} described in Sections \ref{sec:SAV2} and \ref{sec:SAV2impr}, respectively.
\end{enumerate}

In the first test we validate the numerical consistency of the schemes by comparing the exact solution $u_{ex}$ with numerical solutions computed at a final time $T=5$ for ten different values of discretisation time step $\delta t \in [10^{-4}, 0.2]$. Our results are reported in Figure \ref{fig:orders}. We observe that the \eqref{eq:SAV_1st} scheme,  its improved version \eqref{eq:SAV_1st_improved} and the Cvx split scheme \eqref{eq:Conv_concave1st} show indeed a $O(\delta t)$ consistency, whereas the scheme \eqref{eq:SAV_1st_improved}  shows a smaller error. The second order consistency $O(\delta t^2)$ is obvious for schemes  \eqref{eq:SAV_2}  and \eqref{eq:SAV_2impr}. 

\begin{figure}[!t]
	\centering
	\includegraphics[height=7cm]{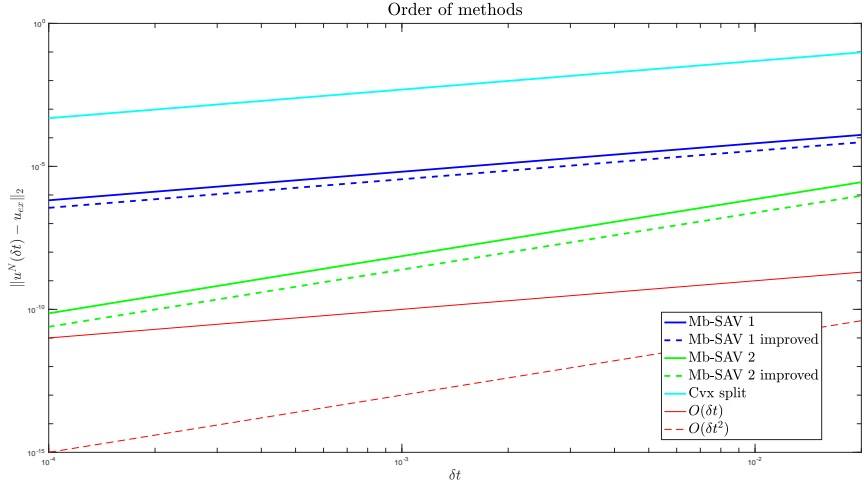}
	\caption{Consistency: for each method and for several values of $\delta_t$, plots of errors $\|u^N - u(T) \|$ with $T=5$, $N = T/\delta_t$.}
	\label{fig:orders}
\end{figure}

We now perform a stability analysis testing the numerical schemes above for different choices of the discretization time step
$\delta t \in \left\{0.1, 1, 4\right\}$.  Notice that  the time step $\delta_t = 4$ equals the CFL for the explicit scheme. For each choice of $\delta t$, we show the evolution of the trajectory of $u^n$ in the $x$-$y$ plane in comparison to the `exact' one computed at each time point by means of an exact (matrix exponential) solver in Figure \eqref{test_stability_p}. We also plot the evolution of the energy  $E(u^n)$ and the mobility $J(\mu^n)$ along iterations with $T=200$ in Figure \ref{test_stability_E} and \ref{test_stability_J}, respectively.

Interestingly, we observe in Figure \ref{test_stability_p} that for small time steps  $\delta t$ all methods show the decay of both the energy $E$ and the mobility $J$,  as well as an accurate approximation of the trajectory, while as $\delta t$ increases trajectories become fuzzier with possible significant deviations from the exact one. A better approximation of the trajectory can be  observed for the SAV approaches which clearly appear more accurate than that computed by the Cvx scheme. Note, however, that as $\delta t$ increases, first-order methods tend to be more effective than second-order approaches, which tend to give oscillations. 

In Figure \ref{test_stability_E}, the decay of the energy $E$ is observed independently of the time step, in accordance with Propositions \ref{th:decay_1stSAV} and \ref{th:decay_2ndSAV}. Note that the same does not hold for Cvs split solutions (cyan line) for which an increase of the energy values is observed in the early iterations. This clearly demonstrates the interest of SAV approaches for stabilizing the numerical behaviour.

As shown in Figure \ref{test_stability_J},  $J$ is not guaranteed to decrease along the iterations as such property holds only for its relaxed versions $\displaystyle\tilde{J}$ defined in \eqref{eq:def_Jtilde}, as shown in Figure \ref{test_stability_Jrelax}.

Lastly, we show in Figure\ref{test_approx_rn} the evolution of the numerical approximation $r^n$ of $\sqrt{J_2(\mu^n)}$ along iterations. Again, the approximation is accurate for small time steps ($\delta_t=0.1$) but, for time steps comparable with the CFL condition ($\delta_t=1$ and $\delta_t=4$), the approximation is of lower quality. Furthermore, $r^n$ remains positive for the first-order SAV scheme, but cancels out quite fast for the second-order scheme. The latter fact implies that, for excessively large time steps, the second-order SAV scheme flows with respect to the gradient structure of $J_1$, instead of $J$.

These numerical experiments show clearly that, in contrast with the classical splitting scheme \eqref{eq:Conv_concave1st}, the SAV approaches guarantee that the decay of the energy $E$, even for very large time steps. We also always observe the decay of the relaxed mobility $\tilde J$, but not always the decay of $J$, in particular for second-order schemes. For time steps below the CFL condition (equal for these experiments to $0.1$) second-order SAV approaches seem more accurate than first-order SAV approaches, i.e. the error between the iterate $u^n$ and the limit is smaller. But for time steps above the CFL condition, first-order SAV approaches become more accurate, at least on these examples, and we recommend to use them for large time steps instead of second-order SAV approaches.

\begin{figure}[!htbp]
	\centering 
	\includegraphics[height=4cm]{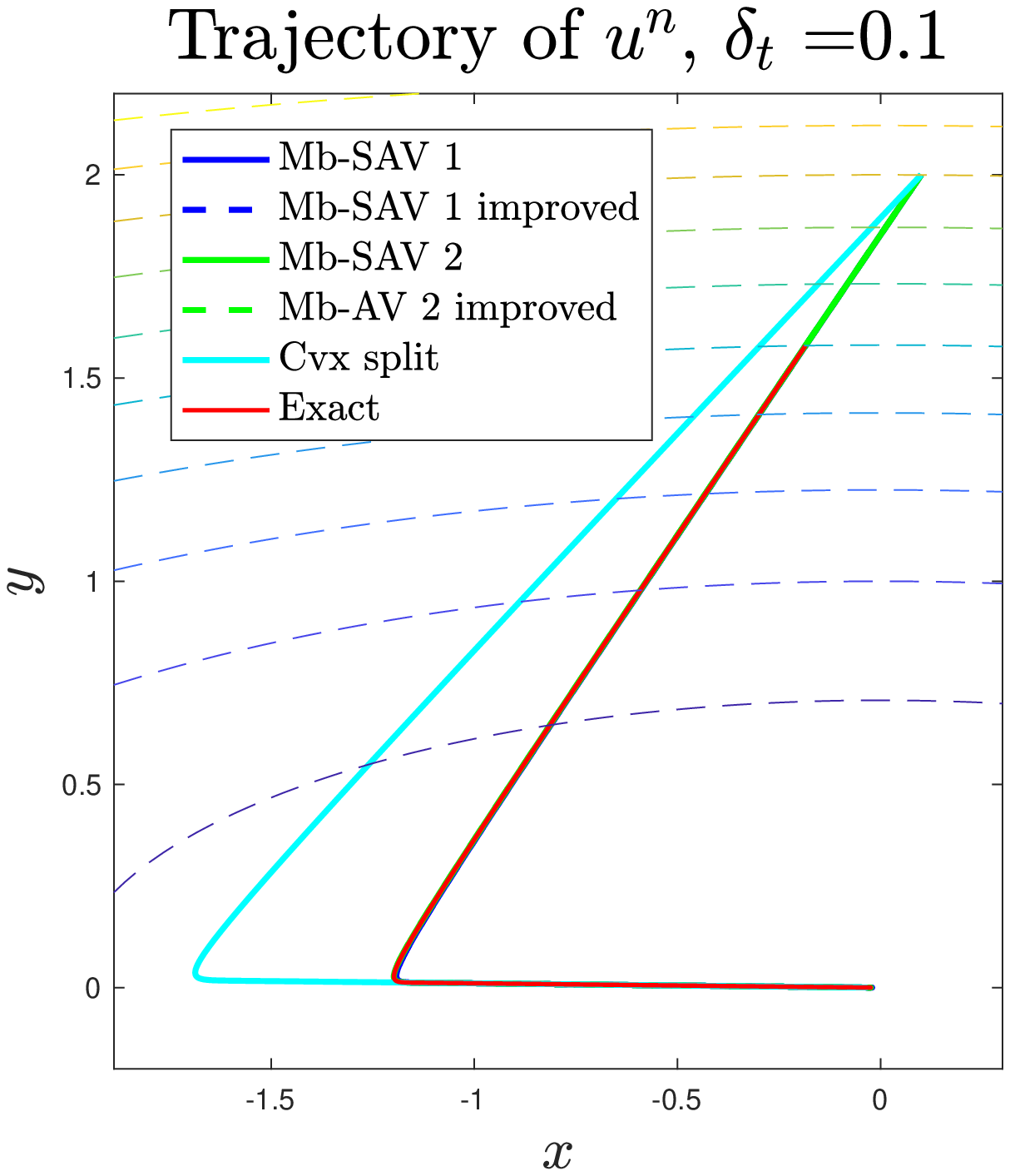}
	\includegraphics[height=3cm]{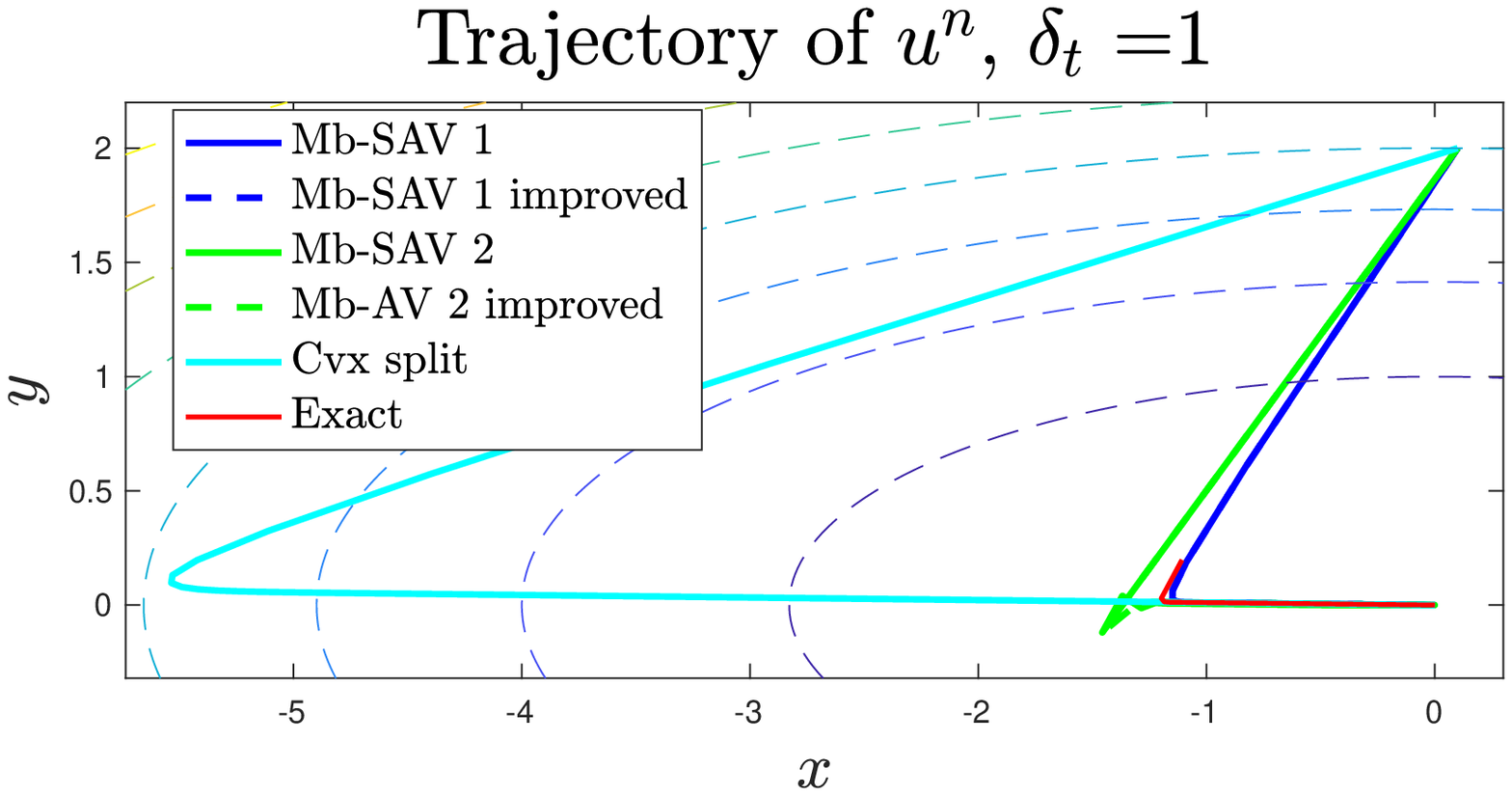}
	\includegraphics[height=3cm]{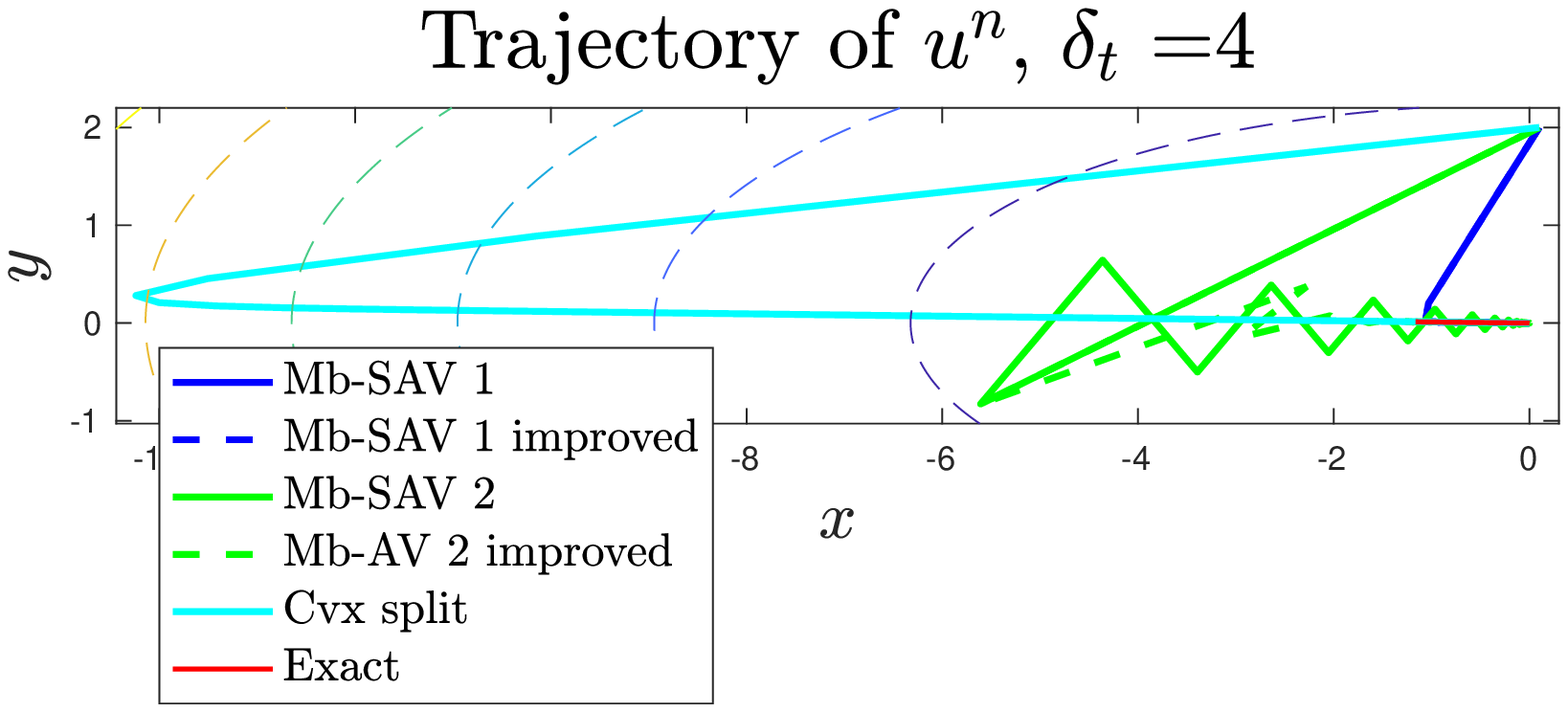}
	\caption{Trajectories of  $u^n$ along iterations for $\delta_t\in\{0.1,1,4\}$.}
	\label{test_stability_p}
	
\end{figure}

\begin{figure}[!htbp]
	\centering
	\includegraphics[width=0.32\textwidth]{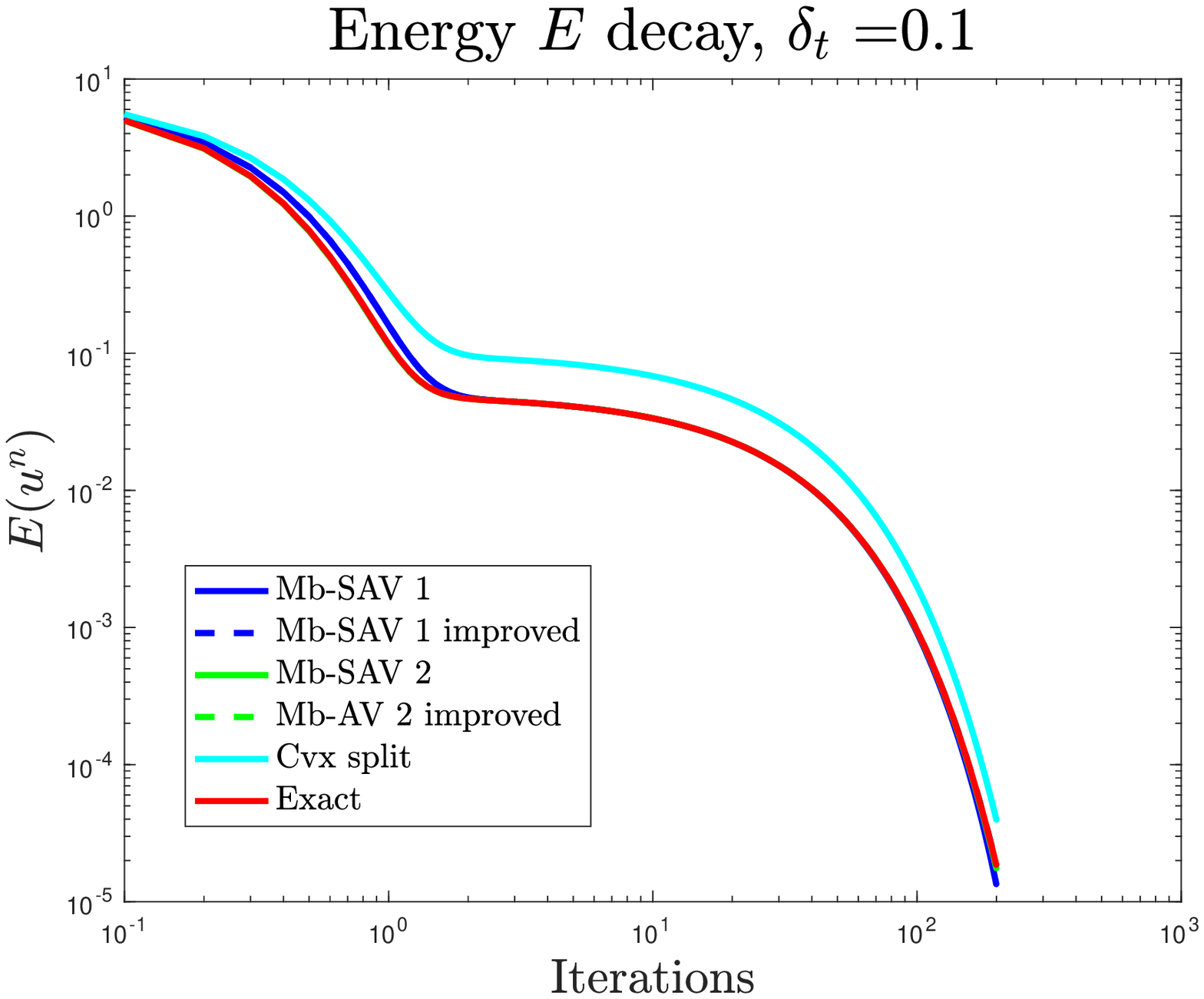}
	\includegraphics[width=0.32\textwidth]{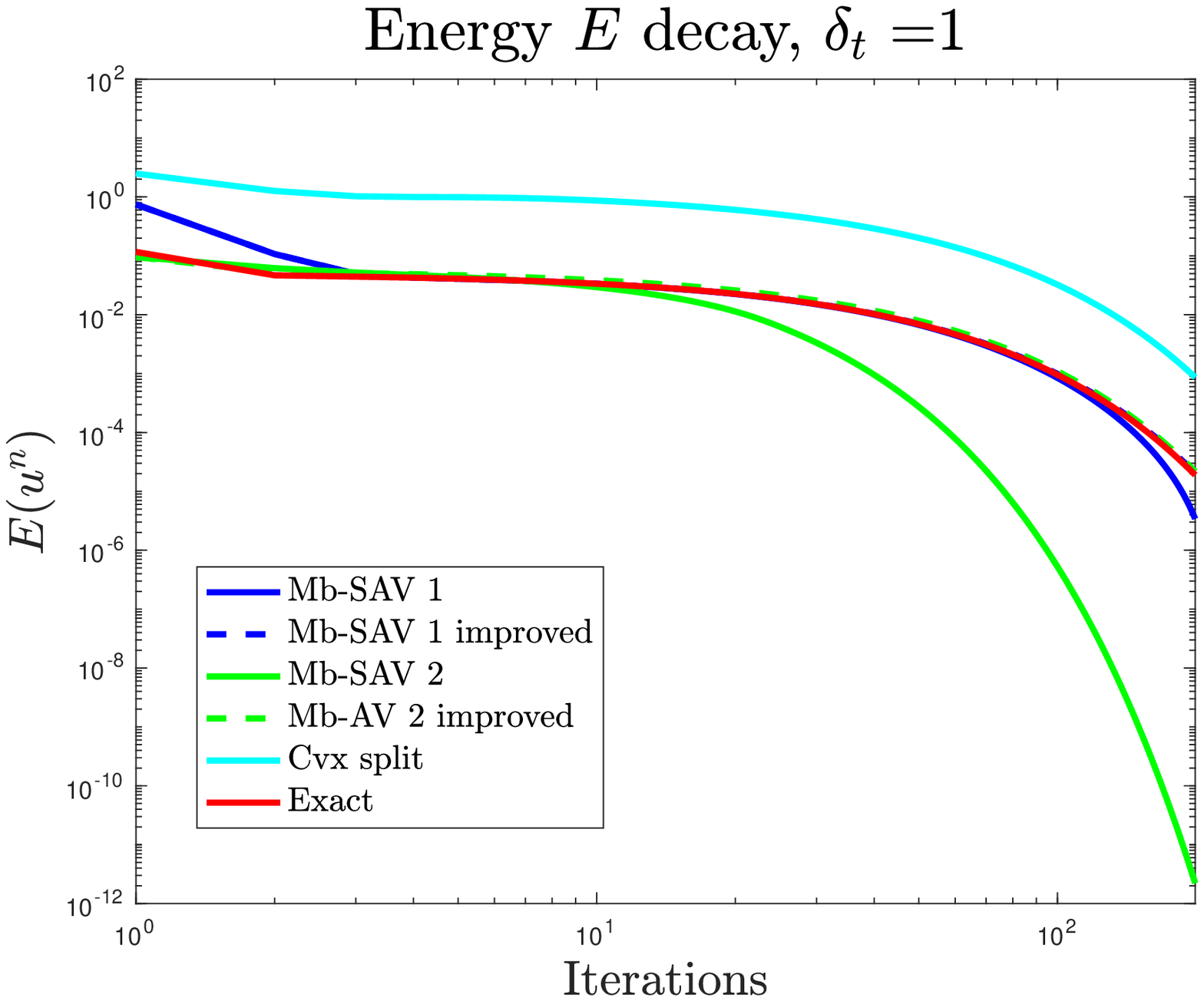}
	\begin{tikzpicture}[spy using outlines={circle,yellow,magnification=3.5,size=1.9cm, connect spies}]
		\node[anchor=south west,inner sep=0]  at (0,0) 
		{\includegraphics[width=0.32\textwidth]{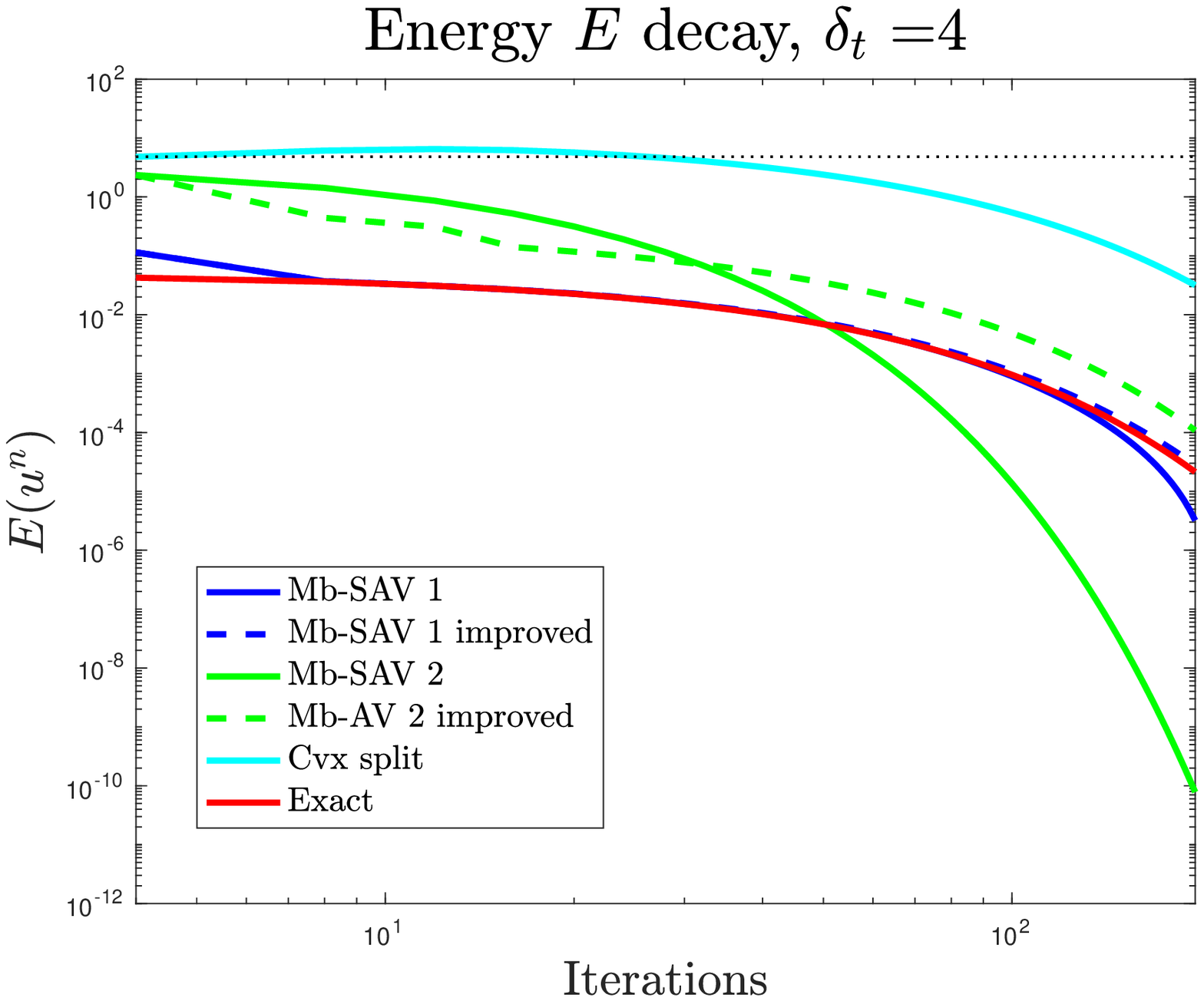}};
		\spy on (1.5,3.7) in node [left] at (4.8,1.8);
	\end{tikzpicture}
	\caption{Evolution of $E$ along iterations for $\delta_t\in\{0.1,1,4\}$. For $\delta_t =4$, the black dotted line $y=E_{\textrm{cvx}}(u^0)$ is plotted as a reference to show the increase of $E$ for the Cvx split scheme.}
	\label{test_stability_E}
\end{figure}

\begin{figure}[!htbp]
	\centering
	\includegraphics[width=0.32\textwidth]{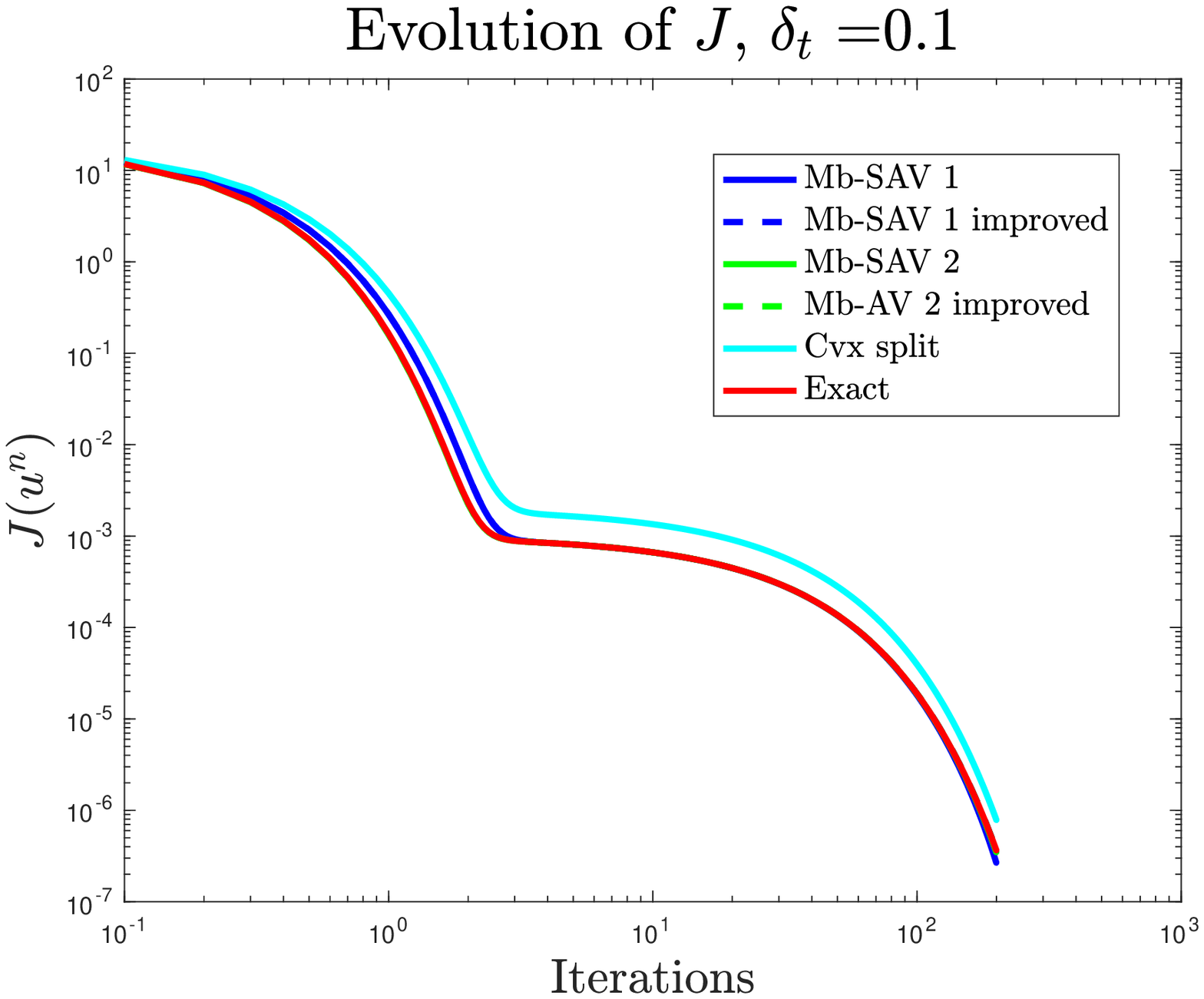}
		\begin{tikzpicture}[spy using outlines={circle,yellow,magnification=3.5,size=1.6cm, connect spies}]
		\node[anchor=south west,inner sep=0]  at (0,0) 
		{	\includegraphics[width=0.32\textwidth]{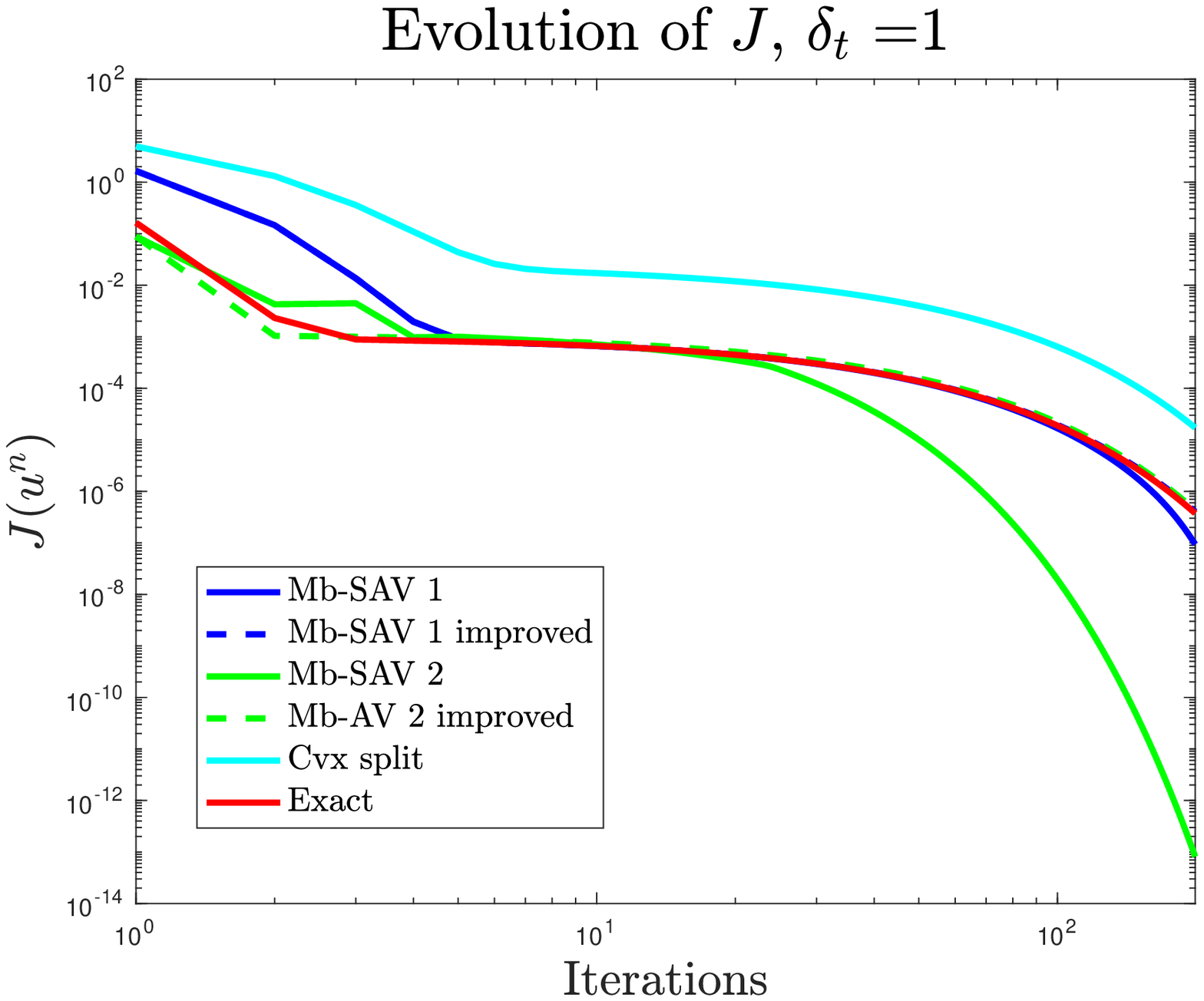}};
		\spy on (1.3,3) in node [left] at (4.8,1.4);
	\end{tikzpicture}
	\includegraphics[width=0.32\textwidth]{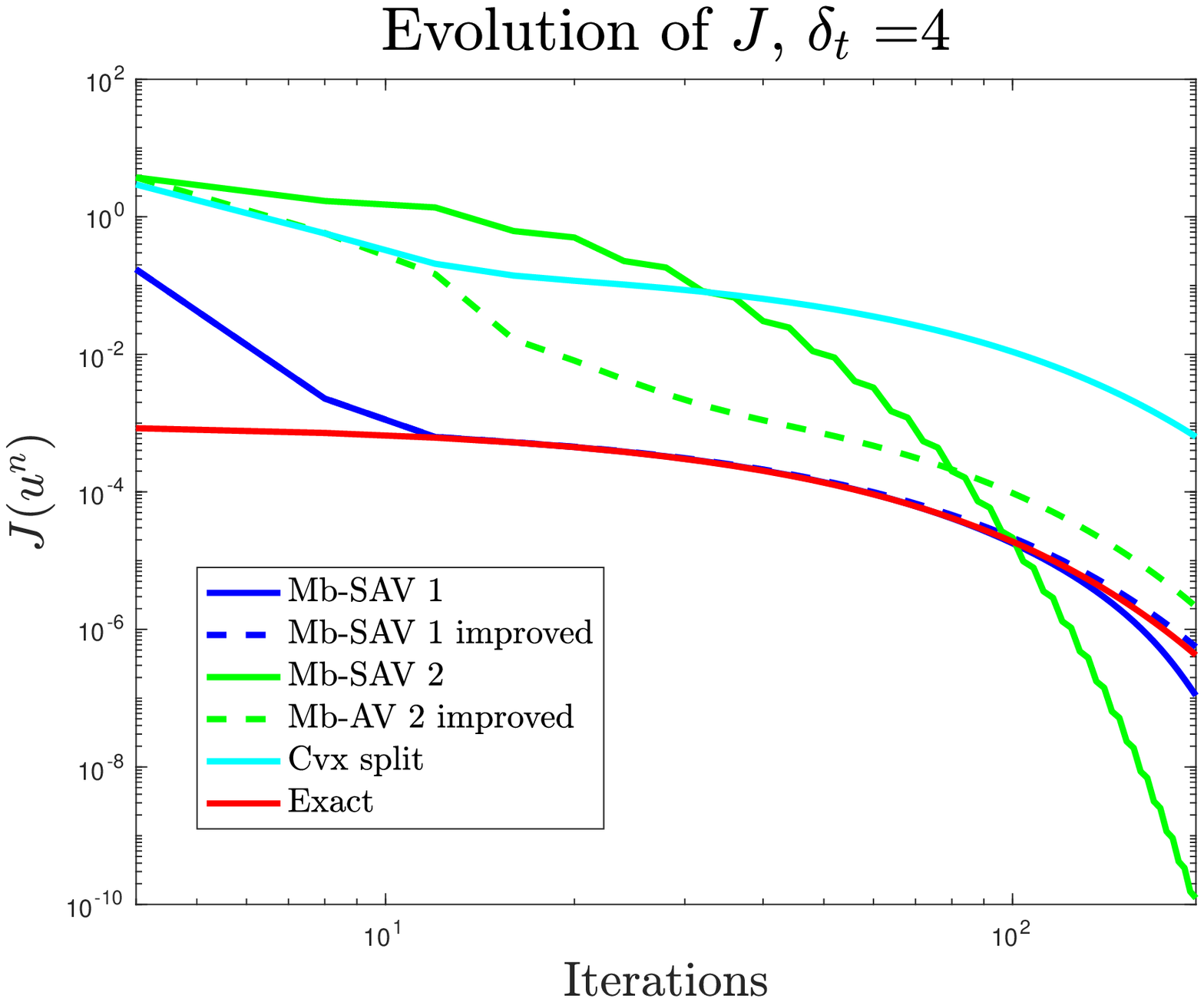}
	\caption{Evolution of $J$ along iterations for $\delta_t\in\{0.1,1,4\}$.}
	\label{test_stability_J}
\end{figure}

\begin{figure}[!htbp]
	\centering
	\includegraphics[width=0.32\textwidth]{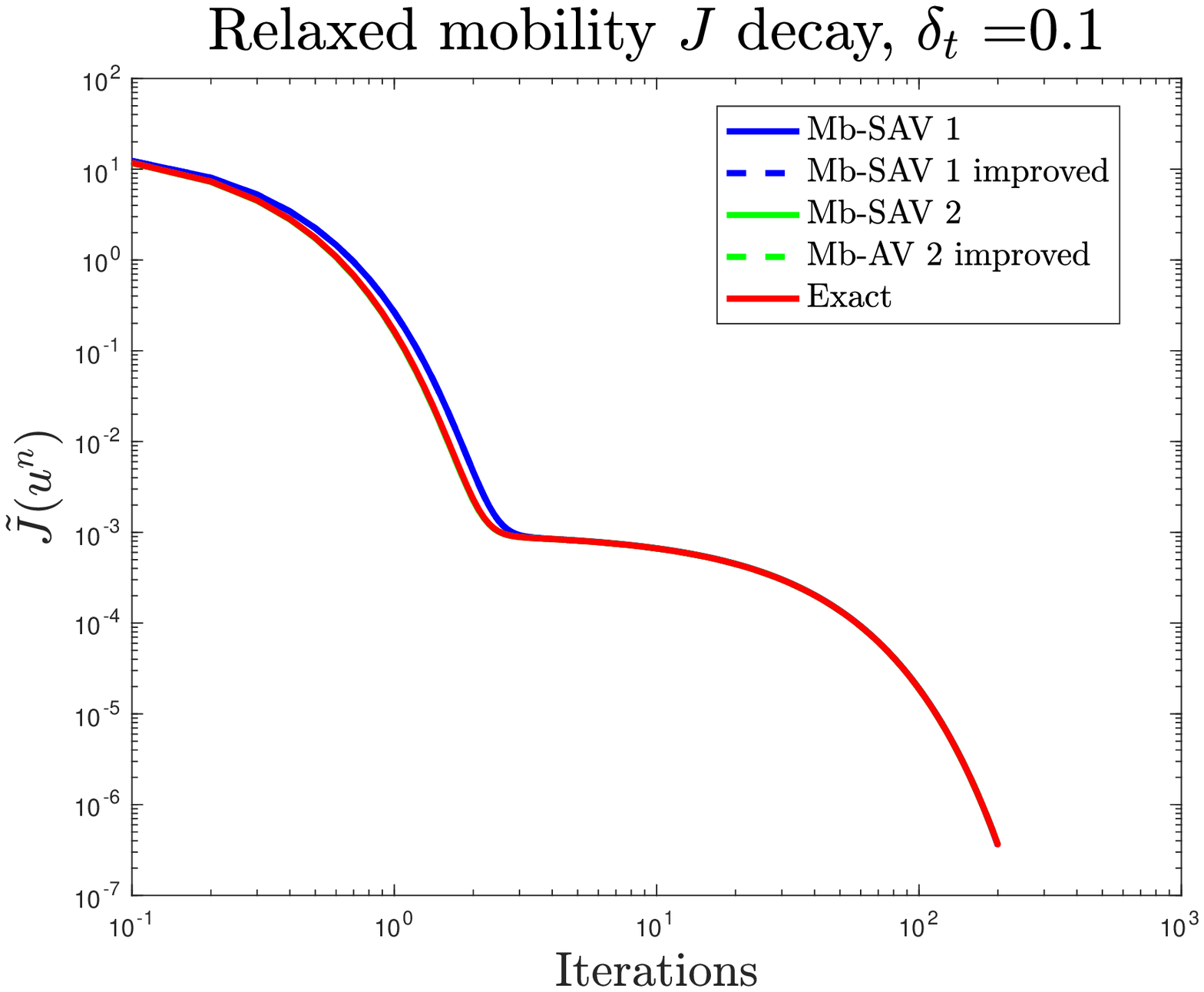}
	\includegraphics[width=0.32\textwidth]{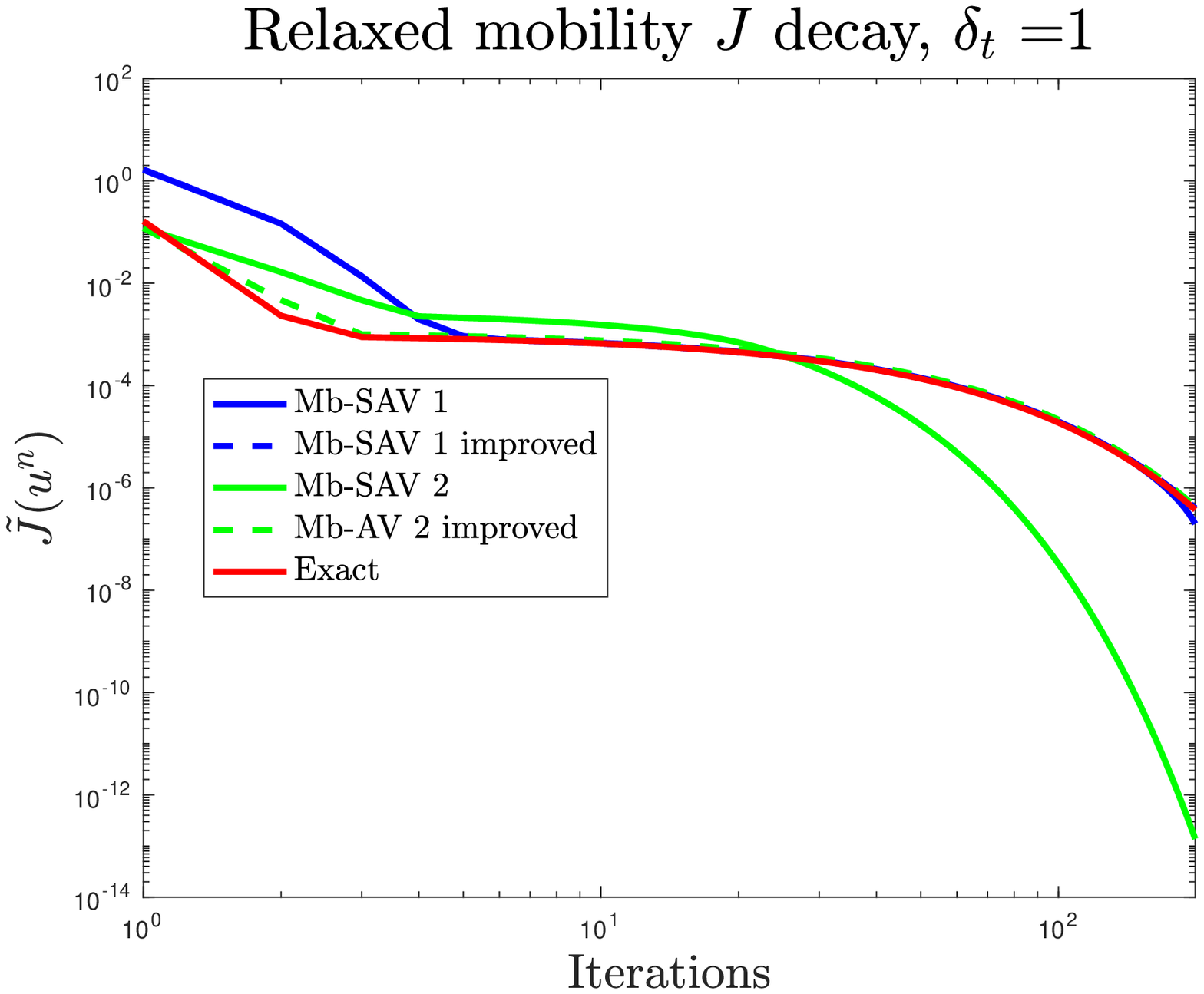}
	\includegraphics[width=0.32\textwidth]{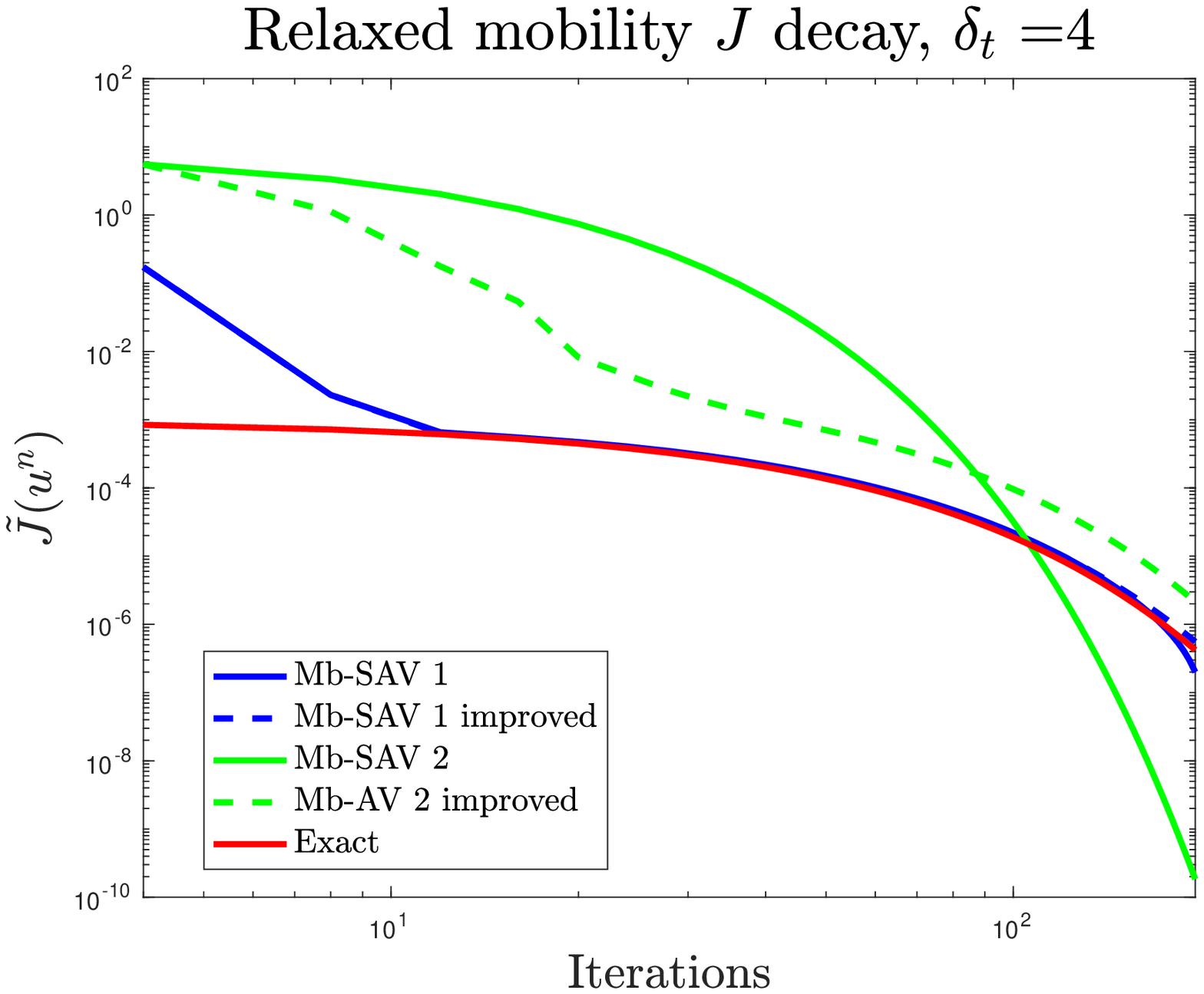}
	\caption{Evolution of $\tilde{J}$ in \eqref{eq:def_Jtilde} along iterations for $\delta_t\in\{0.1,1,4\}$.}
	\label{test_stability_Jrelax}
\end{figure}

\begin{figure}[!htbp]
	\centering
	\includegraphics[width=0.32\textwidth]{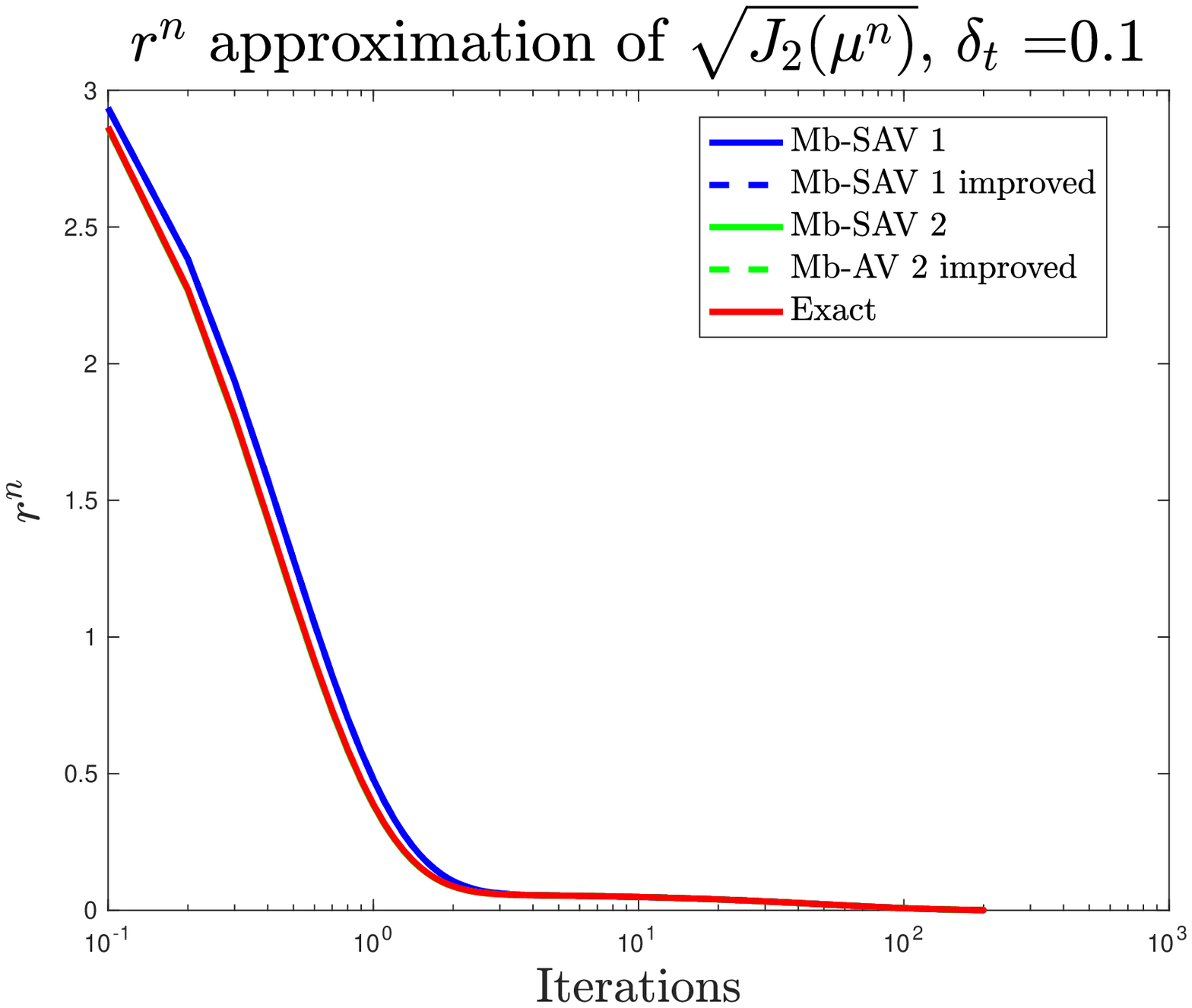}
	\includegraphics[width=0.32\textwidth]{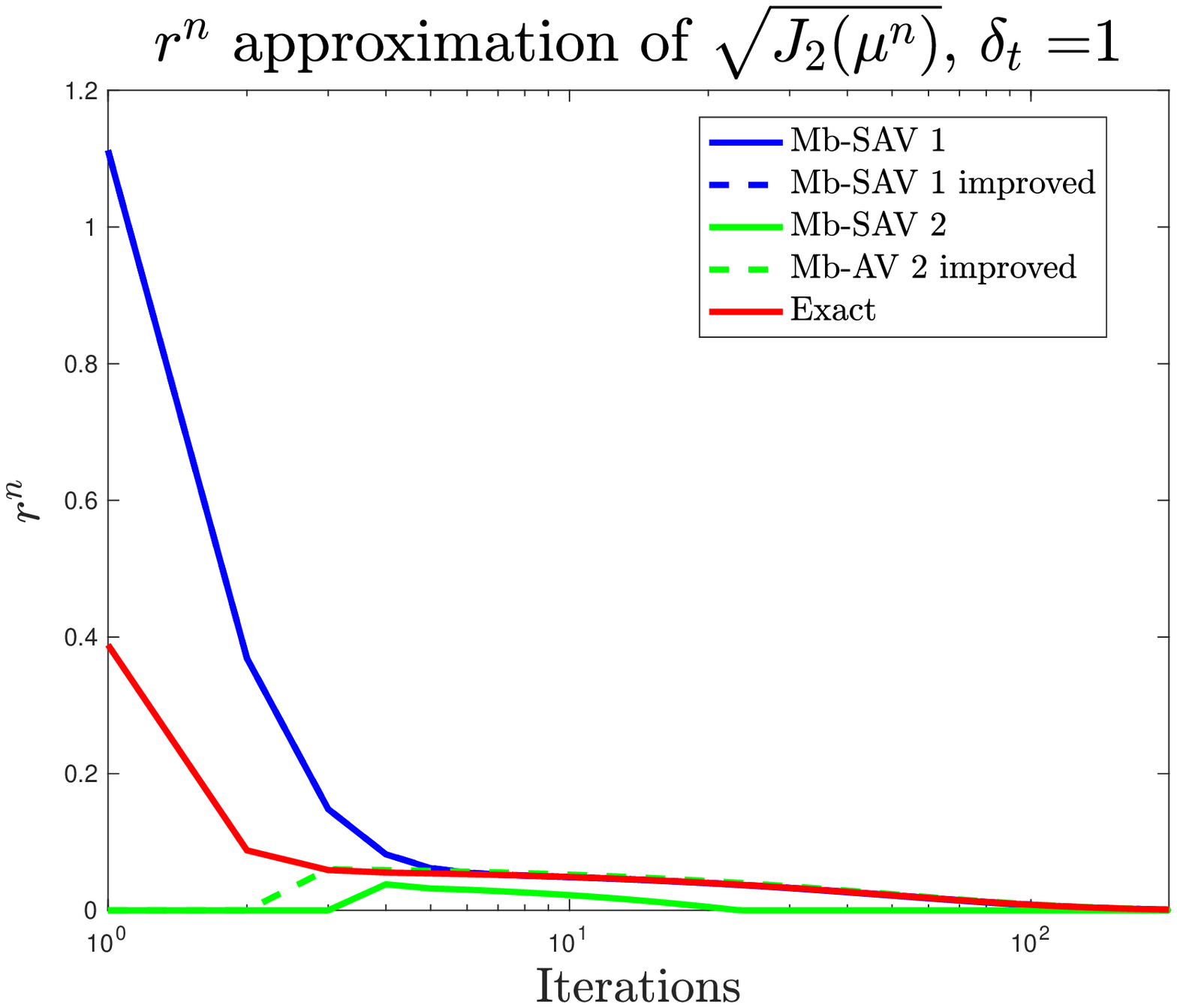}
	\includegraphics[width=0.32\textwidth]{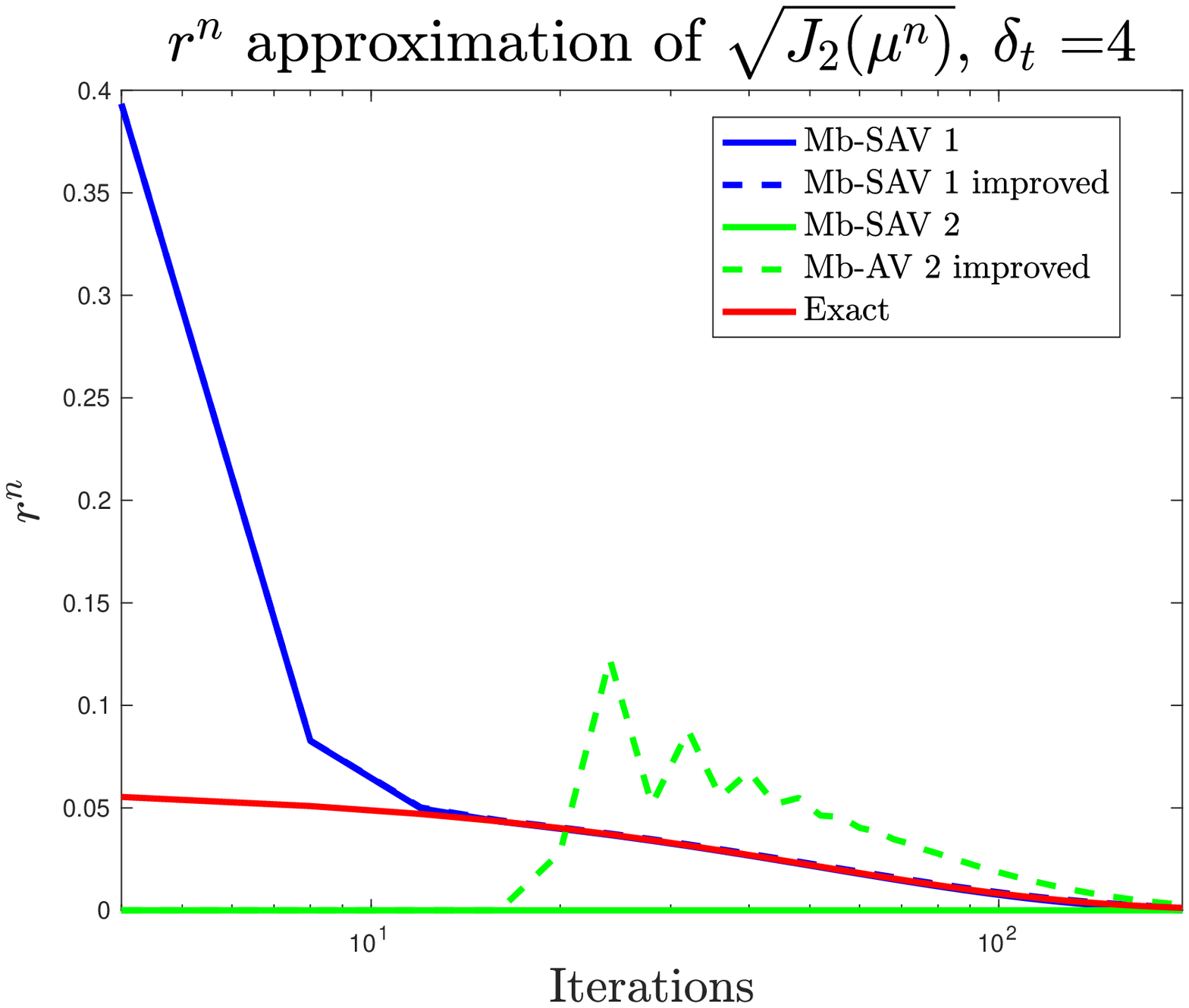}
	\caption{Quality of the approximation $r^n\approx \sqrt{J_2(\mu^n)}$ for the different schemes along iterations for $\delta_t\in\{0.1,1,4\}$.}
	\label{test_approx_rn}
\end{figure}

%
%
%

\section{Application  to Cahn-Hilliard models}  \label{sec:applic_CH}
The application of the SAV numerical schemes discussed in the previous section to the Cahn-Hilliard models with general mobilities \eqref{eq:MCH} and \eqref{eq:NMN_CH}
is quite immediate as it is based on the following observations:
\begin{itemize}
\item The mobility $J_{M,N,u}$ depends on $u$ and evolves as the flow of:
$$J_{M,N,u}(\mu) =  \int_Q \frac{M(u)}{2} |\nabla (N(u) \mu)| ^2 dx.$$
Recalling the discussion in Section \ref{eq:cvx-conc_split}, we can thus set up the SAV schemes by considering the splitting
\begin{equation}  \label{eq:split_NMN_CH_J}
J_{M,N,u}(\mu) = J_{M,N,u,1}(\mu) - J_{M,N,u,2}(\mu),
\end{equation}
with $ J_{M,N,u,1}$ defined as in \eqref{eq:def_m_MCH} and \eqref{eq:def_e1_NMNCH}, and $ J_{M,N,u,2}$ defined as in \eqref{eq:def_e2_MCH} and \eqref{eq:def_e2_NMNCH} for the  \eqref{eq:MCH} and \eqref{eq:NMN_CH}, respectively.

\item  The Cahn-Hilliard energy  $P_{\varepsilon}$ in \eqref{eq:P_epsilon} is neither convex nor quadratic.
Hence, we define an energy $E$ which evolves along iterations depending on the estimate $u^n$ at time $t_n$, i.e. we consider: 
$$  E_n(u) := \frac{\overline{P}_{\varepsilon,u^{n}}(u)}{\varepsilon}  :=\frac{1}{\varepsilon}(P_{\varepsilon,c}(u) + P_{\varepsilon,e}(u^n) +
\langle \nabla_u  P_{\varepsilon,e}(u^n)  , (u - u^n) \rangle),$$
where
$$ P_{\varepsilon,c}(u) =  \frac{1}{2}\int_Q \varepsilon |\nabla u|^2  + \frac{\alpha}{\varepsilon} u^2  dx \quad\text{ and } \quad P_{\varepsilon,e}(u) =  \int_{Q} \frac{1}{\varepsilon} (W(u) - \alpha \frac{u^2}{2})dx.$$
\end{itemize}

The adaptation of \eqref{eq:SAV_1st} to this setting is:
\begin{equation}  \label{eq:SAV_CH_NMN}
\begin{cases}
 \tilde{r}^{n} &= \sqrt{J_{M,N,u^{n},2}(\mu^n)} \\
 (u^{n+1} - u^{n})/dt &= - \nabla J_{M,N,u^{n},1}(\mu^{n+1})  + \frac{r^{n+1}}{\sqrt{J_{M,N,u^{n},2}(\mu^n)}}\nabla J_{M,N,u^{n},2}(\mu^{n}) \\
 r^{n+1} - \tilde{r}^{n} &= \frac{1}{2 \sqrt{J_{M,N,u^{n},2}(\mu^n)} } \displaystyle\int_Q \nabla J_{M,N,u^{n},2}(\mu^n) ( r^{n+1} - \tilde{r}^{n}) dx, \\
 \mu^{n+1} &= \nabla \overline{P}_{\varepsilon,u^{n}}(u^{n+1})/\varepsilon.
\end{cases}
\tag{\textbf{CH-Mb-SAV1}}
\end{equation}
By Proposition \ref{th:decay_1stSAV}, we deduce that
$$ \overline{P}_{\varepsilon,u^{n}}(u^{n+1}) \leq \overline{P}_{\varepsilon,u^{n}}(u^{n}).$$
The convex-concave splitting of $P_{\varepsilon}$
further entails that $P_{\varepsilon}(u^{n+1}) \leq P_{\varepsilon,u^{n}}(u^{n+1})$, so that using the property
$P_{\varepsilon}(u^{n}) = P_{\varepsilon,u^{n}}(u^{n})$,
we deduce that the Cahn-Hilliard energy is non-increasing along the iterations, that is:
$$ 
P_{\varepsilon}(u^{n+1}) \leq P_{\varepsilon}(u^{n}). 
$$

\begin{remark}
 Note that since $\overline{P}_{\varepsilon,u^n}$ is an approximation of $P_{\varepsilon}$, the use of SAV schemes of order $2$ does not appear useful here.  For extending the result to order $2$, first an approximation of $\tilde{u}^{n+1}$ at time $t_{n+1}$ should be considered and then used to define an approximation $u^{n+1/2} = (u^{n} + \tilde{u}^{n+1})/2$ and an energy: 
 $$  E_n(u) := \overline{P}_{\varepsilon,u^{n+1/2}}(u)/\varepsilon$$
 However, the stability analysis renders more difficult in this case. Applying an analogous argument as above entails only the decay of
 $\overline{P}_{\varepsilon,u^{n+1/2}}$ from which the decay of $P_{\varepsilon}$ cannot be easily deduced.
 A different idea would be to employ a second-order SAV relaxation for both the mobility and the energy $P_{\varepsilon}$. Although this choice may indeed lead to a simpler proof of stability, the corresponding numerical scheme would be more difficult.
\end{remark}

\subsection{Application to the \eqref{eq:MCH} model}
Let us now specify \eqref{eq:SAV_CH_NMN} for the \eqref{eq:MCH} model. We recall that in this case:
$$ J_{M,N,u,1}(\mu)=J_{M,u,1}(\mu) =  \int_Q \frac{m}{2} |\nabla \mu|^2dx\quad \text{ and }\quad J_{M,N,u,2}(\mu)=J_{M,u,2}(\mu) =  \int_Q \frac{m - M(u)}{2} |\nabla \mu|^2dx,$$
so that  \eqref{eq:SAV_CH_NMN} takes the form:
\begin{equation}  \label{eq:SAV_MCH}
\begin{cases}
  \tilde{r}^n &= \sqrt{J_{M,u^n,2}(\mu^n)} \\
  \frac{u^{n+1} - u^{n}}{\delta_t} &= m \Delta \mu^{n+1} - \frac{r^{n+1}}{\sqrt{J_{M,u^n,2}(\mu^n)}} \div((m - M(u^n)) \nabla \mu^{n})\\ 
  r^{n+1} - \tilde{r}^{n} &=  \displaystyle\frac{1}{2 \sqrt{J_{M,u^n,2}(\mu^n)}} \int_Q (m-M(u^n)) \nabla \mu^n \cdot \nabla (\mu^{n+1} - \mu^n) dx.\\
  \mu^{n+1} &=   - \Delta u^{n+1} + \frac{\alpha}{\epsilon^2} u^{n+1} + \frac{1}{\epsilon^2} \left(W'(u^{n}) - \alpha u^{n} \right),
\end{cases}
\tag{\textbf{M-CH-SAV1}}
\end{equation}
which shows that $(u^{n+1},\mu^{n+1})$ is solution of
$$
\begin{pmatrix} I_d  & -\delta_t m \Delta \\ \Delta - \alpha/\varepsilon^2 I_d & I_d  \end{pmatrix} 
\begin{pmatrix}  u^{n+1} \\  \mu^{n+1}   \end{pmatrix} =  \begin{pmatrix}  u^n \\   \frac{1}{\varepsilon^2}( W'(u^{n}) - \alpha u^n)  \end{pmatrix}   + r^{n+1} \begin{pmatrix}  - \frac{1}{\sqrt{J_{M,u^n,2}(\mu^n)}} \div((m - M(u^n)) \nabla \mu^{n})  \\   0  \end{pmatrix}.
$$
The problem above can be solved with the following two steps:
\begin{enumerate}
 \item Solve two linear systems by noticing that:
 $$u^{n+1} = u^{n+1}_1 + r^{n+1} u_2^{n+1}\quad  \text{ and }\quad  \mu^{n+1}=\mu^{n+1}_1 + r^{n+1} \mu_2^{n+1}  $$
where
$$ \begin{pmatrix}  u_1^{n+1} \\  \mu_1^{n+1}   \end{pmatrix} =  \begin{pmatrix} I_d  & - \delta_t m \Delta \\ \Delta - \alpha/\varepsilon^2 I_d & I_d  \end{pmatrix}^{-1}   \begin{pmatrix}  u^n \\    \frac{1}{\varepsilon^2}( W'(u^{n}) - \alpha u^n) \end{pmatrix},$$
and
$$
\begin{pmatrix}  u_2^{n+1} \\  \mu_2^{n+1}   \end{pmatrix}   =  \begin{pmatrix} I_d  & - \delta_t m \Delta \\  \Delta - \alpha/\varepsilon^2 I_d & I_d  \end{pmatrix}^{-1}  \begin{pmatrix}  - \frac{1}{\sqrt{J_{M,u^n,2}(\mu^n)}} \div((m - M(u^n)) \nabla \mu^{n})  \\   0  \end{pmatrix}.
$$
\item Estimate  $r^{n+1}$ by:
$$ r^{n+1} = \frac{ h_n(\mu_1^{n+1})}{ 1 -  h_n(\mu_2^{n+1})}\quad \text{ where } \quad h_n(\mu) := \frac{1}{2 \sqrt{J_{M,u^n,2}(\mu^n)}} \int_Q (m-M(u^n)) \nabla \mu^n \cdot \nabla \mu dx $$
 
\end{enumerate}

\subsection{Application to the \eqref{eq:NMN_CH} model}

For the \eqref{eq:NMN_CH} model, the scheme \eqref{eq:SAV_CH_NMN} can be specified analogously by considering the splitting \eqref{eq:split_NMN_CH_J} with:
$$ J_{M,N,u,1}(\mu) =  \int_Q \frac{m}{2} |\nabla (\mu)|^2 + \frac{\beta}{2} |\mu|^2 dx $$
and, for $N(u)=1/\sqrt{M(u)}$,
$$J_{M,N,u,2}(\mu) :=   \int_Q \frac{m}{2} |\nabla (\mu)|^2 + \frac{\beta}{2} |\mu|^2 - \frac{1}{2} M(u) | \nabla (N(u) \mu) |^2  dx $$
The corresponding scheme \eqref{eq:SAV_CH_NMN} takes the form:
\begin{equation} \label{eq:SAV_NMNCH}
\begin{cases}
  \tilde{r}^{n} &= \sqrt{J_{M,N,u^n,2}(\mu^n)} \\
  \frac{u^{n+1} - u^{n}}{\delta t} &=   \div(m \nabla \mu^{n+1})- \beta \mu^{n+1} \\
  &  ~ - \frac{r^{n+1}}{\sqrt{J_{M,N,u^n,2}(\mu^{n})}} \left( \div(m \nabla \mu^{n})  - \beta \mu^{n}  -    N(u^{n}) \div(M(u^{n}) \nabla(N(u^{n}) \mu^{n}) \right), \\
  r^{n+1} - \tilde{r}^{n} &=  h_n(\mu^{n+1}) -   h_n(\mu^{n})\\
 \mu^{n+1} &=  - \Delta u^{n+1} + \frac{\alpha}{\epsilon^2} u^{n+1}  + \frac{1}{\epsilon^2} \left(W'(u^{n}) - \alpha u^{n} \right),
\end{cases}
\tag{\textbf{NMN-CH-SAV1}}
\end{equation}
with
$$  h_n(\mu)  := \frac{1}{2\sqrt{J_{M,N,u^n,2}(\mu^n)}} \int_Q m \nabla \mu^{n} \cdot \nabla \mu +  \beta \mu^n \mu -  M(u^n) \nabla (N(u^n)\mu^n) \cdot \nabla (N(u^n) \mu ) dx, $$
and $ h_n(\mu^{n}) =  \sqrt{J_{M,N,u^n,2}(\mu^n)}$. Just as previously, this scheme can be solved in two steps
\begin{enumerate}
 \item Solve two linear systems by noticing that
 $$u^{n+1} = u^{n+1}_1 + r^{n+1} u_2^{n+1} \quad  \text{ and } \quad \mu^{n+1}= \mu^{n+1}_1 + r^{n+1} \mu_2^{n+1},  $$
where
$$ \begin{pmatrix}  u_1^{n+1} \\  \mu_1^{n+1}   \end{pmatrix} =  \begin{pmatrix} I_d  & - \delta_t (m \Delta - \beta I_d) \\ + \Delta - \alpha/\varepsilon^2 I_d & I_d  \end{pmatrix}^{-1}   \begin{pmatrix}  u^n \\    \frac{1}{\varepsilon^2}( W'(u^{n}) - \alpha u^n) \end{pmatrix},$$
and
$$
\begin{pmatrix}  u_2^{n+1} \\  \mu_2^{n+1}   \end{pmatrix}   =  \begin{pmatrix} I_d  & - \delta_t (m \Delta - \beta I_d) \\  \Delta - \alpha/\varepsilon^2 I_d & I_d  \end{pmatrix}^{-1}  \begin{pmatrix}  - \frac{ \left[ \div(m \nabla \mu^{n})  - \beta \mu^{n}  -    N(u^{n}) \div(M(u^{n}) \nabla(N(u^{n}) \mu^{n})) \right]}{\sqrt{J_{M,N,u^n,2}(\mu^{n})}}  \\   0  \end{pmatrix}.
$$
\item Estimate $r^{n+1}$ by:
$$ r^{n+1} = \frac{ h_n(\mu_1^{n+1})}{ 1 -  h_n(\mu_2^{n+1})}.
$$
 
\end{enumerate}

\subsection{Numerical experiments}
We now report a few numerical experiments using the schemes \eqref{eq:MCH} and \eqref{eq:NMN_CH}. The codes used for these experiments are available on the freely accessible GitHub repository \footnote{\url{https://github.com/lucala00/SAV_Cahn_Hilliard.git}}.

\subsubsection{Asymptotic expansion and flow: numerical comparison of the models}
The first numerical example is inspired from \cite{bretin2020approximation,refId0}
and  concerns the evolution of an initial connected set. The objective of this test is to show that the numerical solutions computed by \eqref{eq:MCH} and \eqref{eq:NMN_CH} are similar to those obtained  in \cite{bretin2020approximation}.  For both models,
we plot in Figure~\ref{fig_test1} the numerical phase field $u^n$ computed at different times. 
Each experiment is performed using the same numerical parameters: $N_1=N_2=2^8$, $\varepsilon = 2/N$, $\delta_t = \varepsilon^4$, $\alpha = 2/\varepsilon^2$, $m=1$, and $\beta= 2/\varepsilon^2$. The first, second and third line of Figure~\ref{fig_test1} show the approximate solutions computed by \eqref{eq:SAV_MCH} and \eqref{eq:SAV_NMNCH}, respectively. For both schemes, the stationary limit appears to correspond to a ball of the same mass as that of the initial set. 

The numerical flows obtained by by \eqref{eq:SAV_MCH} and \eqref{eq:SAV_NMNCH} are very similar and,  according to the asymptotic analysis of \cite{bretin2020approximation},  they are close to the surface diffusion flow.  
In the first two plots of Figure~\ref{fig_test1_profile}, we show  the slice $x_1 \mapsto u^T(x_1,0)$  of the solution computed at the final time $T = 10^{-4}$, where the middle plot represents a zoom in the right transition zone of the slice. The profile associated to the \eqref{eq:MCH} model is plotted in red and clearly shows that the solution $u^T$ does not remain in $[0,1]$  with an overshoot of order $O(\varepsilon)$.  In contrast, the profile obtained using the \eqref{eq:NMN_CH} model (in green) appears closer to the optimal profile $q(d(\cdot)/\varepsilon)$ and almost remains in $[0,1]$ up to an error of order $O(\varepsilon^2)$, see Section \ref{sec:doubly_mob}.
The last plot of Figure~\ref{fig_test1_profile} shows the decreasing of the Cahn-Hilliard energy $P_\varepsilon$
along the flow for each model. \\

These numerical experiments show the good approximation properties of both phase-field models.   Moreover, in spite of the apparent complexity of the \eqref{eq:NMN_CH} model, the proposed SAV approach allows to obtain relatively simple schemes with equivalent complexity for both models.

\begin{figure}[htbp]
\centering
	\includegraphics[width=3.5cm]{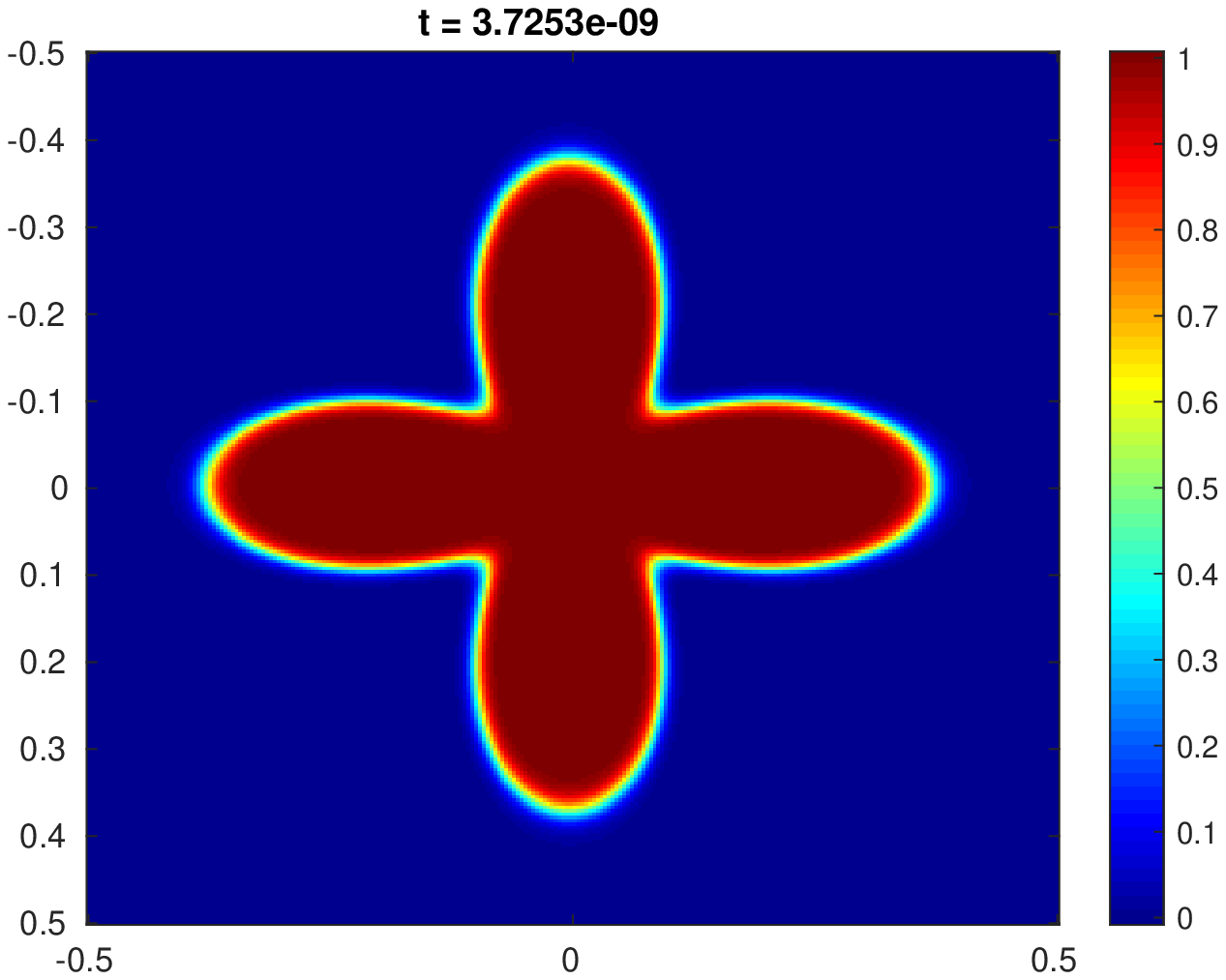}
	\includegraphics[width=3.5cm]{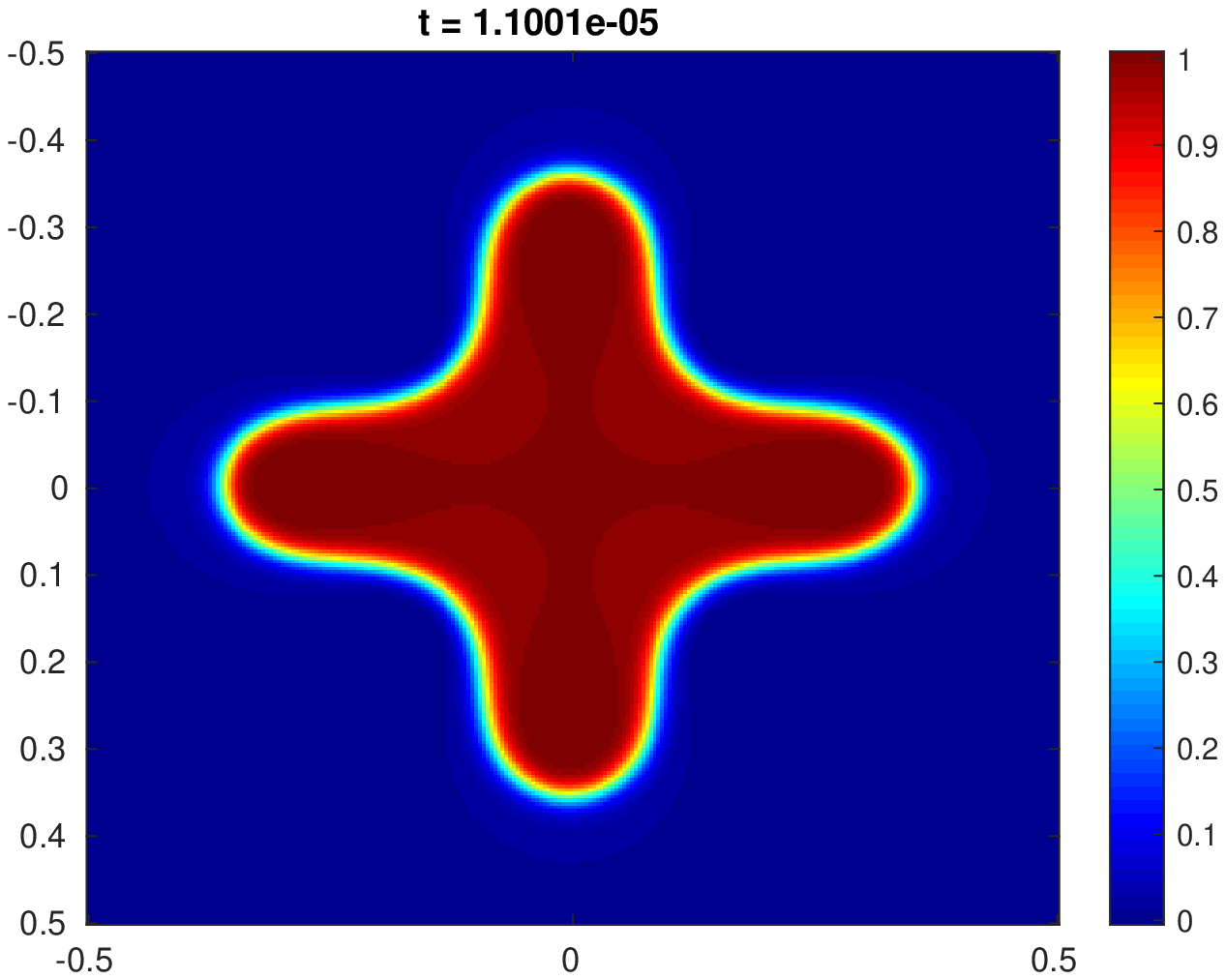}
	\includegraphics[width=3.5cm]{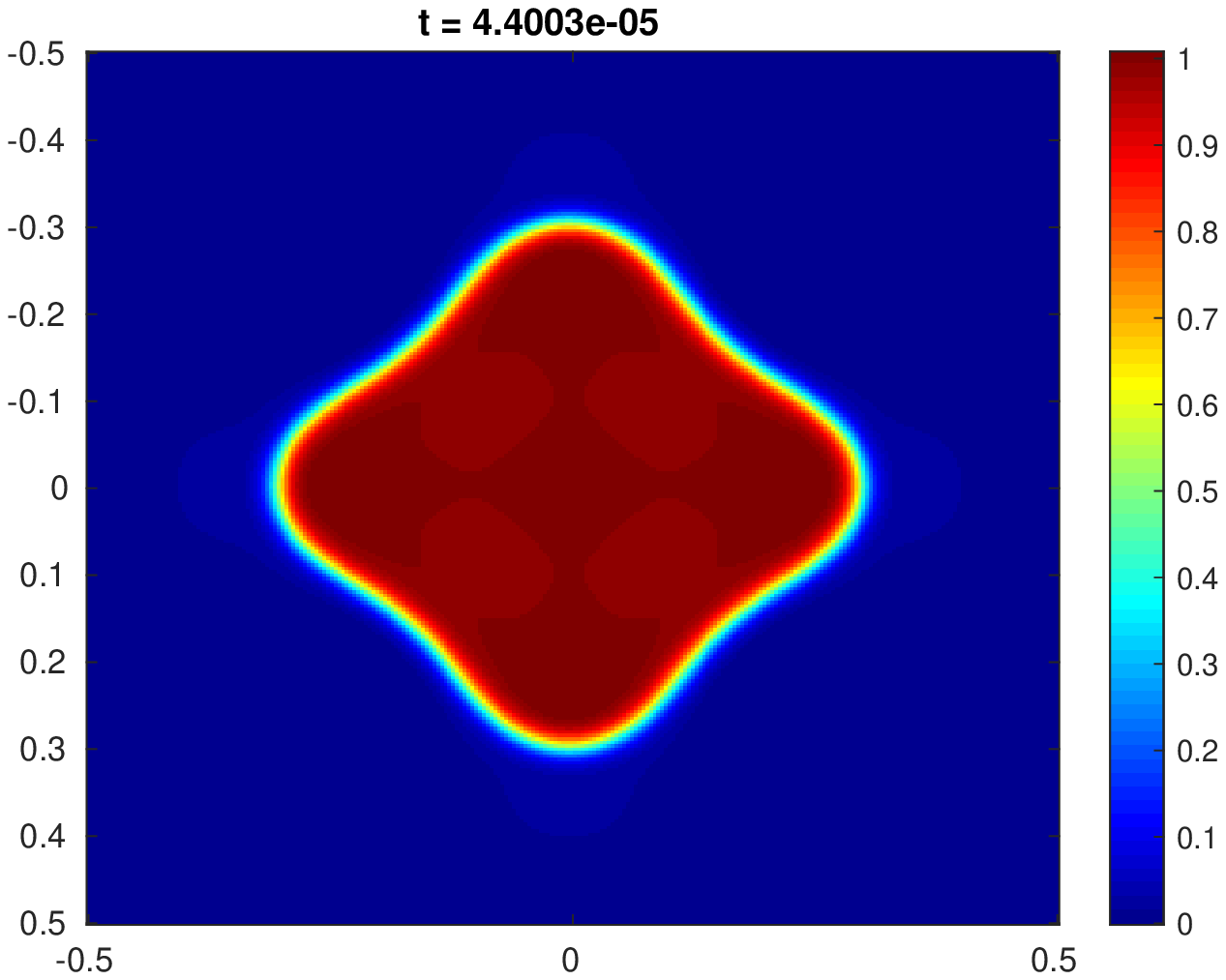}
	\includegraphics[width=3.5cm]{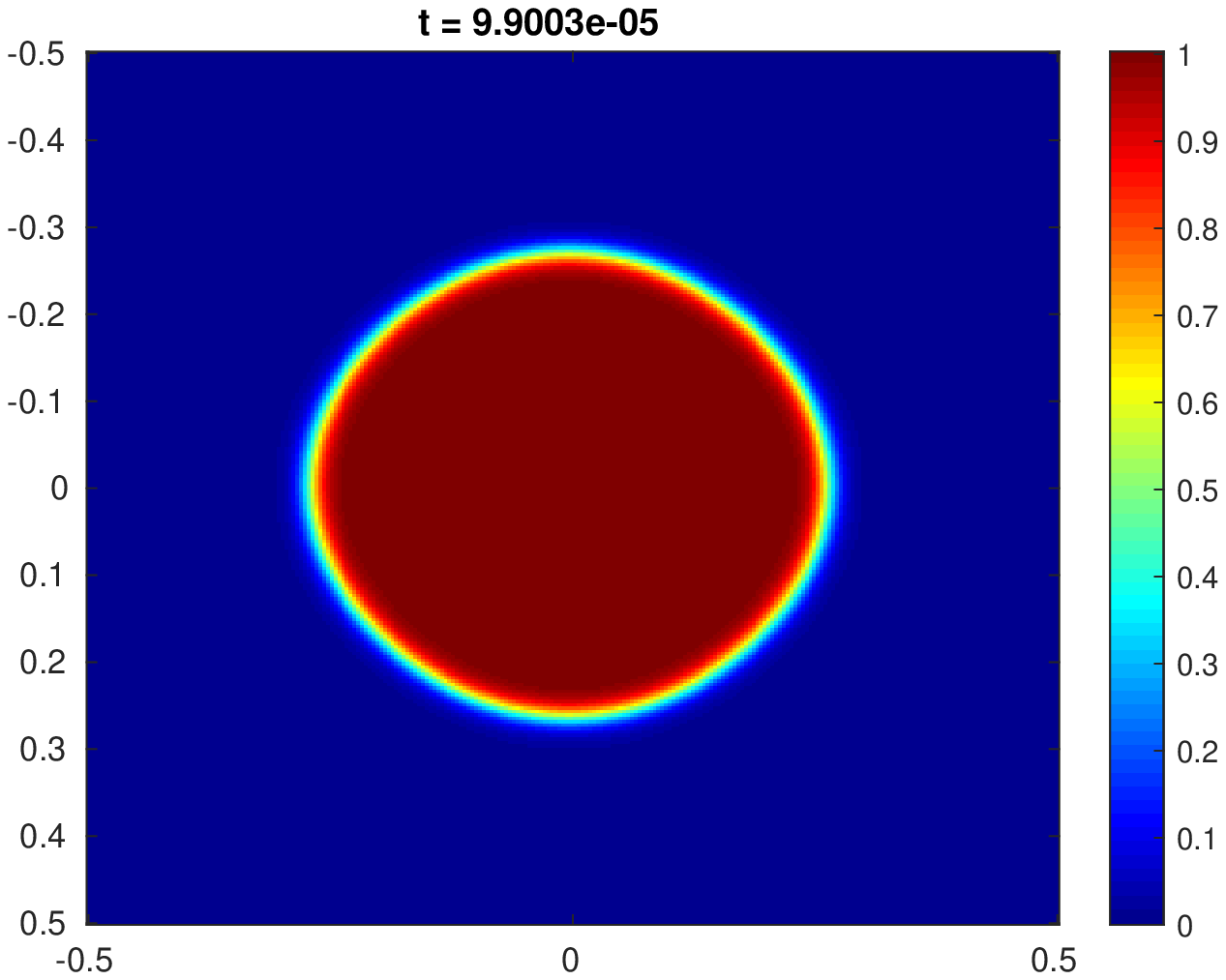} \\
	\includegraphics[width=3.5cm]{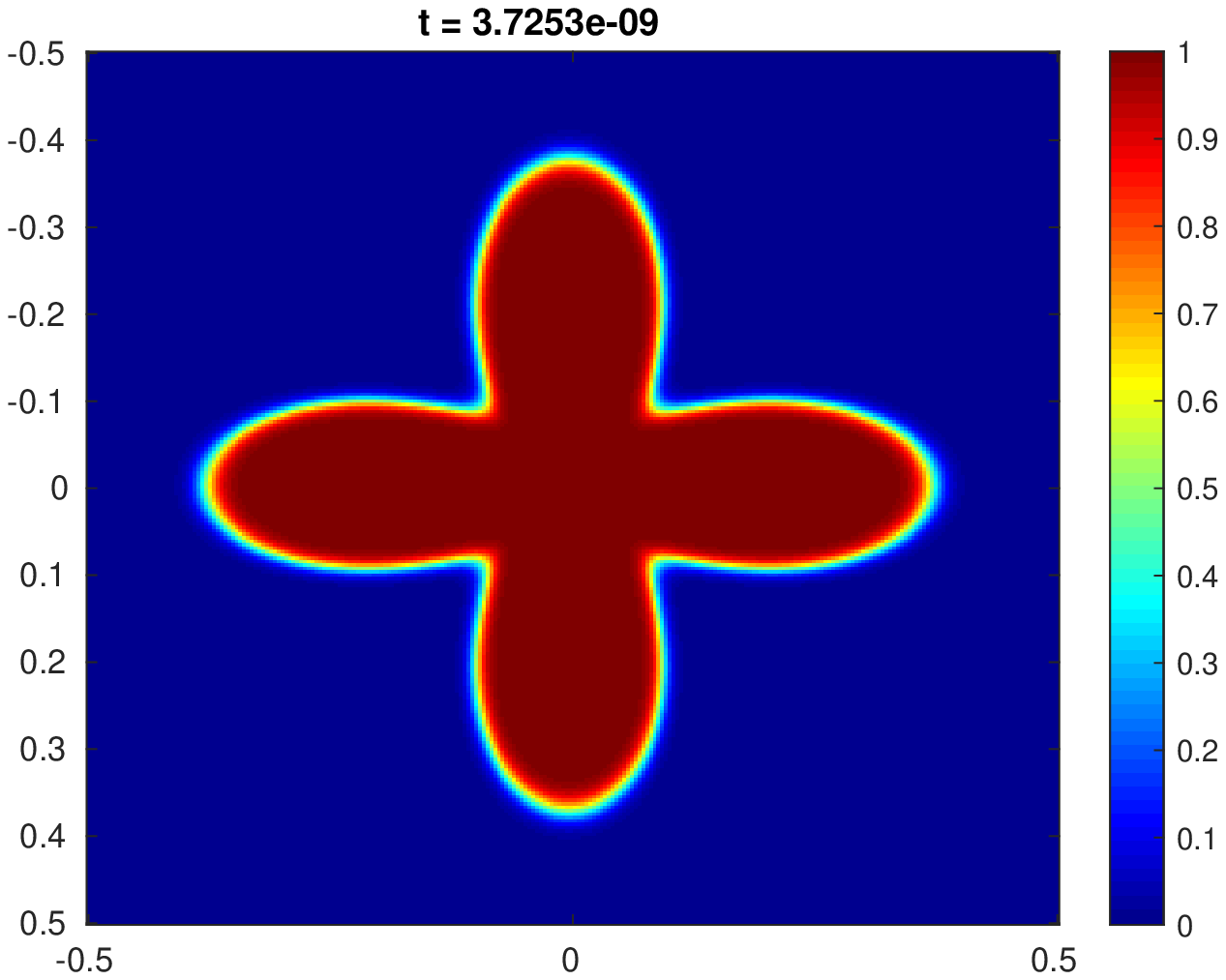}
	\includegraphics[width=3.5cm]{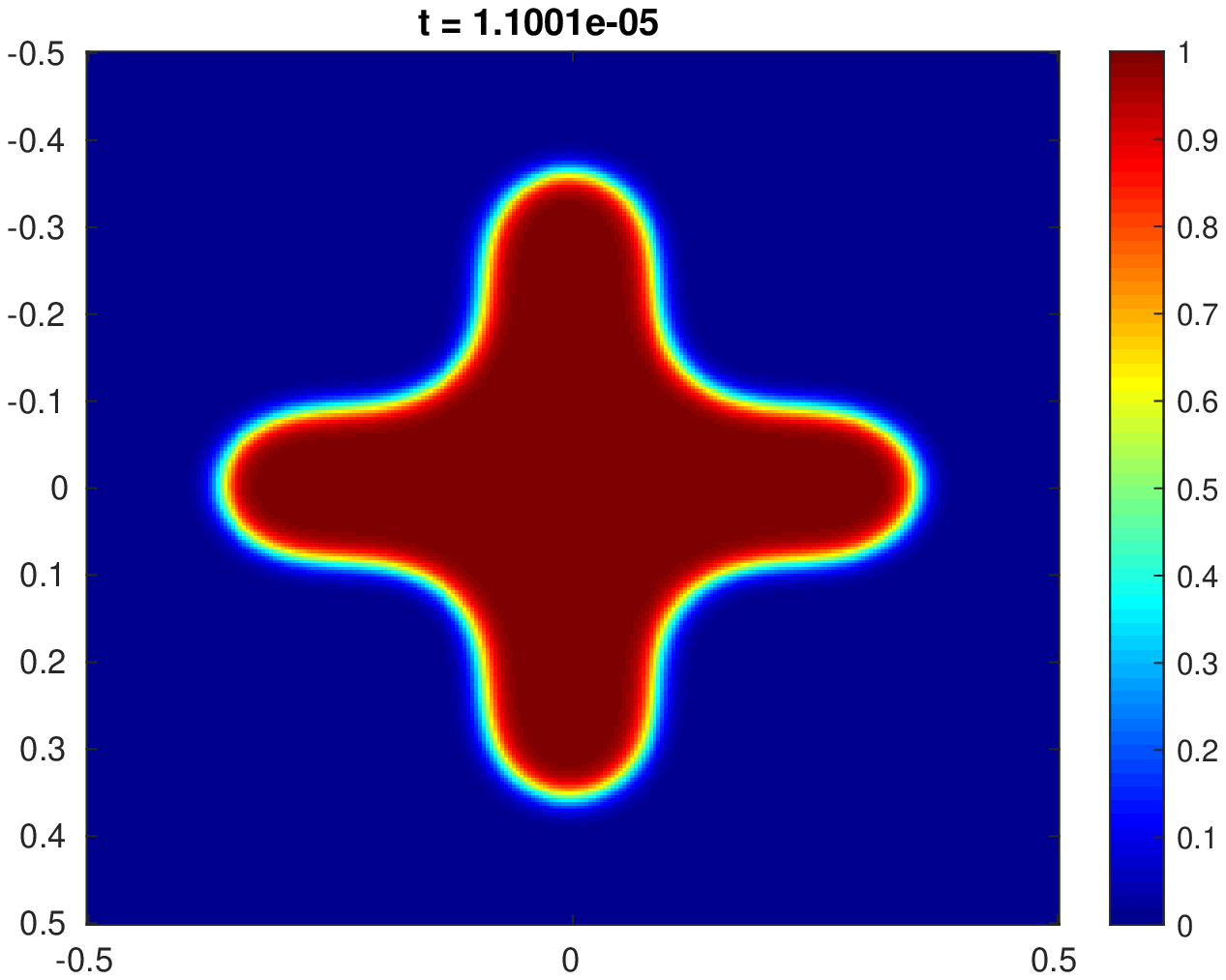}
	\includegraphics[width=3.5cm]{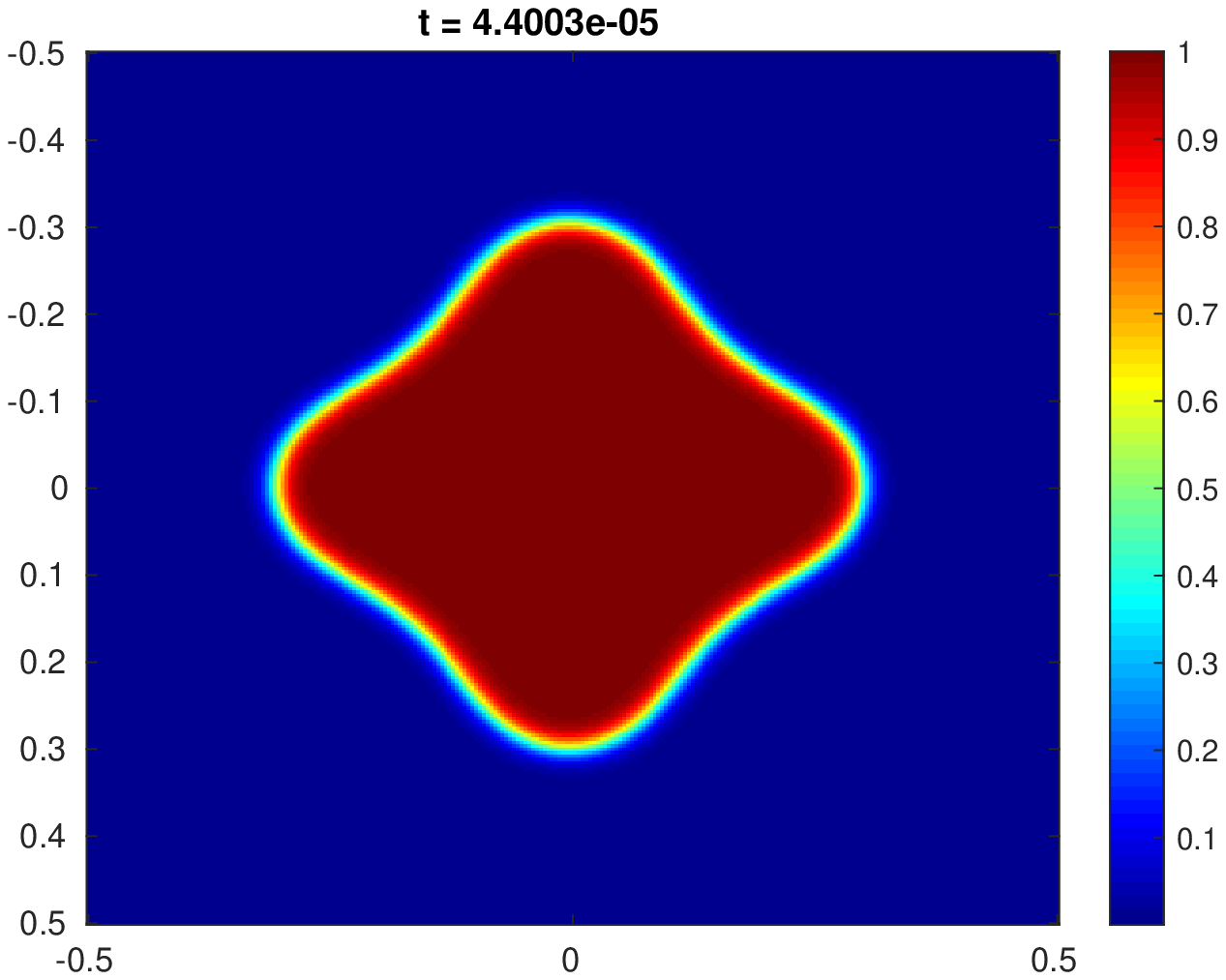}
	\includegraphics[width=3.5cm]{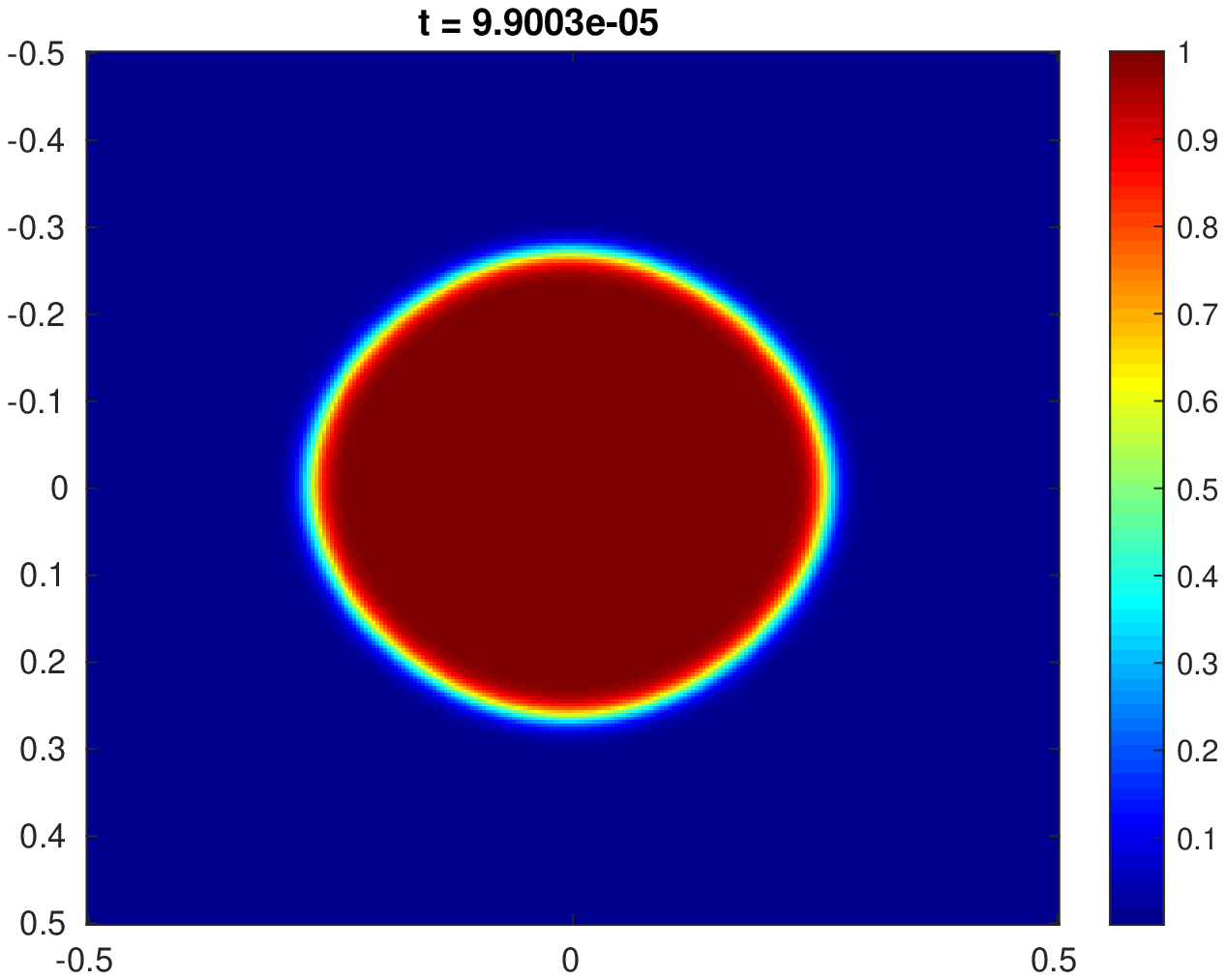}\\
\caption{Time evolution of the numerical phase field solutions of the \eqref{eq:MCH} model solved by \eqref{eq:SAV_MCH} (first row) and the \eqref{eq:NMN_CH}  model solved by \eqref{eq:SAV_NMNCH} (second row).}
\label{fig_test1}
\end{figure}

%

\begin{figure}[htbp]
\centering
		\includegraphics[height=4cm]{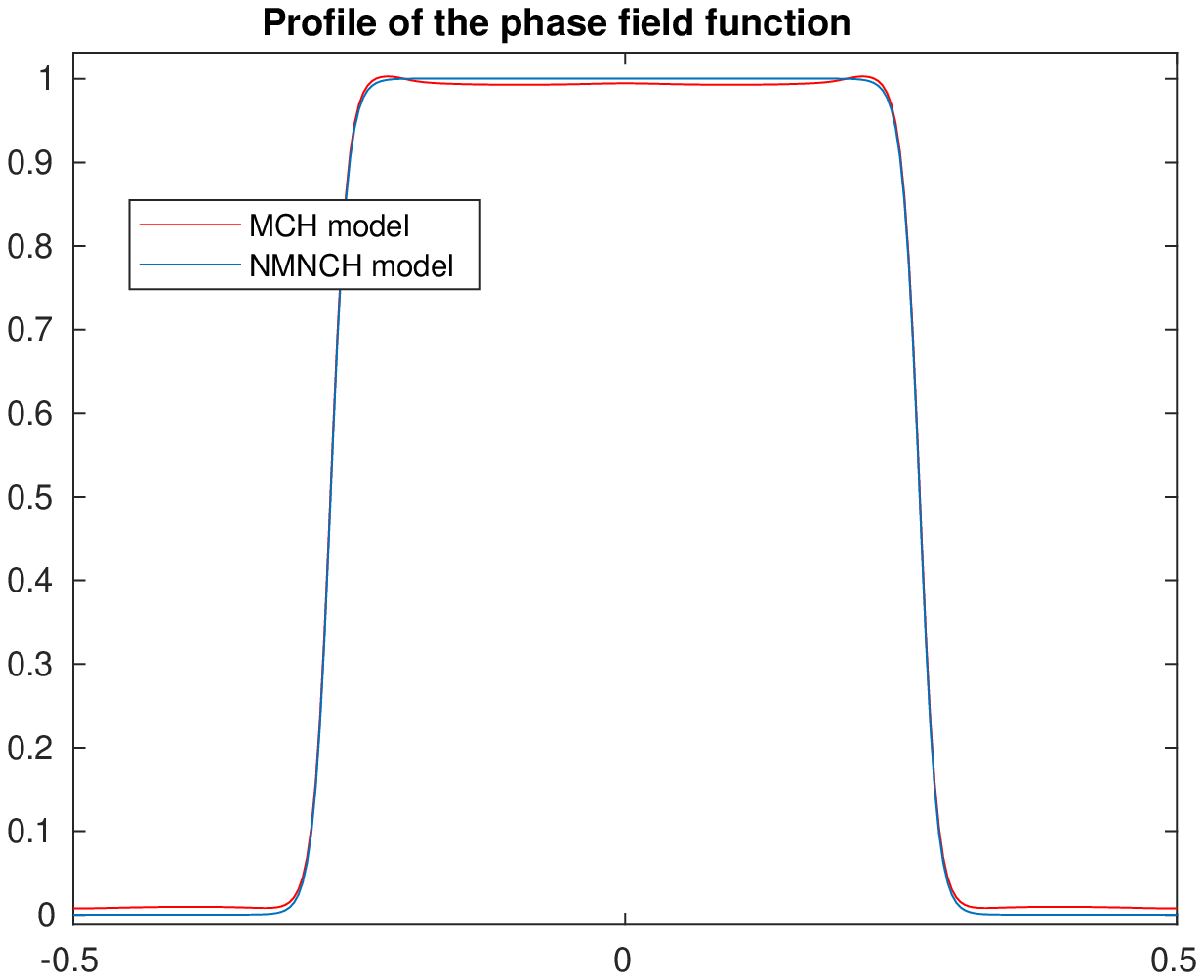}~
	\includegraphics[height=4cm]{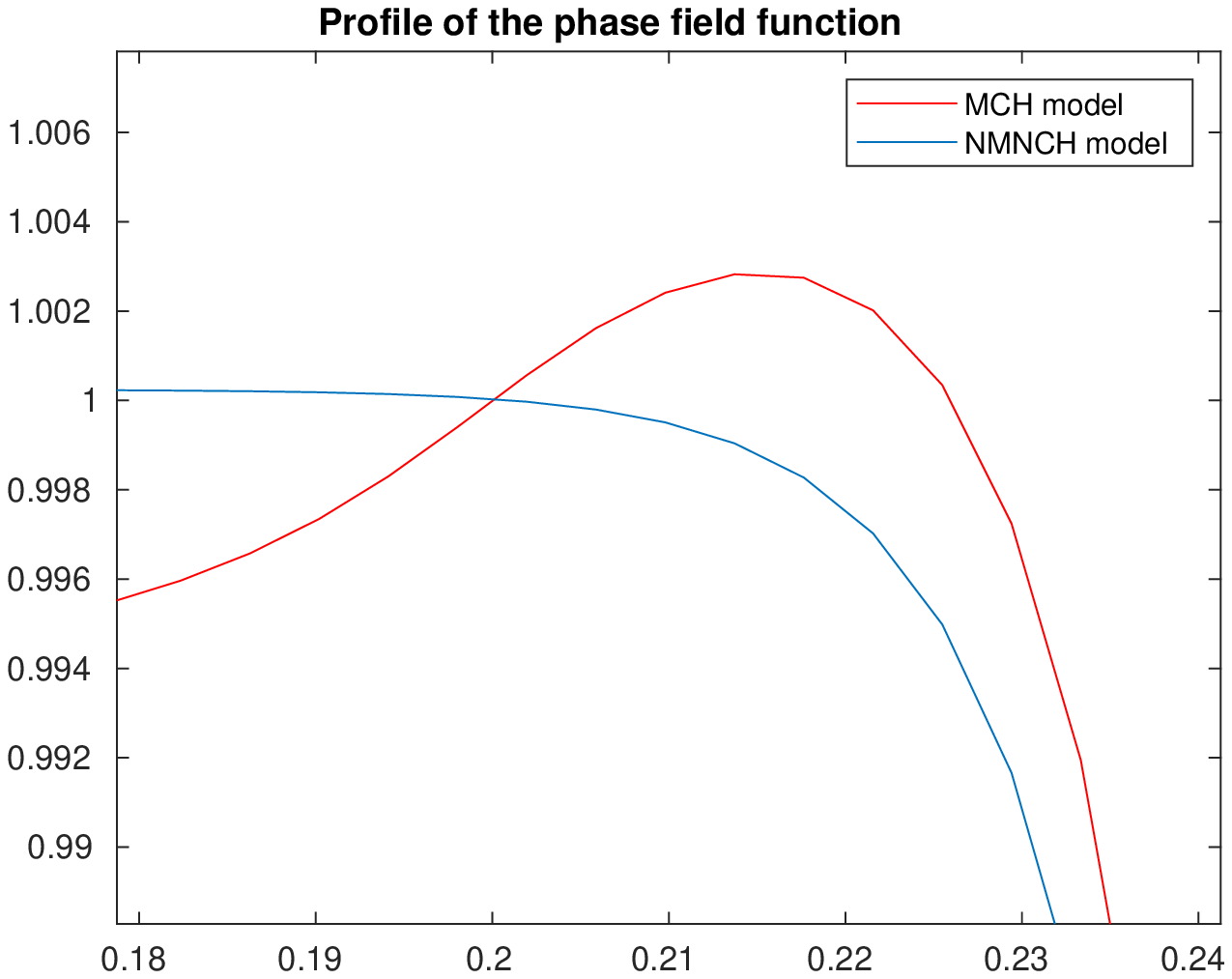}~
	\includegraphics[height=4cm]{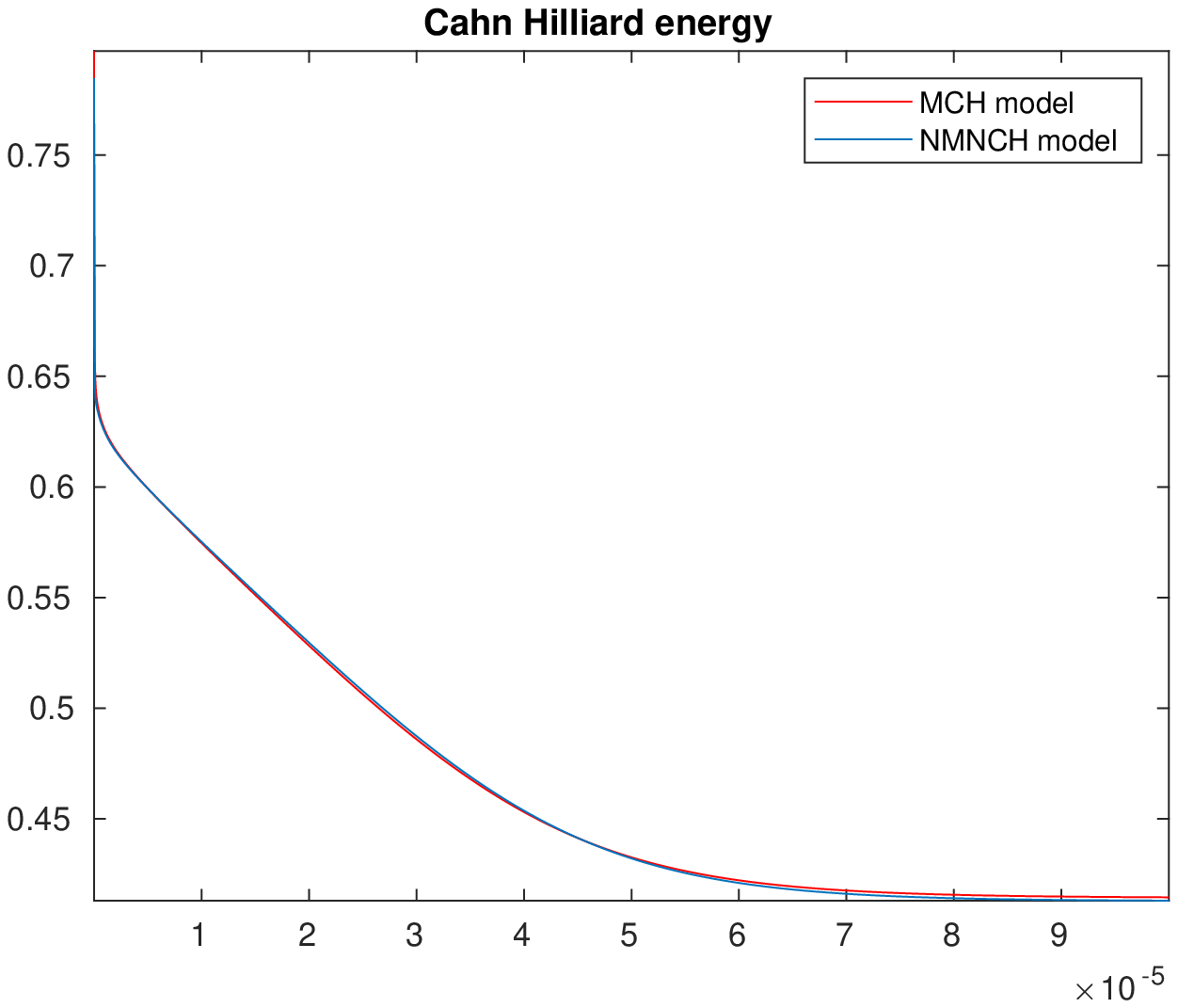}
\caption{Left: representation of the slice $x_1 \mapsto u^T(x_1,0)$ where $u^T$ is the numerical phase field solution at the final time to \eqref{eq:SAV_MCH} (in red) and \eqref{eq:SAV_NMNCH} (in blue). Middle:  Zoom of the right transition zone. Right: decrease of the Cahn-Hilliard energy for each model.}
\label{fig_test1_profile}
\end{figure}

\subsubsection{Evolution of a thin structure in dimension $3$}

Following \cite{bretin2020approximation}, we now validate the relevance of using the \eqref{eq:NMN_CH} model in comparison with \eqref{eq:MCH} for the evolution of thin structures. Indeed, the imprecise nature of the \eqref{eq:MCH} model results in volume losses of the order of $\varepsilon$, which are particularly significant when dealing with this type of problems. On the other hand, the \eqref{eq:NMN_CH} model suffers from these losses only at order  $O(\varepsilon^2)$, which seems to have a smaller impact on the surface diffusion flow approximation. A numerical validation of this is  presented in Figure \ref{fig_test3D_MCH} and  \ref{fig_test3D}, 
where the SAV schemes \eqref{eq:SAV_MCH} and \eqref{eq:SAV_NMNCH} are used instead of convex-concave splitting for  the treatment of the mobility. 
For the experiment showing the dewetting of a thin tube (Figures \ref{fig_test3D_MCH} and~\ref{fig_test3D}, top), we use  $\ell_1=1$, $\ell_2=\ell_3=0.25$, $N_1=2^8$, $N_2=N_3=2^6$.
For the experiment showing the dewetting of a thin plate (Figure~\ref{fig_test3D}, bottom), we use $\ell_1=\ell_2=1$, $\ell_3=0.25$, $N_1=N_2=2^8$, $N_3=2^6$. The other settings for both experiments are $\varepsilon = 2/N_1$, $\delta_t = \varepsilon^4$, $\alpha = 2/\varepsilon^2$, $m=1$, and $\beta= 2/\varepsilon^2$.  In terms of performance, the numerical approximation of the \eqref{eq:NMN_CH} is better with the SAV scheme \eqref{eq:SAV_NMNCH} in comparison with the convex-concave splitting \eqref{eq:Conv_concave1st} used in  \cite{bretin2020approximation}, in the sense that the latter shows volume losses, see in particular the central small ball which is preserved in Figure~\ref{fig_test3D} but has disappeared on Figure 8 in  \cite{bretin2020approximation}. This illustrates well the better accuracy of the SAV approach adapted to the \eqref{eq:NMN_CH} model.

\begin{figure}[htbp]
	\centering
	\includegraphics[width=3.5cm]{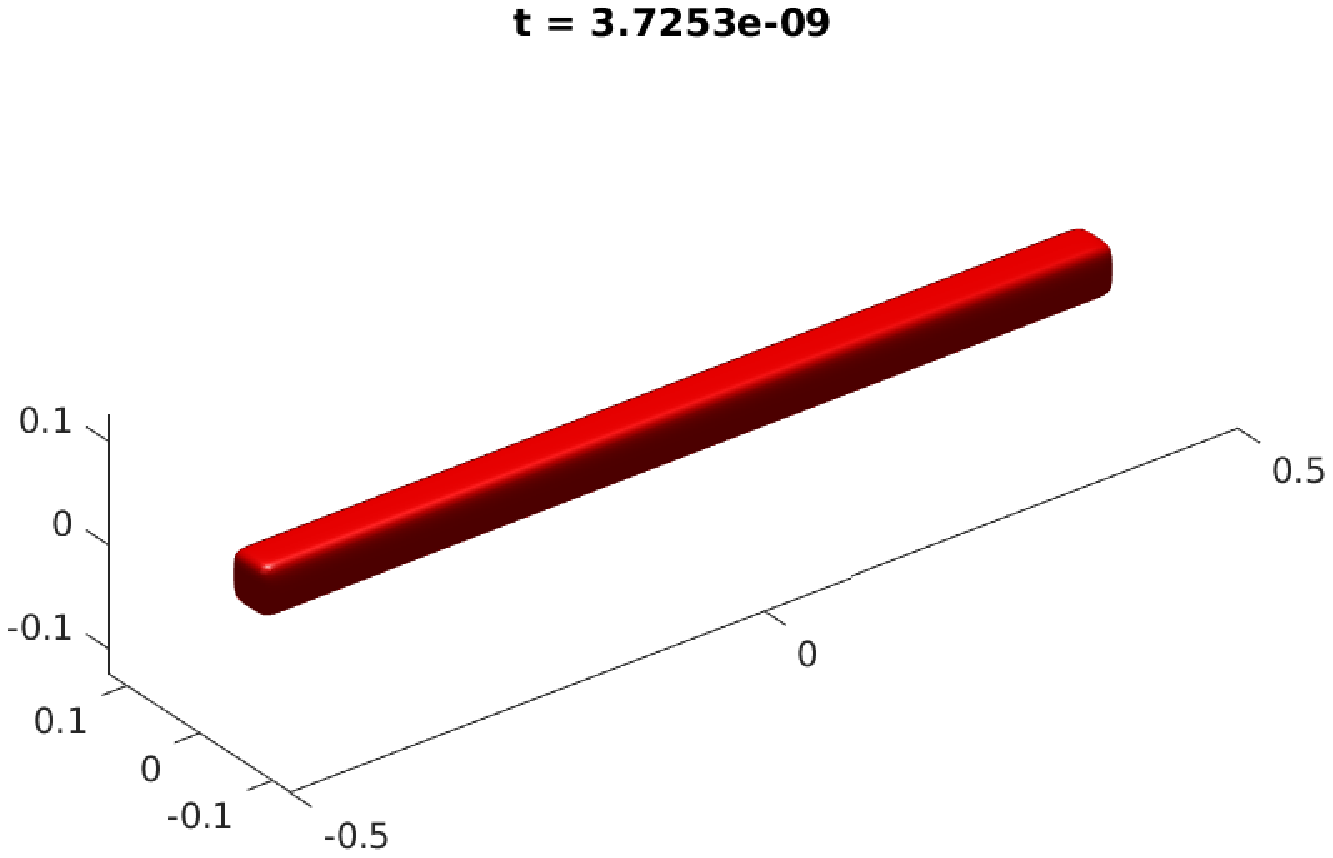}
	\includegraphics[width=3.5cm]{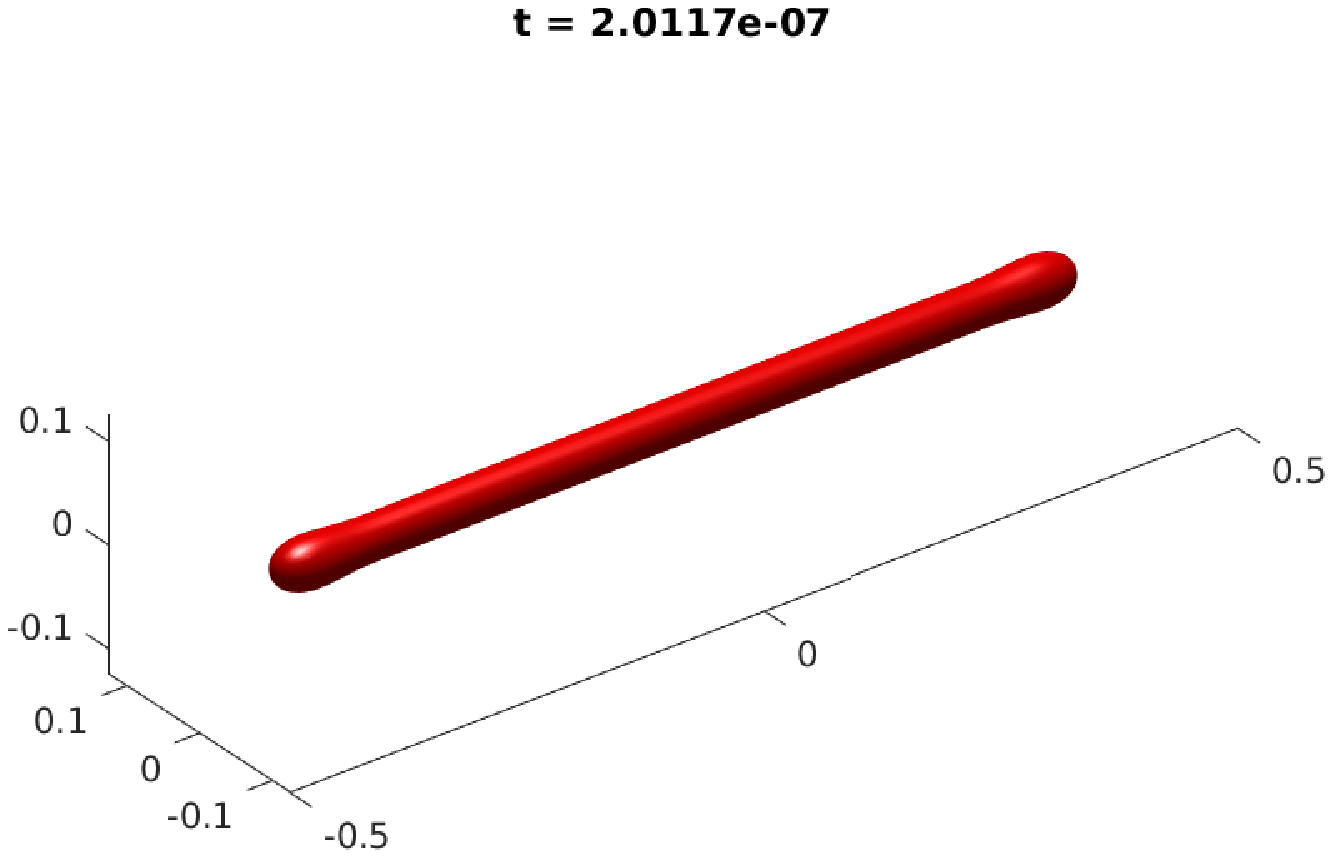}
	\includegraphics[width=3.5cm]{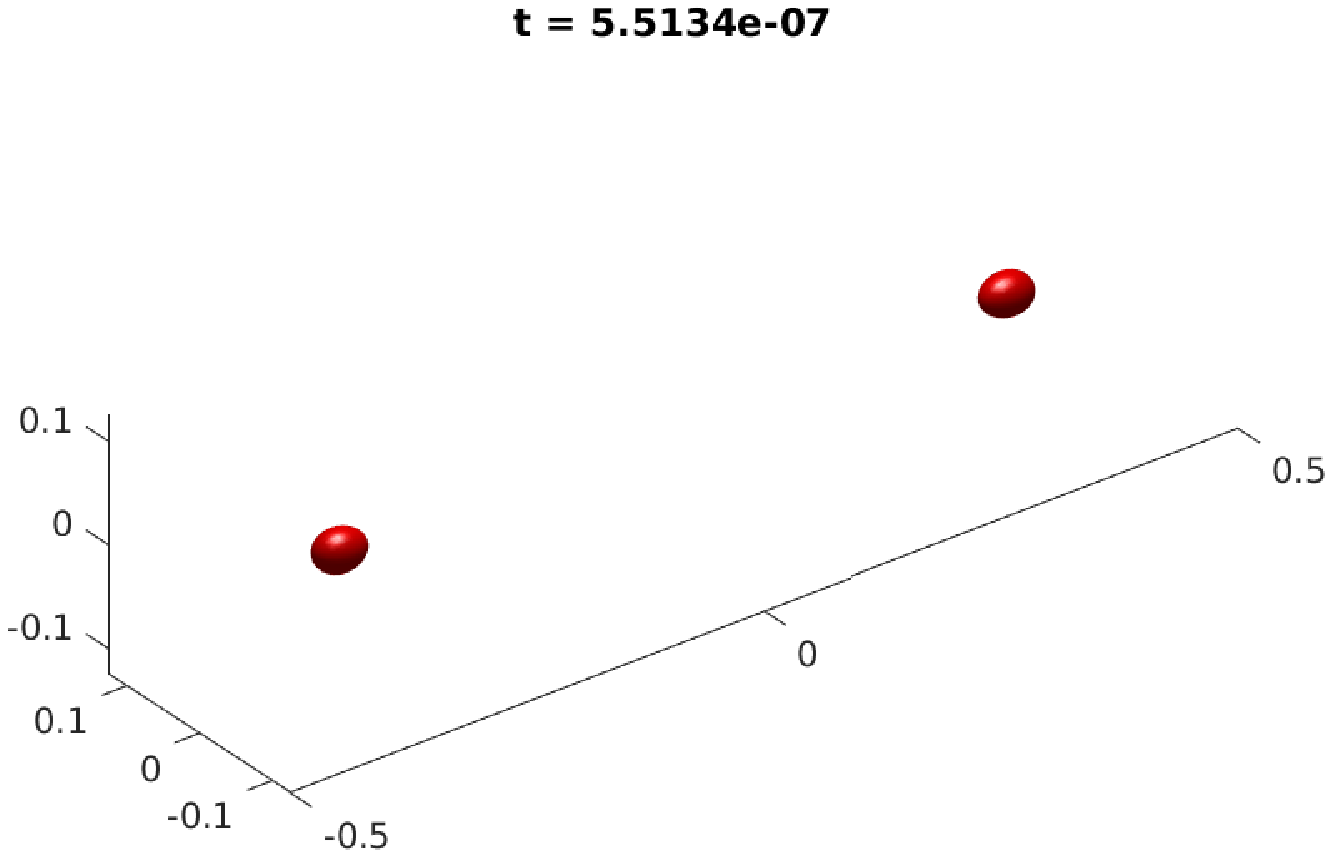}
	\includegraphics[width=3.5cm]{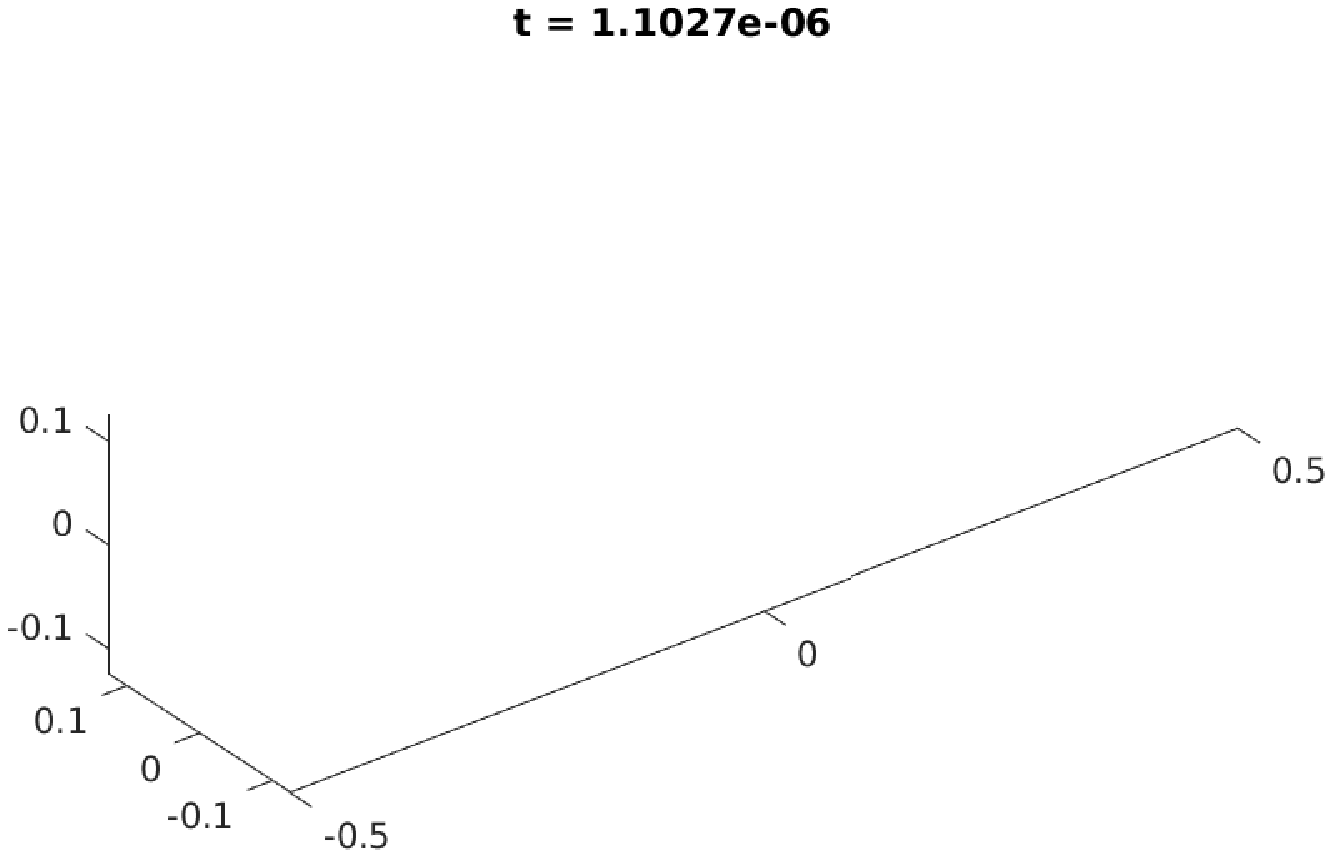} \\
	\caption{Example of dewetting in dimension $3$ using the \eqref{eq:MCH} model solved by \eqref{eq:SAV_MCH}. The figures show the evolution of $\{u_\varepsilon\leq 1/2\}$ over time.}
	\label{fig_test3D_MCH}
\end{figure}

\begin{figure}[htbp]
\centering
    \includegraphics[width=3.5cm]{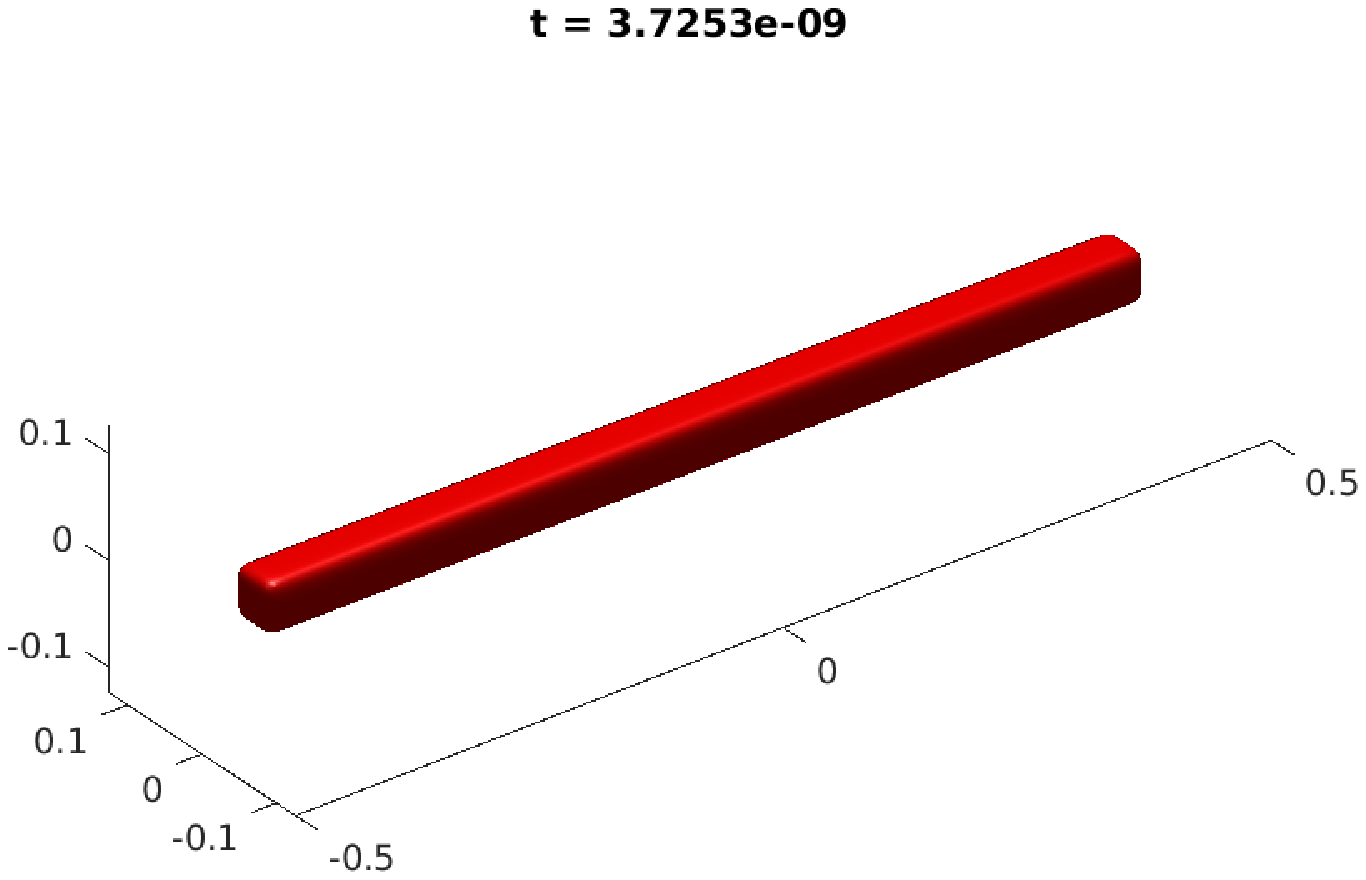}
	\includegraphics[width=3.5cm]{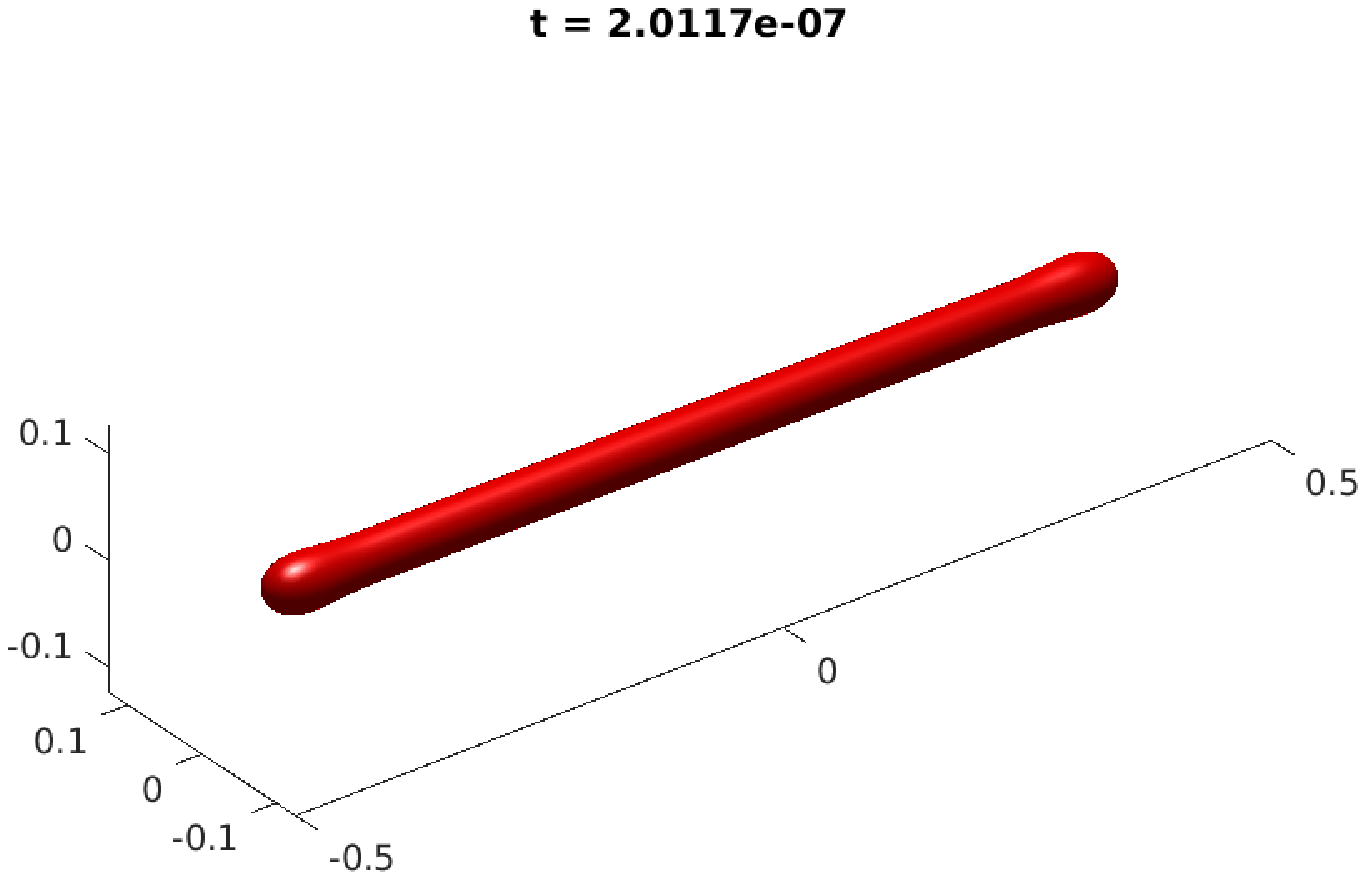}
	\includegraphics[width=3.5cm]{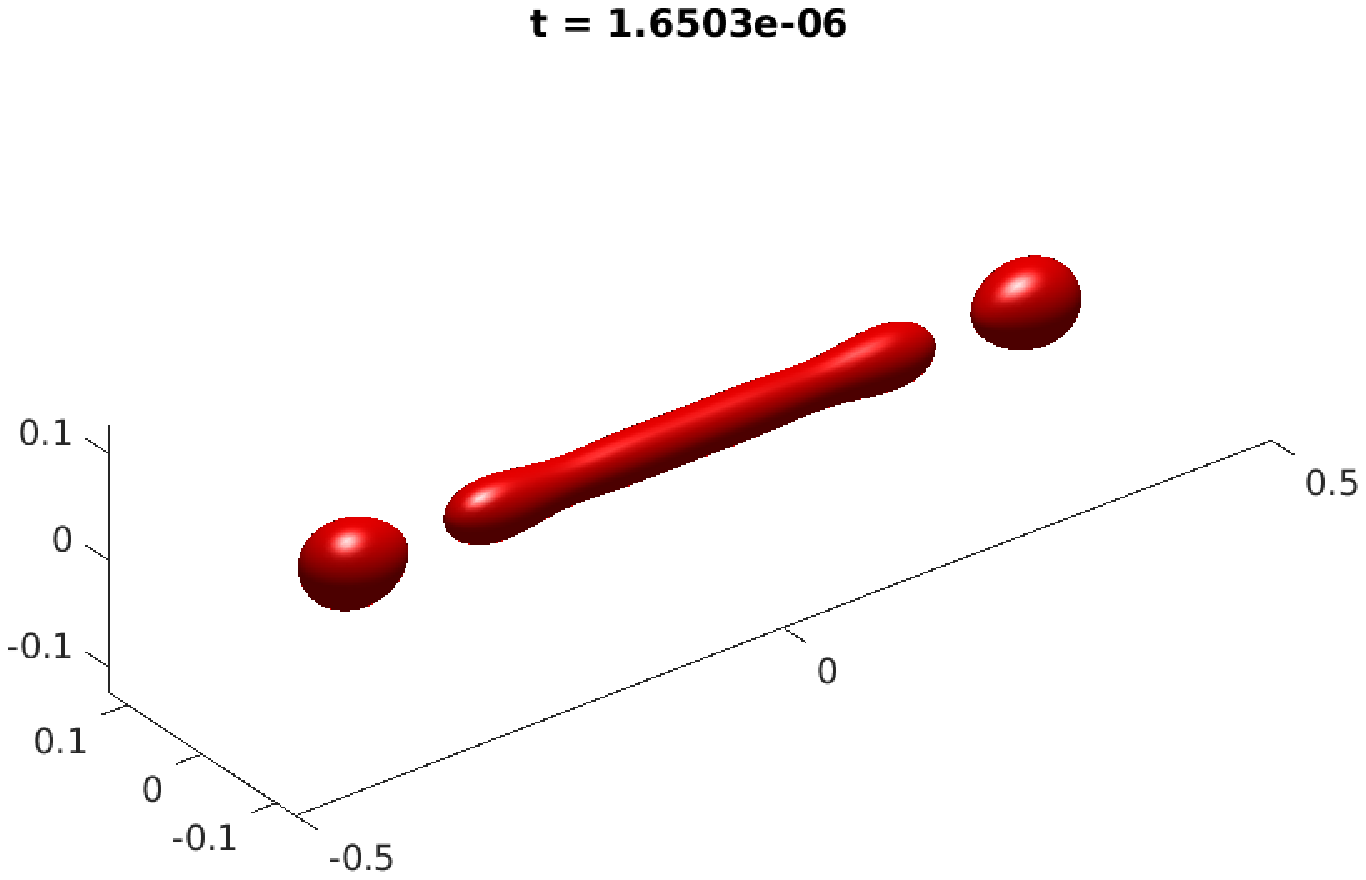}
	\includegraphics[width=3.5cm]{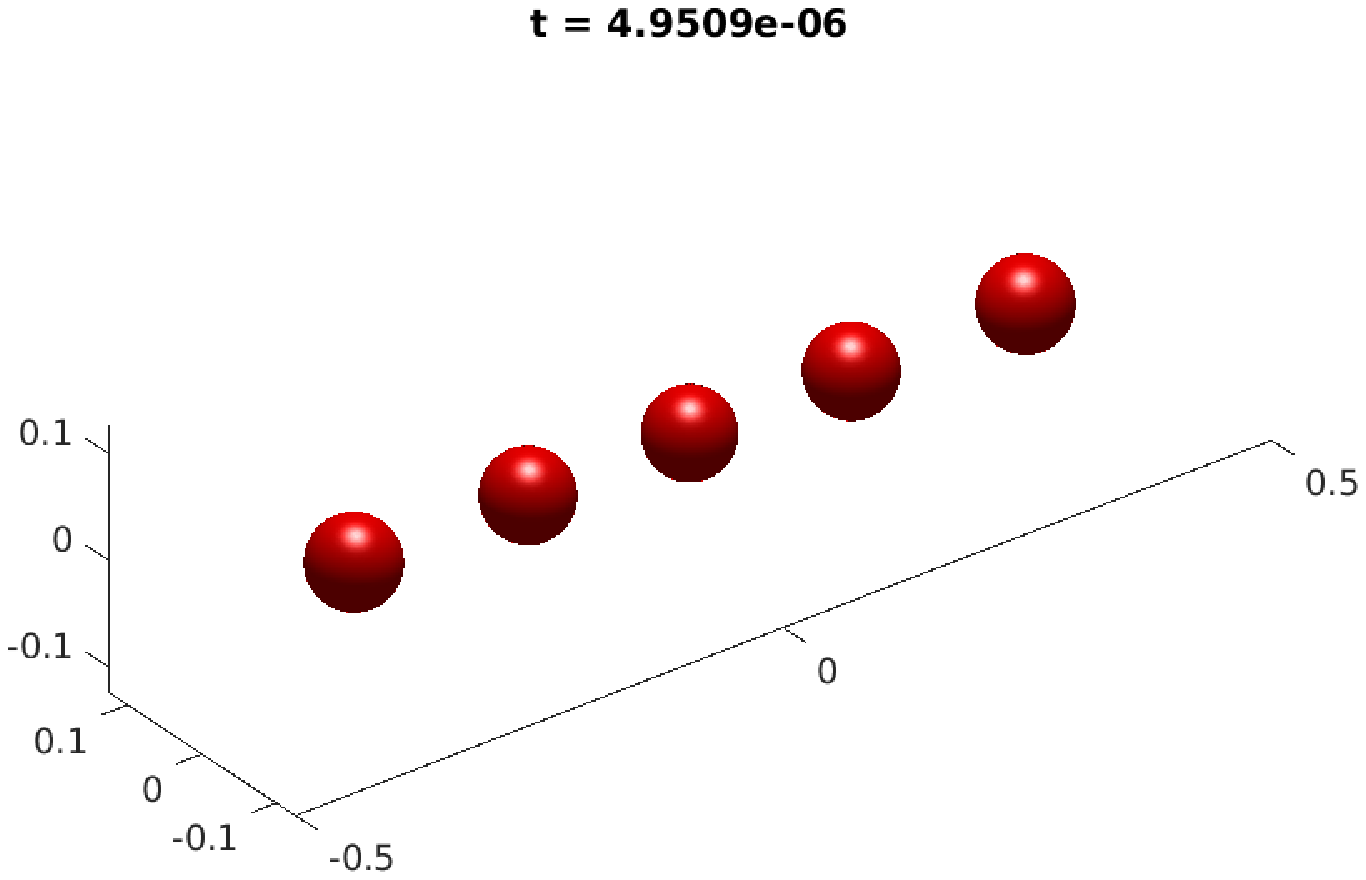} \\
	    \includegraphics[width=3.5cm]{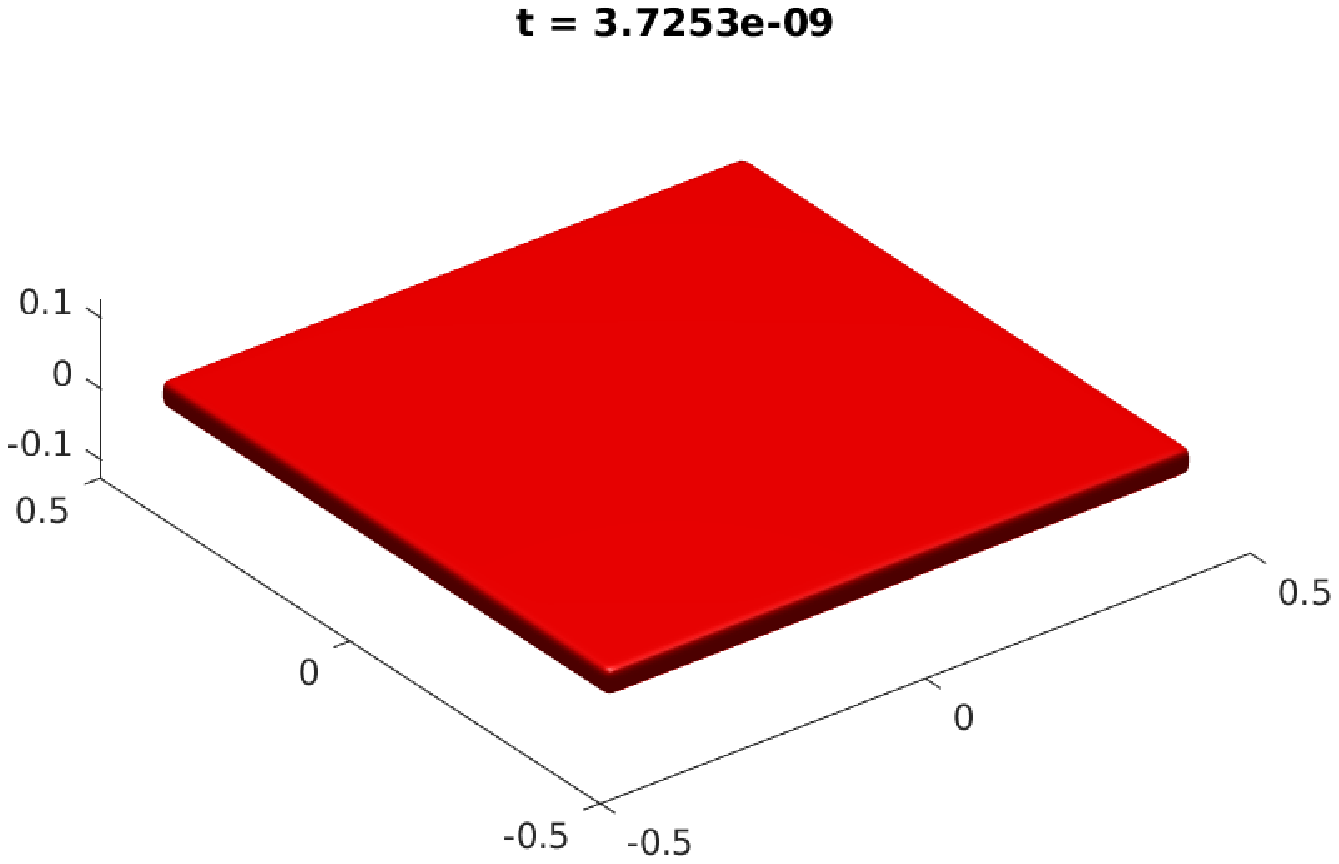}
	\includegraphics[width=3.5cm]{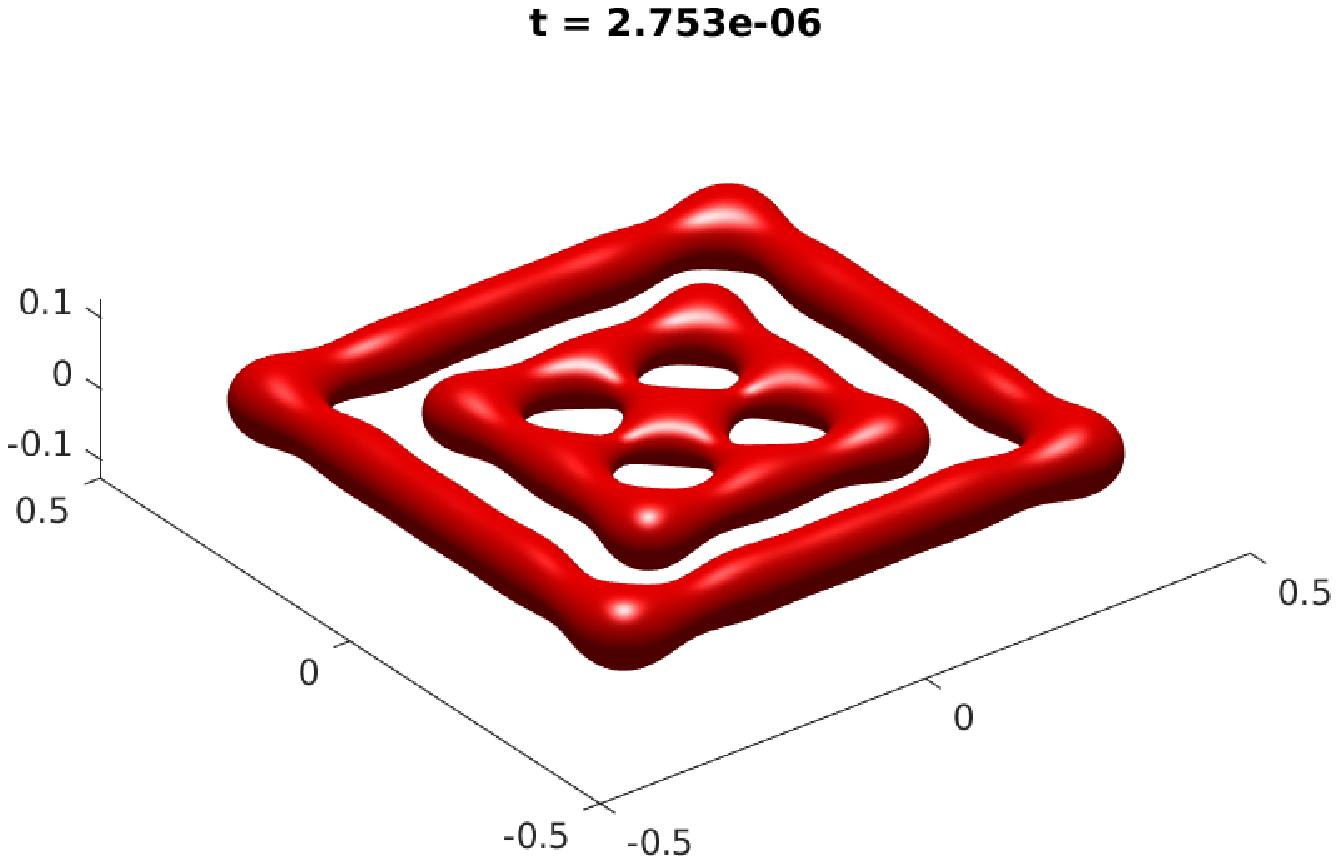}
	\includegraphics[width=3.5cm]{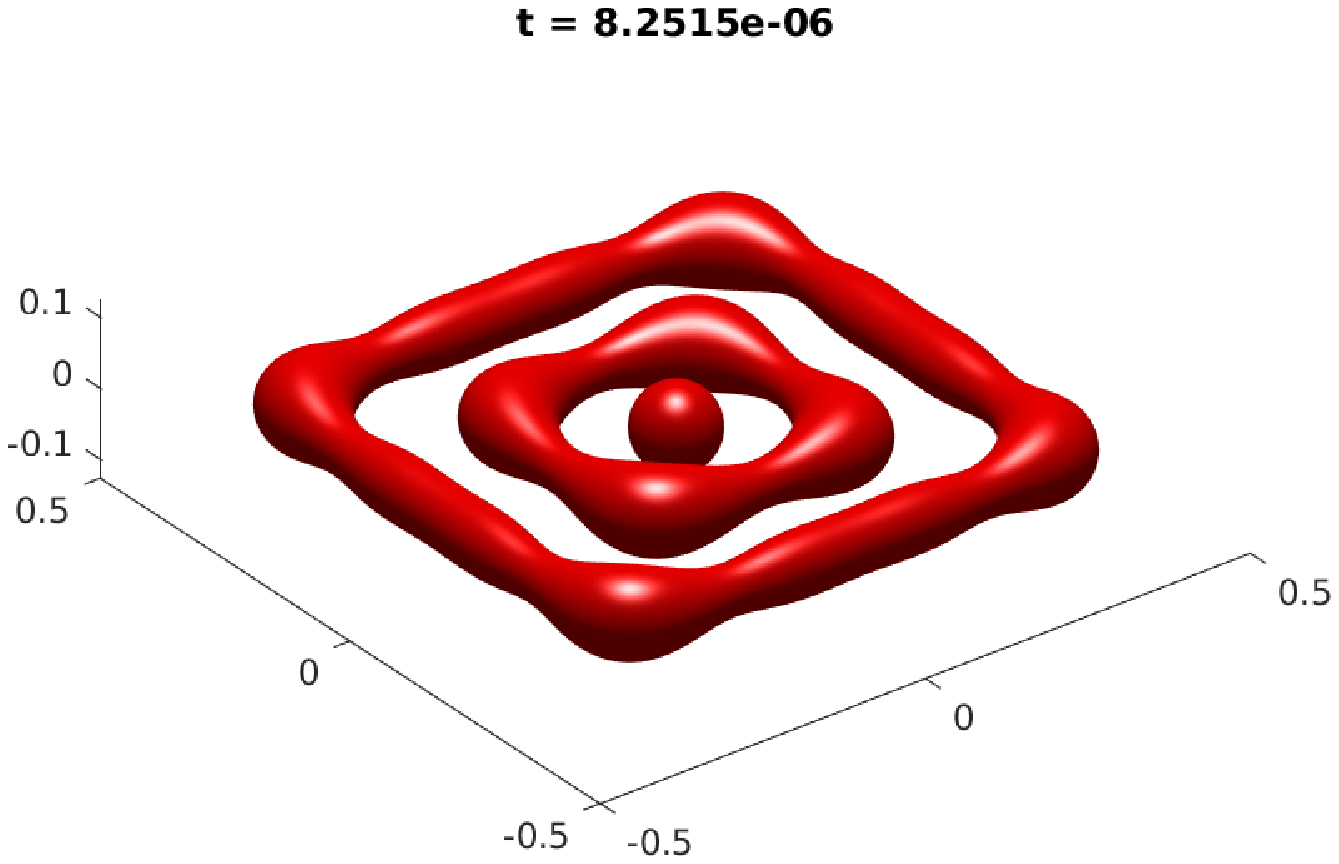}
	\includegraphics[width=3.5cm]{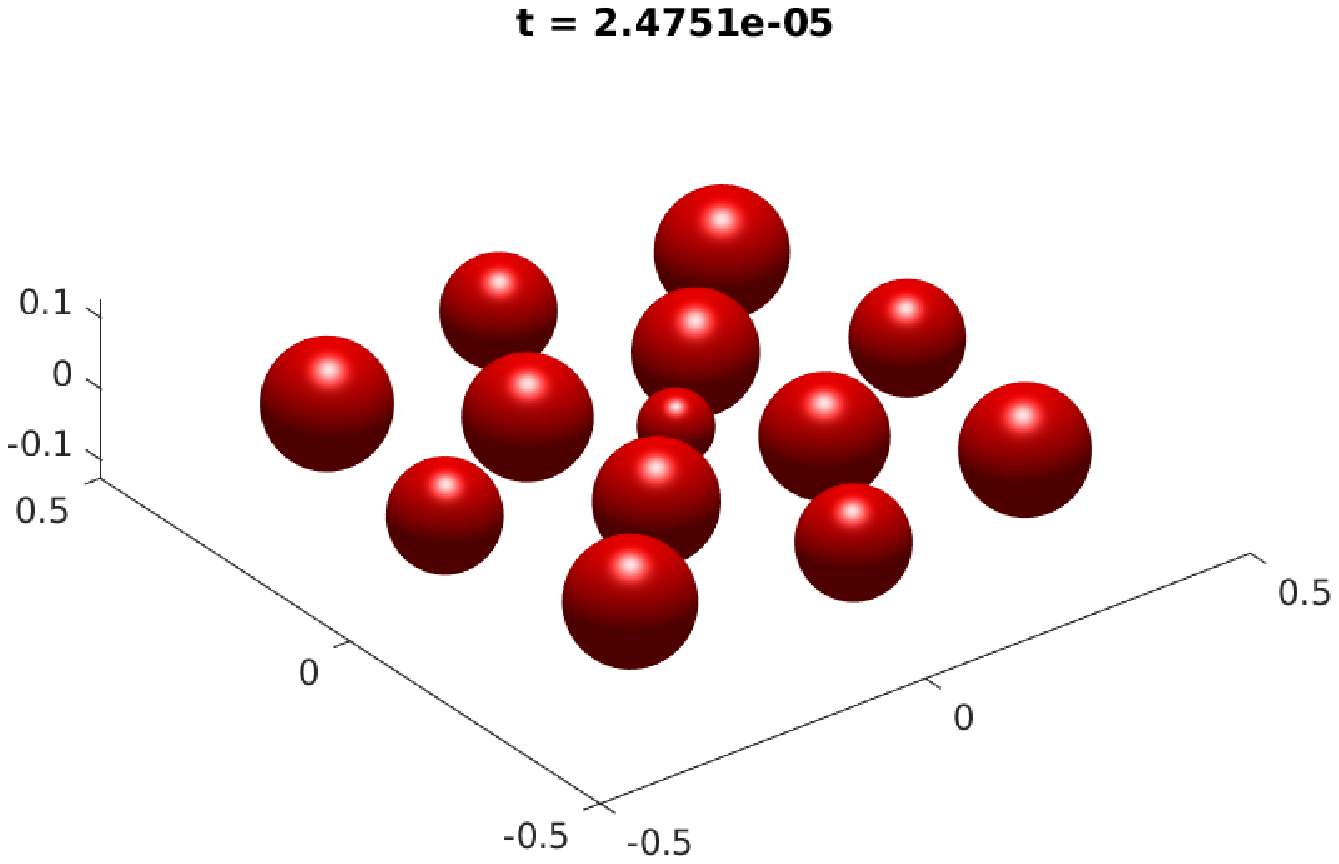} 
	\caption{Two examples of dewetting in dimension $3$ using the \eqref{eq:NMN_CH} model solved by \eqref{eq:SAV_NMNCH}. The figures show the evolution of $\{u_\varepsilon\leq 1/2\}$ over time.}
\label{fig_test3D}
\end{figure}

\section{Conclusion and perspectives}

We introduced, analysed and  validated numerically an SAV-type approach  providing first- and second-order numerical schemes the approximation of  solutions to fairly general gradient flows of convex functionals. 
We used this approach in the context of degenerate Cahn-Hilliard models by coupling the SAV relaxation with a convex-concave splitting of the associated energy. Theoretically, the resulting numerical scheme shows  unconditional stability and order-one accuracy. Numerically, this approach naturally leads to linear systems which can be efficiently solved in Fourier spaces and thus applies also to the approximation of surface diffusion flows in three dimensions.

Among the numerous perspectives to this work,  both in terms of applications and numerical modelling, let us mention for instance:
\begin{itemize}
	\item The design of a double SAV approach based on the splitting of both the energy and the mobility to obtain second-order schemes for general $J^{-1}$ gradient flows of non quadratic convex functionals;
	\item In the context of phase-field models, the adaptation of the method to deal with general (including crystalline) anisotropies or higher-order energies (such as the Willmore energy);
	\item The extension of the schemes to multiphase systems as in \cite{refId0}.
\end{itemize}

\section*{Acknowledgements}

EB and SM  acknowledge support from the French National Research Agency (ANR) under grants ANR-18-CE05-0017 (project BEEP) and ANR-19-CE01-0009-01 (project MIMESIS-3D).
LC acknowledges the support received by the ANR-22-CE48-0010 (project TASKABILE).  Part of this work was also supported by the LABEX MILYON (ANR-10-LABX-0070) of Universit\'e de Lyon, within the program "Investissements d'Avenir" (ANR-11-IDEX- 0007) operated by the French National Research Agency (ANR), and by the European Union Horizon 2020 research and innovation programme under the Marie Sklodowska-Curie grant agreement No 777826 (NoMADS).


\end{document}